\documentclass[11pt]{article}
\usepackage[letterpaper,margin=1.00in]{geometry}
\usepackage{amscd}
\usepackage{mathrsfs}
\usepackage[IL2]{fontenc}

\usepackage{comment} 

\usepackage{color}
\usepackage{amssymb,amsmath,amsthm,amsfonts,latexsym}
\usepackage[utf8x]{inputenc}
\usepackage{bbm} 

\usepackage{graphicx}
\usepackage[]{todonotes}

\usepackage{amsmath,amsfonts,amssymb,amsthm} 
\usepackage{tikz}
\usepackage{hyperref}
\usepackage{thmtools}
\usepackage{thm-restate}

\newtheorem{theorem}{Theorem}[section]

\newtheorem*{claim*}{Claim}
\newtheorem*{theorem*}{Theorem}
\newtheorem*{definition*}{Definition}

\newtheorem{corollary}[theorem]{Corollary}
\newtheorem{lemma}[theorem]{Lemma}
\newtheorem{remark}[theorem]{Remark}
\newtheorem{claim}[theorem]{Claim}
\newtheorem{fact}[theorem]{Fact}

\newtheorem{proposition}[theorem]{Proposition}

\newtheorem{definition}[theorem]{Definition}

\newcommand{\comm}[1]{}
\usepackage[capitalize, nameinlink]{cleveref}
\crefname{theorem}{Theorem}{Theorems}
\crefname{proposition}{Proposition}{Propositions}
\crefname{observation}{Observation}{Observations}
\crefname{lemma}{Lemma}{Lemmas}
\crefname{claim}{Claim}{Claims}
\crefname{problem}{Problem}{Problems}
\crefname{conjecture}{Conjecture}{Conjectures}
\crefname{question}{Question}{Questions}
\crefname{example}{Example}{Examples}
\crefname{fact}{Fact}{Facts}

\newcommand{\dom}{\operatorname{dom}}

\newcommand{\FINdep}{\mathsf{FINDEP}}

\newcommand{\local}{\mathsf{LOCAL}}

\newcommand{\olocal}{\mathsf{ULOCAL}}
\newcommand{\ulocal}{\mathsf{ULOCAL}}
\newcommand{\lca}{\mathsf{LCA}}
\newcommand{\LOCAL}{\mathsf{LOCAL}}
\newcommand{\RLOCAL}{\mathsf{R-LOCAL}}
\newcommand{\CLOCAL}{\mathsf{C-LOCAL}}

\newcommand{\slocal}{\mathsf{SLOCAL}}

\newcommand{\rlocal}{\mathsf{RLOCAL}}

\newcommand{\CONT}{\mathsf{CONTINOUS}}
\newcommand{\cont}{\mathsf{CONTINOUS}}

\newcommand{\LTOAST}{\mathsf{LAZY\;TOAST}}
\newcommand{\RLTOAST}{\mathsf{LAZY\;RTOAST}}

\newcommand{\RTOAST}{\mathsf{RTOAST}}
\newcommand{\TOAST}{\mathsf{TOAST}}
\newcommand{\ltoast}{\LTOAST}
\newcommand{\rltoast}{\RLTOAST}
\newcommand{\lrtoast}{\RLTOAST}
\newcommand{\rtoast}{\RTOAST}

\newcommand{\toast}{\TOAST}

\newcommand{\tow}{\mathrm{tower}}
\newcommand{\BOREL}{\mathsf{BOREL}}
\newcommand{\BAIRE}{\mathsf{BAIRE}}
\newcommand{\MEASURE}{\mathsf{MEASURE}}
\newcommand{\FIID}{\mathsf{fiid}}

\newcommand{\ffiid}{\mathsf{ffiid}}
\newcommand{\borel}{\mathsf{BOREL}}
\newcommand{\TOWER}{\mathsf{TAIL}(\tow(\Omega(r)))}

\newcommand{\baire}{\mathsf{BAIRE}}

\newcommand{\tail}{\mathsf{TAIL}}

\newcommand{\Cay}{\operatorname{Cay}}

\newcommand{\LG}{\operatorname{\bf FG}^{S}}
\newcommand{\RLG}{\operatorname{\bf FG}_{\bullet}^{S}}
\renewcommand{\P}{\mathrm{P}}
\newcommand{\E}{\textrm{\textbf{E}}}

\newcommand{\poly}{\operatorname{poly}}

\newcommand{\eps}{\varepsilon}

\newcommand{\fA}{\mathcal{A}}
\newcommand{\fB}{\mathcal{B}}
\newcommand{\fC}{\mathcal{C}}
\newcommand{\fD}{\mathcal{D}}
\newcommand{\fE}{\mathcal{E}}
\newcommand{\Epsilon}{\fE}

\newcommand{\fG}{\mathcal{G}}

\newcommand{\fP}{\mathcal{P}}

\newcommand{\R}{\mathbb{R}}
\newcommand{\Z}{\mathbb{Z}}

\def\leukfrac#1/#2{\leavevmode
               \kern.1em
                \raise.9ex\hbox{\the\scriptfont0 ${}_#1$}
                \hskip -1pt\kern-.1em
                /\kern-.15em\lower.10ex\hbox{\the\scriptfont0 ${}_#2$}}

\makeatletter
\def\diam{\mathop{\operator@font diam}\nolimits}
\makeatother

\usepackage{pdflscape}

\begin{document}
\title{Local Problems on Grids from the Perspective of Distributed Algorithms, Finitary Factors, and Descriptive Combinatorics
}

\author{Jan Grebík \\ \small{University of Warwick} \\ \small{jan.grebik@warwick.ac.uk} \and Václav Rozhoň\\ \small{ETH Zurich} \\ \small{rozhonv@ethz.ch}
}


\date{}
\maketitle
\renewcommand{\thefootnote}{\fnsymbol{footnote}} 
\renewcommand{\thefootnote}{\arabic{footnote}} 


\begin{abstract}

    We present an intimate connection among the following fields:
    \begin{enumerate}
    \item [(a)] \emph{distributed local algorithms}: coming from the area of computer science,
    \item [(b)] \emph{finitary factors of iid processes}: coming from the area of analysis of randomized processes,
    \item [(c)] \emph{descriptive combinatorics}: coming from the area of combinatorics and measure theory. 
    \end{enumerate}    
    In particular, we study locally checkable problems on grids from all three perspectives. 
    Most of our results are for perspective (b) where we prove time hierarchy theorems similar to those known in the field (a) [Chang, Pettie FOCS 2017].   
    This approach that borrows techniques from the fields (a) and (c) implies a number of results about possible complexities of finitary factor solutions. 
    Among others, it answers three open questions of [Holroyd, Schramm, Wilson Annals of Prob. 2017] or the more general question of [Brandt et al. PODC 2017] who asked for a formal connection between the fields (a) and (b). 

    In general, we hope that our treatment will help to view all three perspectives as a part of a common theory of locality, in which we follow the insightful paper of [Bernshteyn 2020+]. 
\end{abstract}

\section{Introduction}

In this paper, we study local problems on $d$-dimensional grid graphs, for $d>1$, from three different perspectives: (a) the perspective of the theory of distributed algorithms, (b) the perspective of the theory of random processes, (c) the perspective of the field of descriptive set theory. 
Although different communities in these three fields ask very similar questions, no systematic connections between them were known, until the insightful work of Bernshteyn \cite{Bernshteyn2021LLL} (see also \cite{elek}) who applied results from the field (a) to the field (c). In this work, we get most specific results by applying techniques from the field (a) to (b), but perhaps more importantly, we try to present all three different perspectives as a part of one common theory (see \cref{fig:big_picture_grids}). 

In this introductory section, we first quickly present the three fields and define the object of interest, local problems, a class of problems that contains basic problems like vertex or edge coloring, matchings, and others. 
Then, we briefly present our results, most of which can be seen as inclusions between different classes of local problems shown in \cref{fig:big_picture_grids}.
We note that such connections are also studied in the context of paths and trees in two parallel papers \cite{grebik_rozhon2021LCL_on_paths,brandt_chang_grebik_grunau_rozhon_vidnyaszky2021LCLs_on_trees_descriptive}.

\paragraph{Distributed Computing}

The $\local$ model of computing \cite{linial92LOCAL} is motivated by understanding distributed algorithms in huge networks. 
As an example, consider the network of all wifi routers, where two routers are connected if they are close enough to exchange messages. In this case, each router should choose a different channel from its neighbors to communicate with user devices to avoid interference. In the graph-theoretic language, we want to color the network. Even if the maximum degree is $\Delta$ and we want to color it with $\Delta+1$ colors, the problem remains nontrivial, because the decision of each node should be done after a few rounds of communication with its neighbors. 

The $\local$ model formalizes this setup: we have a large network, with each node knowing its size, $n$, and perhaps some other parameter like the maximum degree $\Delta$. In the case of randomized algorithms, each node has an access to a random string, while in the case of deterministic algorithms, each node starts with a unique identifier from a range of size polynomial in $n$. 
In one step, each node can exchange any message with its neighbors and can perform arbitrary computations. We want to find a solution to a problem in as few rounds of communication as possible. 
Importantly, there is an equivalent view of $t$-steps $\local$ algorithms: such an algorithm is simply a function that maps $t$-hop neighborhoods to the final output. An algorithm is correct if and only if applying this function to each node solves the problem. 

There is a rich theory of distributed algorithms and the local complexity of many problems is understood. One example: if the input graph is the $d$-dimensional torus (think of the grid $\mathbb{Z}^d$ and identify nodes with the same remainder after dividing the coordinate with $n^{1/d}$), a simple picture emerges. Any local problem (we define those formally later) is solvable with local complexity $O(1)$, $\Theta(\log^* n)$, or we need to see the whole graph to solve it, i.e., the local complexity is $\Theta(n^{1/d})$ (see \cref{thm:classification_of_local_problems_on_grids}). For example, a proper coloring of nodes with $4$ colors has local complexity $\Theta(\log^* n)$, while coloring with $3$ colors has already complexity $\Theta(n^{1/d})$.

\paragraph{Finitary factors of iid processes}

Consider the following example of the so-called anti-ferromagnetic Ising model. 
Imagine atoms arranged in a grid $\mathbb{Z}^3$. Each atom can be in one of two states and there is a preference between neighboring atoms to be in a different state. We can formalize this by defining a potential function $V(x,y)$ indicating whether $x$ and $y$ are equal.  
The Ising model models the behavior of atoms in the grid as follows: for the potential $V$ and inverse temperature $\beta$, if we look at the atoms, we find them in a state $\Lambda$ with probability proportional to $\exp(-\beta \sum_{\{u, v\} \in E(\mathbb{Z}^d)} V(\Lambda(u), \Lambda(v)))$. If $\beta$ is close to zero, we sample essentially from the uniform distribution. If $\beta$ is infinite, we sample one of the two $2$-colorings of the grid. 

To understand better the behavior of phase transitions, one can ask the following question. For what values of $\beta$ can we think of the distribution $\mu$ over states from the Ising model as follows: there is some underlying random variable $X(u)$ at every node $u$, with these variables being independent and identically distributed. We can think of $\mu$ as being generated by a procedure, where each node explores the random variables in its neighborhood until it decides on its state, which happens with probability $1$. If this view of the underlying distribution is possible, we say it is a finitary factor of an iid process (ffiid). 
Intuitively, being an ffiid justifies the intuition that there is a correlation decay between different sites of the grid.

Although this definition was raised in a different context than distributed computing, it is actually almost identical to the definition of a distributed algorithm! The differences are two. First, the ffiid procedure does not know the value of $n$ since, in fact, it makes the most sense to consider it on infinite graphs (like the infinite grid $\mathbb{Z}^d$). Second, the natural measure of the complexity is different: it is the speed of the decay of the random variable corresponding to the time until the local algorithm finds the solution. 
Next, note that the algorithm has to return a correct solution with probability $1$, while the usual randomized local algorithm can err with a small probability. We call ffiids \emph{uniform local algorithms} to keep the language of distributed computing (see \cref{sec:preliminaries} for a precise definition of ffiid and more explanations). 
We define the class $\olocal(t(\eps))$ as the class of problems solvable by a uniform local algorithm if the probability of the needed radius being bigger than $t(\eps)$ is bounded by $\eps$ (see \cref{fig:big_picture_grids}). 

\paragraph{Descriptive Combinatorics}

A famous Banach-Tarski theorem states that one can split a three-dimensional unit volume ball into five pieces, translate and rotate each piece and end up with two unit volume balls. 
This theorem serves as an introductory warning in the introductory measure theory courses: not every subset of $\mathbb{R}^3$ can be measurable.
Similarly surprising was when in 1990 Laczkovich \cite{laczkovich} answered the famous Circle Squaring Problem of Tarski: A square of the unit area can be decomposed into finitely many pieces such that one can translate each of them to get a circle of unit area.
Unlike in Banach-Tarski paradox, this result was improved in recent years to make the pieces measurable in various senses \cite{OlegCircle, Circle, JordanCircle}.


This theorem and its subsequent strengthenings are the highlights of a field nowadays called descriptive combinatorics \cite{kechris_marks2016descriptive_comb_survey,pikhurko2021descriptive_comb_survey} that has close connections to distributed computing as was shown in an insightful paper by Bernshteyn \cite{Bernshteyn2021LLL}. 
Going back to the Circle Squaring Problem, the proofs of its versions first choose a clever set of constantly many translations $\tau_1, \dots, \tau_k$. Then, one constructs a bipartite graph between the vertices of the square and the circle and connects two vertices $x, y$ if there is a translation $\tau_i$ mapping the point $x$ in the square to the point $y$ in the circle. The problem can now be formulated as a perfect matching problem and this is how combinatorics enters the picture. In fact, versions of a matching or flow problems are then solved on the grid graph $\mathbb{Z}^d$ using a trick called \emph{toast construction} also used in the area of the study of ffiids \cite{HolroydSchrammWilson2017FinitaryColoring,ArankaFeriLaci,Spinka}. 

We discuss the descriptive combinatorics classes in more detail in the expository \cref{sec:OtherClasses} but we will not need them in the main body. Let us now give an intuitive definition of the class $\borel$. 
Due to a result of Chang and Pettie \cite{chang2016exp_separation}, we know that the maximal independent set problem is complete for the class of problems solvable in $O(\log^* n)$ steps in the sense that this problem can be solved in $O(\log^* n)$ steps and any problem in that class can be solved in a constant number of steps if we are given access to an oracle computing MIS (we make this precise in \cref{rem:chang_pettie_decomposition}).
Imagine now that the input graph is infinite and our local algorithm has access to a MIS oracle. Constructing MIS in the Borel setting is due to Kechris, Solecki and Todor\v{c}evi\'{c} \cite{KST}, and local computations but now both can be used countably many times instead of just once as is the case with the class of problems solvable in $O(\log^* n)$ steps. 
That is, we can think of an algorithm that constructs MIS of larger and larger powers of the input graph and uses this infinite sequence to solve a given local problem. 
Those constructions are in the class $\borel$ from \cref{fig:big_picture_grids}. In fact, the exact relationship between this definition and the full power of $\BOREL$ is not clear.
However, it should give an idea about the difference between descriptive combinatorics and distributed computing -- from the descriptive combinatorics perspective, reasonable classes are closed under repeating basic operations countably many times.

\paragraph{Locally Checkable Problems on Grids as a ``Rosetta Stone''}

Locally checkable problems (LCLs) are graph problems where the correctness of a solution can be checked locally. That is, there is a set of local constraints such that if they are satisfied at every node, the solution to a problem is correct. 
Examples of such problems include vertex coloring, edge coloring, or perfect matching. We also allow inputs on each node, so the list coloring is also an example of an LCL (given that the total number of colors in all lists is finite).

In this paper, we study locally checkable problems on grid graphs in the three settings introduced above. This approach reveals a number of connections. For example, in \cref{fig:big_picture_grids}, we can see that the two classes of problems solvable in $O(1)$ and $O(\log^* n)$ steps are equivalent to different natural classes from the perspective of (finitary) factors of iid and descriptive combinatorics. On the other hand, those two worlds are richer. Many of our discussions in fact apply to graphs of bounded growth, but the grids serve as the simplest, but still quite nontrivial model in which one can try to understand all three settings together. 

Needless to say, the grid is an important graph per se -- the Ising model is studied on grids, and we already mentioned that the famous Circle Squaring Problem and other problems in descriptive combinatorics \cite{GJKS,GJKS2}, or \cite[Section~6]{CGP} boil down to constructions on grids.

\subsection{Our Contribution }

Our contribution can be mostly summed up as several new arrows in \cref{fig:big_picture_grids}. In general,  this clarifies several basic questions regarding the connections of the three fields and helps to view all three fields as part of a general theory of locality.


For example, consider the following question from the paper of Brandt, Hirvonen, Korhonen, Lempi\"{a}inen, \"{O}sterg\r{a}rd, Purcell, Rybicki, Suomela, and  Uzna\'{n}ski \cite{brandt_grids} that considered LCL problems on grids: 
\begin{quote}
\emph{
However, despite the fact that techniques seem to translate between finitary colorings and distributed complexity, it remains unclear how to directly translate results from one setting to the other in a black-box manner; for instance, can we derive the lower bound for 3-colouring from
the results of Holroyd et al., and does our complexity classification imply answers to the open questions they pose?
}
\end{quote}
We show how to 'directly translate' results between the setting of distributed algorithms and finitary factors of iid in \cref{subsec:preliminaries_tail,subsec:speedup_warmup}, using the language of uniform distributed algorithms. This indeed answers two open questions of Holroyd, Schramm, and Wilson  \cite{HolroydSchrammWilson2017FinitaryColoring}. 

To understand the connection between the world of factors of iids and descriptive combinatorics, we formalize the notion of toast construction used in both setups. We prove general results connecting the toast constructions with uniform local algorithms and this enables us to understand the $\olocal$ complexity  of a number of concrete problems like vertex $3$-coloring, edge $4$-coloring, perfect matching, and others. 

In general, we hope that analyzing uniform local complexities of uniform local algorithms will find other use-cases. 
We discuss in \cref{subsec:congestLCA} how it relates to some concepts close to distributed computing. 

We now list concrete results that we either prove here or they follow from the work of others (see \cref{fig:big_picture_grids}).

\begin{landscape}
\begin{figure}
    \centering
    \includegraphics[width = 1.5\textwidth]{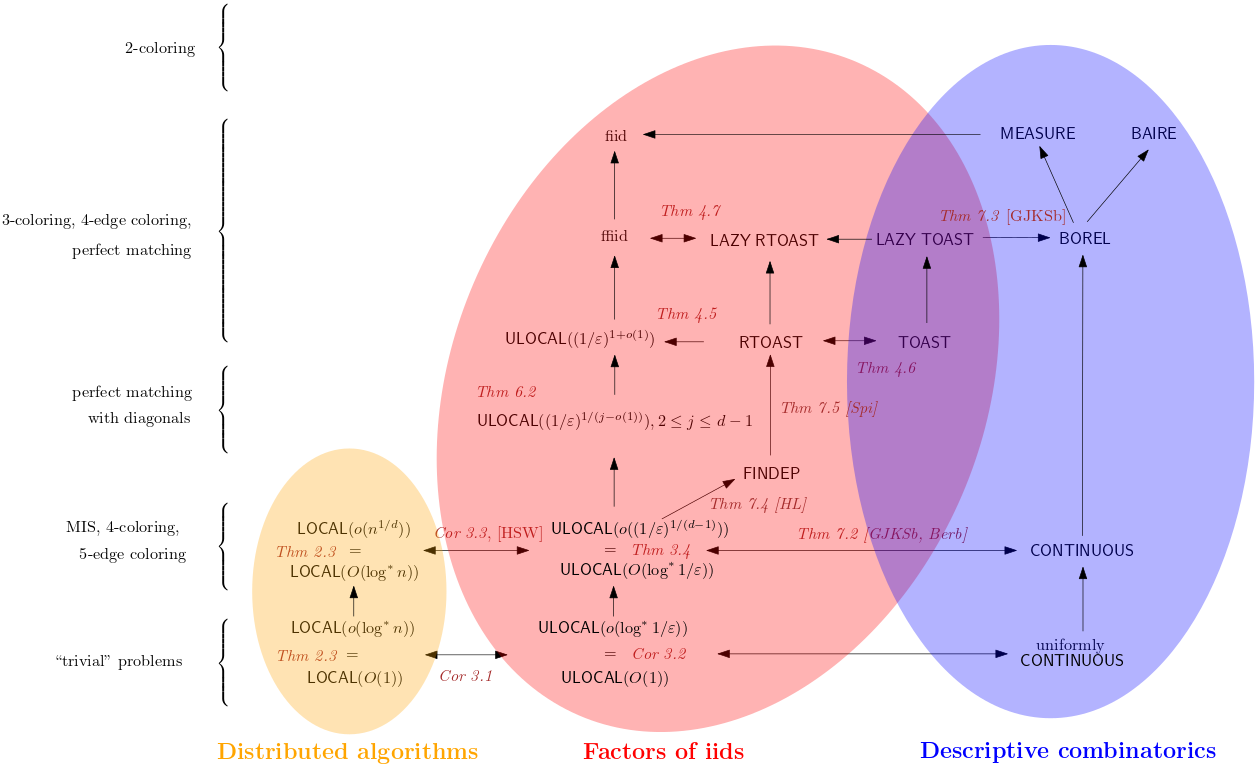}
    \caption{The big picture of local complexity classes on $d$-dimensional grids from the perspective of distributed algorithms, factors of iid, and descriptive combinatorics. The canonical example of each class is on the left side. } 
    \label{fig:big_picture_grids}
\end{figure}
\end{landscape}

\paragraph{Classification Theorems for $\olocal$}

We prove several results that one can view as arrows in \cref{fig:big_picture_grids}. Those can be viewed as an analogy to the line of work aiming for the classification of $\local$ complexities discussed in \cref{subsec:preliminaries_local}. 

\begin{itemize}
    \item $\olocal(o(\log^* 1/\eps)) = \local(O(1))$\hspace*{1cm} We show that any uniform local algorithm with uniform complexity faster than $\log^*$ function can be sped up to $O(1)$ local complexity. This answers a question of \cite{HolroydSchrammWilson2017FinitaryColoring}. 
    \item $\local(O(\log^* n)) = \olocal(O(\log^* 1/\eps)) = \cont$\hspace*{1cm} There are three natural definitions of classes of basic symmetry breaking problems in all three considered fields, and they are the same (this follows mainly from \cite{GJKS,HolroydSchrammWilson2017FinitaryColoring}). 
    As a direct application, one gets that from the three proofs that $3$-coloring of grids is not in $\local(O(\log^* n))$ \cite{brandt_grids}, not in $ \olocal(O(\log^* 1/\eps))$ \cite{HolroydSchrammWilson2017FinitaryColoring} and not in $ \cont$ \cite{GJKS} any one implies the others. These three classes are the same in a much more general setup (see \cref{subsec:CONT}). 
    \item $\olocal(o(\sqrt[d-1]{1/\eps})) = \olocal(O(\log^* 1/\eps))$\hspace*{1cm} We prove that any problem that can be solved with a uniform local algorithm with sufficiently small uniform complexity can in fact be solved in a $O(\log^* 1/\eps)$ uniform complexity. This answers a question from  \cite{HolroydSchrammWilson2017FinitaryColoring}. 
    \item Similarly to the time hierarchy theorems of the $\local$ model \cite{chang2017time_hierarchy,chang2016exp_separation}, we give examples of LCL problems with uniform complexities $\Theta(\sqrt[j - o(1)]{1/\eps})$, for any $1 \le j \le d-1$. 
\end{itemize}

\paragraph{$\toast$ constructions}
We formally define $\toast$ algorithms (these constructions are used routinely in descriptive combinatorics and the theory of random processes) and prove some general results about them. 

Let us first explain the toast constructions in more detail. A $q$-toast in the infinite graph $\mathbb{Z}^d$ is a collection $\fD$ of finite connected subsets of $\mathbb{Z}^d$ that we call \emph{pieces} such that (a) for any two nodes of the graph we can find a piece that contains both, (b) boundaries of any two different pieces are at least distance  $q$ apart. 
That is, imagine a collection of bigger and bigger pieces that form a laminar family (only inclusions, not intersections) that, in the limit, covers the whole infinite graph. 

This structure can be built in the descriptive combinatorics setups \cite{FirstToast,GJKS,Circle} but also with uniform local algorithms \cite{HolroydSchrammWilson2017FinitaryColoring,Spinka} (under the name of finitary factors of iid). 
In fact, we show that one can construct this structure with the uniform local complexity $O((1/\eps)^{1+o(1)})$ in \cref{thm:toast_in_tail}. This is best possible up to the $o(1)$ loss in the exponent. 
If one has an inductive procedure for building the toast, we can try to solve an LCL problem during the construction. That is, whenever a new piece of the toast is constructed, we define how the solution to the problem looks like in this piece. This is what we call a $\toast$ algorithm. 
There is a randomized variant of the $\toast$ algorithm and a lazy variant. In the lazy variant, one can postpone the decision to color a given piece, but it needs to be colored at some point in the future. 
Next, we prove several relations between the $\toast$ classes (see \cref{fig:big_picture_grids}).  

\begin{itemize}
    \item $\toast \subseteq \ulocal (O((1/\eps)^{1+o(1)}))$\hspace*{1cm} This result follows from the construction of the toast structure in this uniform complexity and allows us to translate toast-algorithm constructions used in descriptive combinatorics to the setting of uniform local algorithms. 
    This includes problems like vertex $3$-coloring, edge $2d$-coloring, or perfect matching.
    In the case of $3$-coloring we show that this complexity is optimal (this answers a question from  \cite{HolroydSchrammWilson2017FinitaryColoring}).
    \item $\rtoast = \toast$\hspace*{1cm} This derandomization result implies a surprising inclusion $\FINdep \subseteq \borel$ -- these two classes are discussed in \cref{sec:OtherClasses}. 
    \item $\ffiid = \lrtoast$\hspace*{1cm} Here, $\ffiid$ is defined as the union of $\olocal(t(\eps))$ over all functions $t$. 
    This result shows how the language of toast constructions can in fact characterize the whole hierarchy of constructions by uniform local algorithms. 
    \item We show that the problem of $2d$ edge coloring of $\mathbb{Z}^d$ is in $\toast$, hence in $\borel$ which answers a question from \cite{GJKS}. Independently, this was proven by \cite{ArankaFeriLaci} and \cite{Felix}. 
\end{itemize}

\subsection{Roadmap}

In \cref{sec:preliminaries} we recall the basic definitions regarding the $\local$ model and explain the formal connection between this model and finitary factors of iids. 
Next, in \cref{sec:speedup_theorems} we prove the basic speedup theorems for $\olocal$. 
In \cref{sec:Toast} we define $\toast$ algorithms and relate them to the $\olocal$ model. 
In \cref{sec:SpecificProblems} we discuss what our results mean for specific local problems and in \cref{sec:Inter} we construct problems with $\sqrt[j-o(1)]{1/\eps}$ uniform complexities for any $1 \le j \le d-1$. Finally, in \cref{sec:OtherClasses} we discuss the connection of our work to other classes of problems and finish with open questions in \cref{sec:open_problems}.
Technical proofs are collected in the appendices.

\section{Preliminaries}
\label{sec:preliminaries}

Let $d\in \mathbb{N}$ be fixed.
The class of graphs that we consider in this work contains mostly the infinite oriented grid $\mathbb{Z}^d$, that is $\mathbb{Z}^d$ with the Cayley graph structure given by the standard set of generators 
$$S=\{(1,0,\dots, 0),(0,1,\dots, 0), \dots (0,\dots,0,1)\}. $$
On the side of local algorithms, we also use the finite variant of the graph: the torus of side length $\Theta(n^{1/d})$ with edges oriented and labeled by $S$.
Usually, when we talk about a graph $G$ we mean one of the graphs above and it follows from the context whether we mean finite or infinite graph (for local complexities it is a finite graph, for uniform complexities, and in the descriptive combinatorics setups an infinite graph). 

Given a graph $G$, we denote as $V(G)$ the set of vertices, as $E(G)$ the set of edges, and as $d_G$ the \emph{graph distance} on $G$, i.e., $d_G(v,w)$ is the minimal number of edges that connects $v$ and $w$.
We write $\mathcal{B}_G(v,r)$ for the \emph{ball of radius $r\in \mathbb{N}$} around $v\in V(G)$. When $G$ is the infinite grid $\mathbb{Z}^d$, then we leave the subscripts blank.
A \emph{rooted graph} is a pair $(G,v)$ where $G$ is a graph and $v\in V(G)$.

\subsection{LCL problems on $\mathbb{Z}^d$}

Our goal is to understand locally checkable problems (LCLs) on $\mathbb{Z}^d$. We next give a formal definition. Note that formally we also allow inputs to the problem but do not really use them for specific problems in this paper.  

\begin{definition}[Locally checkable problems \cite{naorstockmeyer}]
A \emph{locally checkable problem (LCL)} on $\mathbb{Z}^d$ is a quadruple $\Pi=(\Sigma_{in}, \Sigma_{out},t,\mathcal{P})$ where $\Sigma_{in}$ and $\Sigma_{out}$ are finite sets, $t\in \mathbb{N}$ and $\mathcal{P}$ is a set of rooted grids of radius $t$ with each vertex labelled by one input label from $\Sigma_{in}$ and one output label from $\Sigma_{out}$. 

Given a graph $G$ together with a coloring $c:V(G)\to \Sigma_{in}\times\Sigma_{out}$, we say that $c$ is a \emph{$\Pi$-coloring} if
for every $v \in V(G)$ the $t$-hop neighborhood of $v$ is in $\fP$. 
\end{definition}

\subsection{Local algorithms}
\label{subsec:preliminaries_local}

Here we formally introduce the $\local$ model of deterministic and randomized distributed local algorithms and state the classification of LCLs on $\mathbb{Z}^d$ from the $\local$ point of view.

\begin{definition}[A local algorithm for $\mathbb{Z}^d$]
\label{def:local_algorithm}
A \emph{local deterministic algorithm $\mathcal{A}$} with local complexity $f(n)$ is a function that takes as input $n$ and a rooted grid of radius $f(n)$ labeled with numbers (identifiers) from $\mathbb{N}$ and input labels from $\Sigma_{in}$. It outputs a single label from $\Sigma_{out}$. 
Running a local algorithm means taking a finite torus $G$ on $n$ vertices, labeling each node of it with a unique identifier from $[n^{O(1)}]$ and applying the above function to every node to get the output label of it. 

We say that $\fA$ solves an LCL $\Pi$ if it produces a $\Pi$-coloring when applied to each finite torus as above.

A local randomized algorithms instead of unique identifiers assumes each node has an infinite (long enough) random string.
We say that it solves an LCL $\Pi$ if it produces a $\Pi$-coloring with probability $1 - 1/n^{O(1)}$. 
\end{definition}

We also say $\Pi \in \local(f(n))$ or $\Pi \in \rlocal(f(n))$ if there is a deterministic or randomized algorithm solving $\Pi$ with local complexity $f(n)$. 

\paragraph{Classification of Local Problems on $\mathbb{Z}^d$}

In this section, we state a special case of \emph{Classification of LCLs}. This is an ongoing effort in the area of distributed algorithms and our current understanding emerges as a corollary of many results that follow from several recent lines of work on local algorithms including \cite{naorstockmeyer, brandt,chang2016exp_separation,chang2017time_hierarchy, ghaffari_kuhn_maus2017slocal,ghaffari_harris_kuhn2018derandomizing,balliu2018new_classes-loglog*-log*, RozhonG19, ghaffari2019distributed, ghaffari_grunau_rozhon2020improved_network_decomposition,chang2020n1k_speedups,balliu2020almost_global_problems,brandt_grebik_grunau_rozhon2021classification_of_lcls_trees_and_grids} and many others.
For our purposes, we will need only a special case of the Classification for grids. Fortunately, here (and on trees) all possible local complexities of an LCL are completely understood. This result is stated in \cref{thm:classification_of_local_problems_on_grids}. 

\begin{theorem}[Classification of LCLs on $\mathbb{Z}^d$ \cite{naorstockmeyer, chang2016exp_separation, chang2017time_hierarchy, brandt_grids, brandt_grebik_grunau_rozhon2021classification_of_lcls_trees_and_grids}]
\label{thm:classification_of_local_problems_on_grids}
Let $\Pi$ be an LCL.
Then one of the following is true. 
\begin{enumerate}
\item $\Pi \in \local(O(1))$ and if $\Sigma_{in} = \emptyset$ then $\Pi$ is trivial, meaning that there is $\alpha\in \Sigma_{out}$ such that the constant $\alpha$ coloring is a $\Pi$-coloring,
\item $\Pi \in \local(\Theta(\log^* n)), \Pi \in \rlocal(\Theta(\log^* n))$,
\item $\Pi$ is not in $\local(o(n^{1/d}))$ or $\rlocal(o(n^{1/d}))$.
\end{enumerate}
\end{theorem}

\begin{remark}[\cite{chang2016exp_separation,Korman_Sereni_Viennot2012Pruning_algorithms_+_oblivious_coloring}]
\label{rem:chang_pettie_decomposition}
In fact, any LCL with inputs  can be solved in $\local(O(\log^* n))$ complexity as follows. First, we compute a maximal independent set (MIS) of a large enough power of the input graph, and then we run a local algorithm with just $O(1)$ local complexity.\footnote{In this sense, MIS is a complete problem for the class $\local(O(\log^* n))$. }
Moreover, the algorithm in points (1) and (2) can be uniform, see \cref{subsec:preliminaries_tail}, and the algorithm in point (1) does not need to access the input unique  identifiers. 
\end{remark}

\begin{remark}[\cite{brandt_grids}]
The problem of checking whether $\Pi \in \local(O(\log^* n))$ is undecidable. 
\end{remark}

\subsection{Uniform Algorithms}
\label{subsec:preliminaries_tail}

In this section, we define uniform algorithms and in \cref{subsec:subshifts} we show that this definition is equivalent to the notion of finitary factors of iid processes. 

An \emph{uniform local algorithm} $\fA$ is similar to the randomized local algorithm, but the definition is changed to suit to infinite graphs. 
The algorithm, run at a vertex $u$, scans its neighborhood until it scans enough of it to output a label from $\Sigma_{out}$. The algorithm needs to be correct with probability $1$. 
In particular, $\fA$ can be viewed as a function on rooted graphs with vertices labeled by both $\Sigma_{in}$ and random strings, that is, formally, $\fA=(\fC,f)$, where $\fC$ is a collection of rooted graphs (from some class) and $f:\fC\to\Sigma_{out}$ is a measurable function. 
We define the \emph{coding radius $R_\fA$} of $\fA$ as the minimal distance needed to output a solution.

\begin{remark}
In this paper, we only consider uniform local algorithms that are defined on infinite grids.
That is, $\fC$ only contains a rooted neighborhood of vertices from the infinite grid labeled with reals from $[0,1]$.
These algorithms can be run on finite tori (\cref{lem:tail->local} and \cref{thm:tail=uniform}), however, if the algorithm does not stop before it explores the whole finite graph, then the output is undefined.

It is not always possible to extend the domain of the algorithm to finite rooted tori.
To see this consider the perfect matching problem.
There is a non-trivial uniform algorithm that produces perfect matching with probability $1$, see \cref{thm:perfect_matching_intoast}.
But there are finite tori of an odd number of vertices that do not admit (existentially) perfect matching.
\end{remark}

As a measure of the complexity of a uniform local algorithm, we can choose the same measure as in the classical case, that is, the distance it needs to look at to provide a correct answer with probability $1 - 1/n^{O(1)}$. 
However, it is more natural to consider the \emph{uniform complexity} of an algorithm. Uniform complexity is a function $f(\eps)$ such that when the algorithm looks into distance $f(\eps)$, the probability it needs to look further is at most $\eps$.
In other words, for any $r \in \mathbb{N}$, the probability that the coding radius $R_\fA$ is bigger than $r$ is at most $f^{-1}(r)$.
In fact, since we have inputs, we should be more careful and emphasize that when we write, e.g.,
$$\P(R_\fA > f(\eps)) < \eps$$
we mean the following.
The left-hand side is a supremum over the events that a rooted graph of radius $f(\eps)$ (from the class that we are interested in) with $\Sigma_{in}$-labeled vertices is not in $\fA$.
Note that if there are no inputs and we work on transitive graphs, e.g. grids, then the left-hand side can be defined by the event that some, or equivalently any, vertex needs to look further than $f(\eps)$.


One can think about uniform algorithms as an analog of Las Vegas algorithms, while the classical randomized local algorithms correspond to Monte Carlo algorithms. 

The correspondence between classical and uniform randomized algorithms is the following. 
First, uniform local algorithms can be seen as a special case of randomized ones. 
This is because if we run a uniform algorithm with uniform complexity $t(\eps)$ on an $n$-node graph and mark its run as a failure if the coding radius exceeds $t(1/n^C)$, the failure at some node happens only with  probability at most $n \cdot 1/n^C = 1/n^{C-1}$. We state this as a fact. 

\begin{fact}
\label{lem:tail->local}
Let $\fA$ be a uniform local randomized algorithm on grids with uniform local complexity $t(\eps)\in o(^d\sqrt{1/\eps})$. Then the randomized local complexity of $\fA$ is $t(1/n^{\Omega(1)})$. 
\end{fact}

On the other hand, we prove that for local complexities  $o(\log n)$, the randomized complexity of a uniform algorithm is equal to its uniform complexity. 
This means that to study uniform complexities below $\log n$, it suffices to check whether existing local algorithms work without knowledge of $n$. The following lemma is stated for the class of bounded degree graphs but works for any reasonable graph class.
In particular, for grids or trees.

\begin{lemma}\label{thm:tail=uniform}
Let $\fA$ be a uniform local randomized algorithm solving an LCL problem $\Pi = (\Sigma_{in},\Sigma_{out},t,\fP)$ on the class of bounded degree graphs with randomized complexity $f(n) = o(\log n)$. Then its uniform complexity is $O(f(1/\eps))$. 
\end{lemma}
\begin{proof}
Take any $n$ large enough.
We claim that $\P(R_\fA > f(n)) < 1/n^{O(1)}$.
Suppose not.
Then by the definition and the fact that $\Delta^{f(n)+t} = o(n)$ we find an $n$-node graph $G$ on which $\fA$ needs to look further than $f(n)$ on some node with probability bigger than $1/n^{O(1)}$.
This contradicts the assumption on the randomized complexity of $\fA$.
Rearranging gives $\P(R_\fA > O(f(1/\eps)) < \eps$. 
\end{proof}
This lemma in general cannot be improved -- an $n$-node tree can be $3$-colored with a uniform algorithm with randomized local complexity $O(\log n)$ by a so-called rake-and-compress method, but the uniform complexity of $3$-coloring is unbounded, as otherwise $99\%$ of vertices of any $\Delta$-regular high girth graph could be $3$-colored. But this is impossible, by a variant of the high-girth high chromatic number graph construction.

\paragraph{Moments}
A coarser, and in a sense more instructive, way to measure the complexity of a problem $\Pi$ is to ask what moments of $R_\fA$ could be finite when $\fA$ is a uniform local algorithm that solves $\Pi$.
For example, the first moment finite corresponds to the fact that the average time until an algorithm finishes is defined (for finite graphs this is independent of the size of the graph). 
In particular, we investigate the question of what is the least infinite moment for any uniform local algorithm to solve a given problem $\Pi$, as this language was used in \cite{HolroydSchrammWilson2017FinitaryColoring}.
To make the notation simpler, we only consider LCLs without inputs.
Recall that $R_\fA$ has a \emph{finite $\alpha$-moment}, where $\alpha\in [0,+\infty)$, if
$$\E\left(\left(R_\fA\right)^\alpha\right)=\sum_{r \in \mathbb{N}} r^\alpha \cdot \P(R_\fA = r) <+\infty.$$

\begin{definition}
Let $\Pi$ be an LCL.
We define ${\bf M}(\Pi)\in [0,+\infty]$ to be the supremum over all $\alpha\in [0,+\infty)$ such that there is a uniform $\fA$ that solves $\Pi$ and has the $\alpha$-moment of $R_\fA$ finite.
\end{definition}

The $\alpha$ moment being finite roughly corresponds to $O(\sqrt[\alpha]{1/\eps})$ uniform local complexity. In fact, we now recall the following fact.

\begin{fact}
\label{fact:moment_implies_tail}
Suppose that the coding radius of a uniform local algorithm $\fA$ has the $\alpha$-moment finite.
Then 
$$\P(R_\fA > \sqrt[\alpha]{1/\eps}) = O(\eps).$$
\end{fact}

\subsection{Subshifts of finite type}
\label{subsec:subshifts}

Here we connect our definition of LCLs and uniform local algorithms with the well-studied notions of subshifts of finite type and finitary factors of iid labelings.

To define subshifts of finite type we restrict ourselves to LCLs without input.
That is $\Sigma_{in}$ is empty or contains only one element.
Moreover, for convenience, we set $b:=\Sigma_{out}$.
Then an LCL $\Pi$ is of the form $\Pi=(b,t,\fP)$.
We define $X_\Pi\subseteq b^{\mathbb{Z}^d}$ to be the space of all $\Pi$-colorings.
That is we consider all possible labelings of $\mathbb{Z}^d$ by elements from $b$ and then include in $X_{\Pi}$ only those that satisfy the finite constraints $\fP$. 
Viewing $\mathbb{Z}^d$ as a group, there is a natural shift action of $\mathbb{Z}^d$ on (the Cayley graph) $\mathbb{Z}^d$.
Namely, $g\in \mathbb{Z}^d$ shifts $h$ to $g+h$.
This action naturally ascends to an action $\cdot $ on $b^{\mathbb{Z}^d}$ as
$$(g\cdot x)(h)=x(g+h)$$
where $x\in b^{\mathbb{Z}^d}$.
It is not hard to see (since $\fP$ is defined on isomorphism types and not rooted in any particular vertex) that $X_\Pi$ is invariant under the shift map.
Then we say that $X_\Pi$ together with the shift structure is a \emph{subshift of finite type}.

A \emph{finitary factor of iid process (ffiid) $F$} is a function that takes as an input $\mathbb{Z}^d$ labeled with real numbers from $[0,1]$, independently according to Lebesgue measure, and outputs a different labeling of $\mathbb{Z}^d$, say with elements of some set $b$. 
This labeling satisfies the following: (a) $F$ is equivariant, that is for every $g\in \mathbb{Z}^d$ shifting the real labels with $g$ and applying $F$ is the same as applying $F$ and then shifting with $g$, (b) the value of $F(x)({\bf 0})$ depends almost surely on  the real labels on some random but finite neighborhood of ${\bf 0}$.
Similarly as for uniform local algorithms, one defines the random variable $R_F$ that encodes minimal such neighborhood.
We say that ffiid $F$ solves $\Pi$ if $F(x)\in X_{\Pi}$ with probability $1$.
We write $\ffiid$ for the class of all such problems.

The difference between ffiid and uniform local algorithms is just formal.
Namely, given a uniform local algorithm $\fA$, we can turn it into an ffiid $F_\fA$ simply by running it simultaneously in every vertex of $\mathbb{Z}^d$.
With probability $1$ this produces a labeling of $\mathbb{Z}^d$ and it is easy to see that both conditions from the definition of ffiid are satisfied.
On the other hand, if $F$ is ffiid, then we collect the finite neighborhoods that determine the value at ${\bf 0}$ together with the decision about the label.
This defines a pair $\fA_F=(\fC,f)$ that satisfies the definition of a uniform local algorithm.
It follows that the coding radius $R_\fA$, resp $R_F$, is the same after the translations.
We chose to work with the language of distributed algorithms as we borrow the techniques mainly from that area. 

\begin{remark}
If in (b), in the definition of ffiid, we only demand that the value at ${\bf 0}$ depends measurably on the labeling, i.e., $F$ is allowed to see the labels at the whole grid, then we say that $F$ is a \emph{factor of iid labels (fiid)}.

Given an LCL $\Pi$, we say that $\Pi$ is in the class $\FIID$, if there is a fiid $F$ that produces an element of $X_\Pi$ almost surely.
It is easy to see that $\ffiid \subseteq \FIID$.
We remark that in the setting without inputs, we have $\ffiid = \FIID$ on $\mathbb{Z}$ and it is unknown, however, whether this holds for $d>1$.
For general LCLs with inputs (where the definition of $\FIID$ is extended appropriately), we have $\ffiid \not=\FIID$ even for $d=1$, see \cite{grebik_rozhon2021LCL_on_paths}.
\end{remark}

\section{Speed-up Theorems for $\olocal$}
\label{sec:speedup_theorems}

Recall the classification \cref{thm:classification_of_local_problems_on_grids}. We are interested in building the analogous classification for uniform local complexities. 
In this section, we first show that already our basic translations from \cref{subsec:preliminaries_tail} give several interesting results and in \cref{subsec:tail_d-1_speedup} we prove a nontrivial speedup result.

\subsection{Warm-up Results}
\label{subsec:speedup_warmup}

First, we observe that already \cref{thm:classification_of_local_problems_on_grids}, together with \cref{lem:tail->local} itself has the following implications: 

\begin{corollary}
\label{cor:block_factor}
$\local(O(1)) = \olocal(O(1))$. 
\end{corollary}
\begin{proof}
This follows from \cref{thm:classification_of_local_problems_on_grids,rem:chang_pettie_decomposition}. 
\end{proof}

\begin{corollary}
\label{cor:BelowTOWERtrivial}
If $\Pi \in \olocal(o(\log^* 1/\eps))$, then $\Pi \in \local(O(1))$. In particular if $\Sigma_{in} = \emptyset$ then $\Pi$ is trivial. 
\end{corollary}
\begin{proof}
\cref{lem:tail->local} gives that $\Pi \in\rlocal(o(\log^* n))$ and \cref{thm:classification_of_local_problems_on_grids} then implies the rest. 
\end{proof}

In particular, Corollary~\ref{cor:BelowTOWERtrivial} answers in negative Problem~(ii) in \cite{HolroydSchrammWilson2017FinitaryColoring} that asked about the existence of a nontrivial local problem (without input labels) with sub $\log^*$ uniform local complexity (in their notation, an ffiid process with a super-tower function tail decay). 
The second corollary is the following:

\begin{corollary}
\label{cor:local=tower}
$\Pi \in \olocal(O(\log^* 1/\eps))$ if and only if $\Pi \in \local(O(\log^* n))$. 
\end{corollary}
\begin{proof}
If $\Pi \in \olocal(O(\log^* 1/\eps))$, \cref{lem:tail->local} yields $\Pi \in \rlocal(O(\log^* n))$. The other direction (special case of \cref{thm:tail=uniform}) follows from \cref{rem:chang_pettie_decomposition} that essentially states that it suffices to check that the independent set problem in the power graph can be solved with a uniform algorithm. But this follows from a version of Linial's algorithm from \cite{HolroydSchrammWilson2017FinitaryColoring} or from \cite{Korman_Sereni_Viennot2012Pruning_algorithms_+_oblivious_coloring} and \cref{thm:tail=uniform}.
\end{proof}

\subsection{Tight Grid Speedup for $\olocal$}
\label{subsec:tail_d-1_speedup}

In this section, we prove the speedup theorem for uniform local complexities. This is the first place where the classical randomized local model and the uniform one differ.

\begin{theorem}
\label{thm:grid_speedup}
Let $d,t\in \mathbb{N}, d \ge 2$.
Then there is $C_{d,t}>0$ such that the following is satisfied.
If an LCL $\Pi=(\Sigma_{in},\Sigma_{out},t,\fP)$ can be solved by a uniform local algorithm $\fA$ for which there is $\epsilon_0>0$ such that
$$\P(R_\fA>C_{d,t} \sqrt[d-1]{1/\eps})<\epsilon$$
holds for every $0 < \eps \le \eps_{0}$, then $\Pi\in  \olocal(O(\log^* 1/\eps))$. 
Analogously, for $d=1$ we have $\ffiid = \olocal(O(\log^* 1/\eps))$. 
\end{theorem}

The proof is a simple adaptation of Lovász local lemma speedup results in \cite{chang2017time_hierarchy,fischer2017sublogarithmic}. 
It goes as follows. Split the grid into boxes of side length roughly $r_0$ for some constant $r_0$. This can be done with $O(\log^* 1/\eps)$ uniform local complexity. 

Next, consider a box $B$ of side length $r_0$ and the volume $r_0^d$. In classical speedup argument, we would like to argue that we can deterministically set the randomness in nodes of $B$ and its immediate neighborhood so that we can simulate the original uniform local algorithm $\fA$ with that randomness and it finishes in at most $r_0$ steps. 
In fact, ``all nodes of $B$ finish after at most $r_0$ steps'' is an event that happens with constant probability if the uniform complexity of $\fA$ is $\sqrt[d]{1/\eps}$, since $r_0 = \sqrt[d]{1/\eps}$ implies $\eps = 1/r_0^d$ and there are $r_0^d$ nodes in $B$. 

However, we observe that it is enough to consider only the boundary of $B$ with $O(r_0^{d-1})$ nodes and argue as above only for the boundary nodes, hence the speedup from the uniform complexity $\sqrt[d-1]{1/\eps}$. 
The reason for this is that although the coding radius of $\fA$ inside $B$ may be very large, just the fact that there \emph{exists} a way of filling in the solution to the inside of $B$, given the solution on its boundary, is enough for our purposes. The new algorithm simply simulates $\fA$ only on the boundaries and then fills in the rest. 

\begin{proof}
We will prove the result for $d \ge 2$, the case $d=1$ is completely analogous and, in fact, simpler. 
Set $C=C_{d,t}:=1 / \left( 4d\cdot t\cdot(2\Delta)^{2\Delta}\right)$, where $\Delta=20^d$.
Suppose that $\fA$ is as above and set $r_0=C\sqrt[d-1]{1/\eps}\in \mathbb{N}$ for some $0<\epsilon\le \epsilon_0$.
Then we have
\begin{equation}\label{eq:tail}
\P(R > r_0)
<\epsilon= \frac{C^{d-1}}{ (r_0)^{d-1}}\le \frac{1}{2d\cdot t\cdot(2\Delta)^{2\Delta}(2r_0)^{d-1}}\le \frac{1}{2d\cdot t\cdot(2\Delta)^{2\Delta} (r_0+1)^{d-1}}.
\end{equation}
Write $\Box_{r_0}$ for the local coloring problem of tiling the input graph (which is either a large enough torus or an infinite grid $\mathbb{Z}^d$) into boxes with length sizes in the set $\{r_0,r_0+1\}$ such that, additionally, each box $B$ is colored with color from $[\Delta+1]$ which is unique among all boxes $B'$ with distance at most $3(r_0+t)$. 

The problem of tiling with boxes of side lengths in $\{r_0, r_0+1\}$ can be solved in $O(\log^* n)$ local complexity \cite[Theorem~3.1]{GaoJackson}, we only sketch the reason here: Find a maximal independent set of a large enough power of the input graph $\mathbb{Z}^d$. Then, each node $u$ in the MIS collects a Voronoi cell of all non-MIS nodes such that $u$ is their closest MIS node. As the shape of the Voronoi cell of $u$ depends only on other MIS vertices in a constant distance from $u$, and their number can be bounded as a function of $d$ (see the ``volume'' argument given next), one can discretize each Voronoi cell into a number of large enough rectangular boxes. Finally, every rectangular box with all side lengths of length at least $r_0^2$ can be decomposed into rectangular boxes with all side lengths either $r_0$ or $r_0+1$. 

Note that in additional $O(\log^* n)$ local steps, the boxes can then be colored with $(\Delta+1)$-many colors in such a way that boxes with distance at most $3(r_0+t)$ must have different colors.
To see this, we need to argue that there are at most $\Delta$ many boxes at distance at most $3(r_0+t)$ from a given box $B$.
This holds because if $B,B'$ are two boxes and $d(B,B')\le 3(r_0+t)$, then $B'$ is contained in the cube centered around the center of $B$ of side length $10(r_0+t)$.
Then a volume argument shows that there are at most $\Delta\ge \frac{(10(r_0+t))^d}{r_0^d}$ many such boxes $B'$ (when $B$ is fixed).

We show now that there is a local algorithm $\mathcal{A'}$ of constant locality that solves $\Pi$ when given the input graph $G$ together with a solution of $\Box_{r_0}$ as an input.
This clearly implies that $\Pi$ is in the class $\LOCAL(O(\log^* n))$.

Let $\{B_i\}_{i\in I}$ be the input solution of $\Box_{r_0}$, together with the coloring $B_i\mapsto c(B_i)\in [\Delta+1]$. 
Write $V(B)$ for the vertices of $B$ and $\partial B$ for the vertices of $B$ that have a distance at most $t$ from the boundary of $B$.
We have $|\partial B| \le 2d\cdot t \cdot (r_0+1)^{d-1}$. 
Define $N(B)$ to be the set of rectangles of distance at most $(r_0+t)$ from $B$.

\begin{figure}
    \centering
    \includegraphics[width= .8\textwidth]{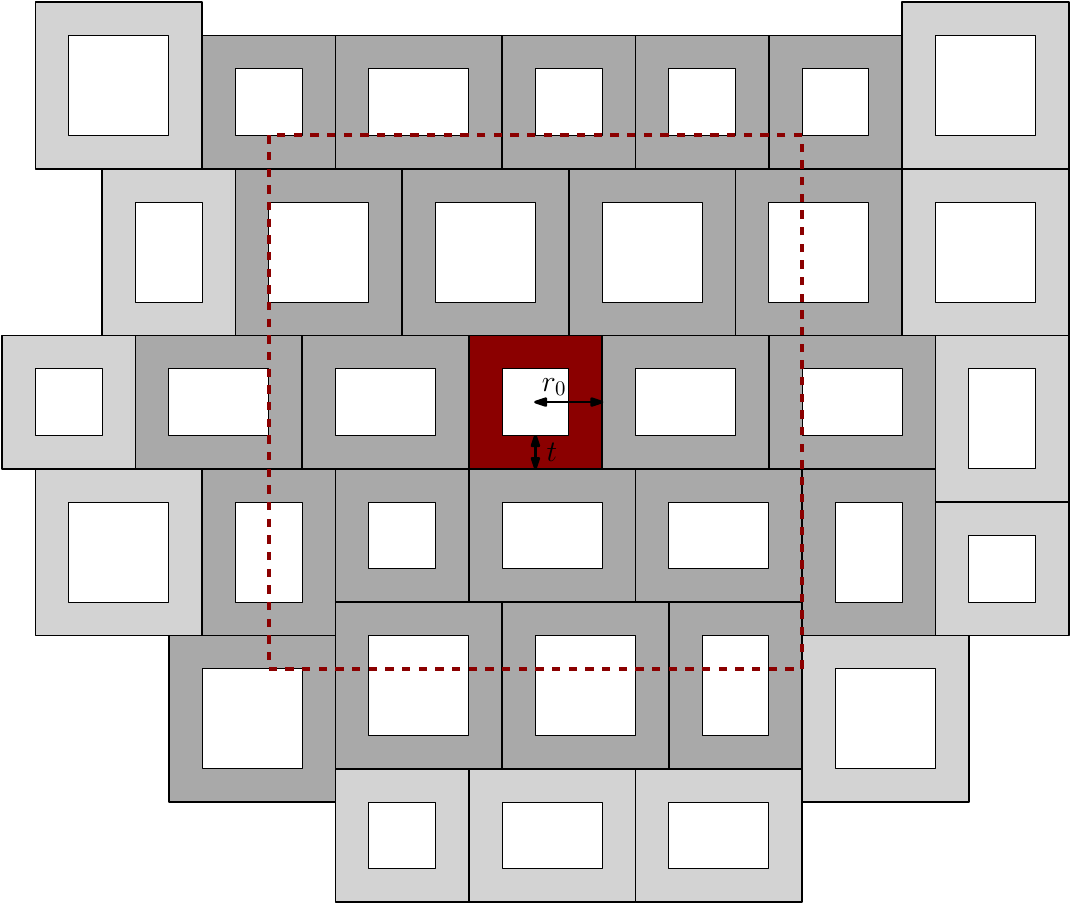}
    \caption{The picture shows one block $B$ with a red boundary $\partial B$ of length $t$. Boxes in $N(B)$, that is with distance at most $(r_0+t)$ from $B$ are dark grey. }
    \label{fig:tail_LLL}
\end{figure}

We are now going to describe how $\fA'$ iterates over the boxes in the order of their colors and for each node $u$ in the box $B$ it carefully fixes a random string $\mathfrak{r}(u)$.
This is done in such a way that after fixing the randomness in all nodes, we can apply the original algorithm $\fA$ to each vertex in $\partial B$ for each $B$ with the randomness given by $\mathfrak{r}$ and the local algorithm at $u \in \partial B$ will not need to see more than its $r_0$-hop neighborhood. 

The algorithm $\mathcal{A}'$ runs in $\Delta+1$ steps that are indexed by $1\le \ell\le \Delta+1 $, and in each step, we denote a $\mathfrak{r}_{\ell}$ the partial function that assigns randomness to all nodes in boxes of color at most $\ell$.
Given such a partial function $\mathfrak{r_\ell}:V(G)\to [0,1]$, every box $B$ can compute a conditional probability $\P(\Epsilon_{\ell}(B))$, where $\Epsilon_{\ell}(B)$ is the (bad) event that there is a vertex in $\partial B$ that needs to look further than $r_0$ to apply the uniform local algorithm $\fA$ given the randomness fixed by $\mathfrak{r}_{\ell}$.
Note that for computing $\P(\Epsilon_{\ell}(B))$ it is enough to know the $r_0$-neighborhood of $B$ and the restriction of $\mathfrak{r}_{\ell}$ to this neighborhood. 

Additionally, in the $0$-th step, we set $\mathfrak{r}_0$ to be the empty function.
Note that the probability that a vertex needs to look further than $r_0$ in the algorithm $\fA$ is smaller than $\frac{1}{2d\cdot t\cdot(2\Delta)^{2\Delta} (r_0+1)^{d-1}}$ by \eqref{eq:tail}.
The union bound gives $\P(\Epsilon_0(B))<\frac{1}{(2\Delta)^{2\Delta}}$ for every $B$.

In each step $\ell\in [\Delta+1]$ we define $\mathfrak{r}_B:V(B)\to [0,1]$ for every $B$ such that $c(B)=\ell$ and put $\mathfrak{r}_\ell$ to be the union of all $\mathfrak{r}_D$ that have been defined so far, i.e., for $D$ such that $c(D)\le \ell$. 
We make sure that the following is satisfied for every $B$:
\begin{enumerate}
	\item [(a)] the assignment $\mathfrak{r}_B$ depends only on $2(r_0+t)$-neighborhood of $B$ (together with previously defined $\mathfrak{r}_{\ell-1}$ on that neighborhood) whenever $c(B)=\ell$,
	\item [(b)] $\P(\Epsilon_\ell(B))<\frac{1}{(2\Delta)^{2\Delta-\ell}}$,
	\item [(c)] restriction of $\mathfrak{r}_\ell$ to $N(B)$ can be extend to the whole $\mathbb{Z}^d$ such that $\fA$ is defined at every vertex, where we view $N(B)$ embedded into $\mathbb{Z}^d$.
\end{enumerate}
We show that this gives the required algorithm.
It is easy to see that by (a) this defines a local construction that runs in constantly many steps and produces a function $\mathfrak{r}:=\mathfrak{r}_{\Delta+1}:V(G)\to [0,1]$.
By (b) every vertex in $\partial B$ for some $B$ looks into its $r_0$ neighborhood and uses $\fA$ to choose its color according to $\mathfrak{r}$.
Finally, by (c) it is possible to assign colors in $B\setminus \partial B$ for every $B$ to satisfy the condition of $\Pi$ within each $B$.
Note that as the thickness of the boundary is more than $t$, this depends only on the output of $\fA$ on the boundary.
Picking a minimal such solution for each $B\setminus \partial B$ finishes the description of the algorithm.

We show how to proceed from $0\le \ell<\Delta+1$ to $\ell+1$.
Suppose that $\mathfrak{r}_\ell$ satisfies the conditions above.
Pick $B$ such that $c(B)=\ell+1$.
Note that $\mathfrak{r}_B$ that we want to find can influence rectangles $D\in N(B)$ and these are influenced by the function $\mathfrak{r}_\ell$ at distance at most $r_0$ from them.
Also given any $D$ there is at most one $B\in N(D)$ such that $c(B)=\ell+1$. 
We make the decision $\mathfrak{r}_B$ based on $2r_0$ neighborhood around $B$ and the restriction of $\mathfrak{r}_\ell$ there, this guarantees that condition (a) is satisfied.

Fix $D\in N(B)$ and pick $\mathfrak{r}_B$ uniformly at random.
Write $\mathfrak{S}_{B,D}$ for the random variable that computes the conditional probability of $\Epsilon_{\mathfrak{r}_B\cup \mathfrak{r}_\ell}(D)$.
We have $\E(\mathfrak{S}_{B,D})=\P(\Epsilon_\ell(D))<\frac{1}{(2\Delta)^{2\Delta-\ell}}$.
By Markov inequality we have
$$\P\left(\mathfrak{S}_{B,D}\ge 2\Delta \frac{1}{(2\Delta)^{2\Delta-\ell}}\right)<\frac{1}{2\Delta}.$$
Since $|N(B)|\le \Delta$ we have that with non-zero probability $\mathfrak{r}_B$ satisfies $\mathfrak{S}_{B,D}<\frac{1}{(2\Delta)^{2\Delta-\ell-1}}$ for every $D\in N(B)$.
Since $\fA$ is defined almost everywhere we find $\mathfrak{r}_B$ that satisfies (b) and (c) above.
This finishes the proof. 
\end{proof}

As we already mentioned, the proof is in fact an instance of the famous Lovász Local Lemma problem. This problem is extensively studied in the distributed computing literature \cite{chang2016exp_separation, fischer2017sublogarithmic, brandt_maus_uitto2019tightLLL, brandt_grunau_rozhon2020tightLLL} as well as descriptive combinatorics \cite{KunLLL,OLegLLL,MarksLLL,Bernshteyn2021LLL}.
Our usage corresponds to a simple LLL algorithm for the so-called exponential criteria setup (see \cite{fischer2017sublogarithmic} and \cite{Bernshteyn2021local=cont}).

\section{$\toast$: General Results}
\label{sec:Toast}

In this section, we introduce the toast construction. This construction is very standard in descriptive combinatorics \cite{GaoJackson,FirstToast,GJKS,Circle} and the theory of random processes (there it is often called \emph{the hierarchy}) \cite{HolroydSchrammWilson2017FinitaryColoring,ArankaFeriLaci,Spinka}.

While the construction itself is well-known, in this paper we study it as a class of problems by itself. An analogy to our work is perhaps the paper of \cite{ghaffari_harris_kuhn2018derandomizing} that defines the class $\slocal$ and shows how to derandomize it. 

Informally, the basic toast construction works as follows. There is an adversary that has access to the (infinite) input grid $\mathbb{Z}^d$. The adversary is sequentially choosing some connected pieces of the grid and gives them to our algorithm $\fA$. The algorithm $\fA$ needs to label all vertices in the piece $D$ once it is given $D$. 
Note that $\fA$ does not know where exactly $D$ lies in $\mathbb{Z}^d$ and its decision cannot depend on the precise position of $D$ in the input graph. In general, the piece $D$ can be a superset of some smaller pieces that were already colored by $\fA$. In this case, $\fA$ needs to extend their coloring to the whole $D$. This game is played for a countable number of steps until the whole grid is colored. The algorithm solves $\Pi$ if it never happens during the procedure that it cannot extend the labeling to $D$ in a valid manner.

The construction sketched above is what we call a spontaneous deterministic toast or just $\toast$ algorithm. 
In the next subsection, we give a formal definition of $\toast$ algorithms and we follow by proving relations between $\toast$ algorithms and $\olocal$ complexities. 

\subsection{Formal Definition of $\toast$ classes}
\label{subsec:toast_definition}
We now give a formal definition of our four $\toast$ classes. First, we define the hereditary structure on which the $\toast$ algorithm is interpreted. This is also called a \emph{toast} (we use lowercase letters) and is parameterized by the minimum distance of boundaries of different pieces of the toast.  
This structure also corresponds to the adversary that asks our algorithm to extend a partial solution to $\Pi$ to bigger and bigger pieces used in the informal introduction. 

\begin{definition}[$q$-toast on $\mathbb{Z}^d$]
For $q >0$, a \emph{$q$-toast on $\mathbb{Z}^d$} is a collection $\mathcal{D}$ of finite connected subsets of $\mathbb{Z}^d$ that we call \emph{pieces} such that
\begin{itemize}
    \item for every $g,h\in \mathbb{Z}^d$ there is a piece $D\in\mathcal{D}$ such that $g,h\in D$,
    \item if $D,E\in \mathcal{D}$ and $d(\partial D,\partial E)\le q$, then $D=E$, where $\partial D$ is the boundary of a set $D$.
\end{itemize}
\end{definition}

Given a $q$-toast $\mathcal{D}$, the \emph{rank} of its piece $D \in \fD$ is the level in the toast hierarchy that $D$ belongs to in $\fD$. Formally, we define a \emph{rank function $\mathfrak{r}_\fD$ for the toast $\mathcal{D}$} as $\mathfrak{r}_\mathcal{D}:\mathcal{D}\to \mathbb{N}$, inductively as
$$\mathfrak{r}_\mathcal{D}(D)=\max\{\mathfrak{r}_{\mathcal{D}}(E)+1:E\subsetneq D, \ E\in \mathcal{D}\}$$ 
for every $D\in \mathcal{D}$.

Next, $\toast$ algorithms are described in the following language of extending functions. 

\begin{definition}[$\toast$ algorithm]
\label{def:extending_function}
We say that a function $\fA$ is a \emph{$\toast$ algorithm} for an LCL $\Pi = (\Sigma_{in},\Sigma_{out}, t, \fP)$ if 
\begin{enumerate}
    \item [(a)] the domain of $\fA$ consists of pairs $(A, c)$, where $A$ is a finite connected $\Sigma_{in}$-labeled subset of $\mathbb{Z}^d$ and $c$ is a partial map from $A$ to $\Sigma_{out}$,
    \item [(b)] $\fA(A,c)$ is a partial coloring from $A$ to $\Sigma_{out}$ that extends $c$, i.e., $\fA(A,c)\upharpoonright \dom(c)=c$, and that depends only on the isomorphism type of $(A,c)$ (not on the position in $\mathbb{Z}^d$).  
\end{enumerate}
A \emph{randomized $\toast$ algorithm} is moreover given as an additional input labeling of $A$ with real numbers from $[0,1]$.  
Finally, if $\fA$ has the property that $\dom(\fA(A,c)) =A$ for every pair $(A,c)$, then we say that $\fA$ is a \emph{spontaneous (randomized) $\toast$ algorithm}.
\end{definition}

\begin{figure}
    \centering
    \includegraphics[width=\textwidth]{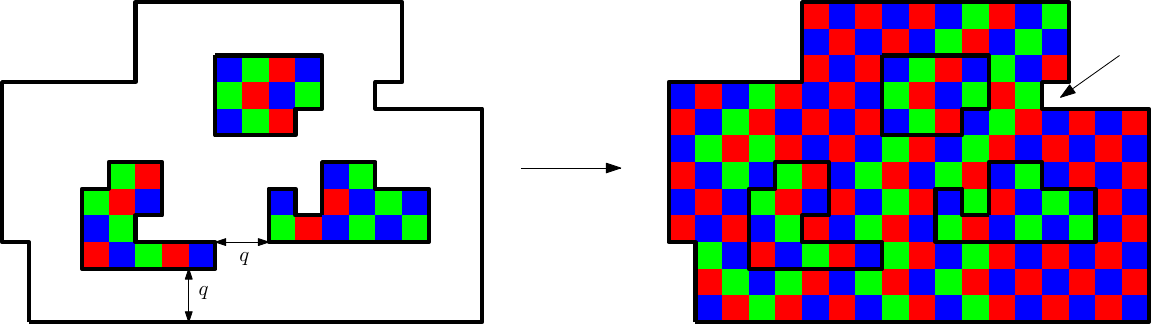}
    \caption{A $q$-toast and an example of an extending function for $3$-coloring. Observe that our extending function does not give a correct algorithm for $3$-coloring since the piece on the right is contained in an even bigger piece and the $3$-coloring constructed on the right cannot be extended to the node marked with an arrow}
    \label{fig:toast_example}
\end{figure}

It is easy to see that a $\toast$ algorithm $\fA$ can be used to color ($\Sigma_{in}$-labeled) $\mathbb{Z}^d$ inductively on the rank $\mathfrak{r}_\mathcal{D}$ of any $q$-toast $\mathcal{D}$ for every $q\in \mathbb{N}$. 
Namely, define $\fA(\mathcal{D})$ to be the partial coloring of $\mathbb{Z}^d$ that is the union of $\fA(D)$, where $D\in \mathcal{D}$ and $\fA(D)$ is defined inductively on $\mathfrak{r}_\mathcal{D}$.
First put $\fA(D):=\fA(D,\emptyset)$ whenever $\mathfrak{r}_\mathcal{D}(D)=0$.
Having defined $\fA(D)$ for every $D\in \mathcal{D}$ such that $\mathfrak{r}_\mathcal{D}(D)=k$, we define
$$\fA(D'):=\fA\left(D',\bigcup\{\fA(E):E\subsetneq D', \ E\in \mathcal{D}\}\right).$$
for every $D'\in \mathcal{D}$ such that $\mathfrak{r}_\mathcal{D}(D')=k+1$. That is, we let $\fA$ extend the partial solution to $\Pi$ to bigger and bigger pieces of the toast $\fD$. 
Observe that if $\fA$ is a spontaneous $\toast$ function, then $\dom(\fA(\mathcal{D}))=\mathbb{Z}^d$ for every $q$-toast $\mathcal{D}$ and every $q\in \mathbb{N}$.

\paragraph{$\toast$ algorithms}
We are now ready to define the four types of toast constructions that we are going to use (cf. \cref{fig:big_picture_grids}). 

\begin{definition}[$\ltoast$]
Let $\Pi$ be an LCL.
We say that $\Pi$ is in the class $\ltoast$ if there is $q>0$ and a $\toast$ algorithm $\fA$ such that $\fA(\mathcal{D})$ is a $\Pi$-coloring for every $q$-toast $\mathcal{D}$.

Moreover, we say that $\Pi$ is in the class $\toast$ if there is $q\in \mathbb{N}$ and a spontaneous $\toast$ algorithm $\fA$ such that $\fA(\mathcal{D})$ is a $\Pi$-coloring for every $q$-toast $\mathcal{D}$.
\end{definition}

Before we define a randomized version of these notions, we make several comments.
First, note that $\toast\subseteq \ltoast$.
Moreover, if we construct a ($q$-)toast using a uniform local algorithm, then we get instantly that every problem from the class $\toast$ is in the class $\olocal$, see Section~\ref{subsec:toast_construction_in_tail_sketch}.
Similar reasoning applies to the class $\borel$, see Section~\ref{sec:OtherClasses}.
We remark that we are not aware of any solvable LCL that is not in the class $\toast$.

\begin{definition}[$\lrtoast$]
Let $\Pi$ be an LCL.
We say that $\Pi$ is in the class $\lrtoast$ if there is $q>0$ and a randomized $\toast$ algorithm $\fA$ such that $\fA(\mathcal{D})$ is a $\Pi$-coloring for every $q$-toast $\mathcal{D}$ with probability $1$.
That is, after $\Sigma_{in}$-labeling is fixed, with probability $1$ an assignment of real numbers from $[0,1]$ to $\mathbb{Z}^d$ successfully produces a $\Pi$-coloring against all possible $q$-toasts $\mathcal{D}$ when $\fA$ is applied.

Moreover, we say that $\Pi$ is in the class $\rtoast$ if there is $q\in \mathbb{N}$ and a spontaneous randomized $\toast$ algorithm $\fA$ as above.
\end{definition}

\subsection{Uniform Local Construction of $\toast$}
\label{subsec:toast_construction_in_tail_sketch}

The class $\toast$ (or $\rtoast$) is especially useful since in \cref{sec:SpecificProblems} we discuss how many local problems that are not in class $O(\log^* n)$ but on the other hand are not ``clearly hard'' like vertex $2$-coloring are in the class $\toast$. Importantly, in \cref{subsec:toast_construction}  we show how one can construct the toast decomposition with a uniform local algorithm in $\olocal(\frac{1}{\eps} \cdot 2^{O(\sqrt{\log 1/\eps})})$ which yields the following theorem.

\begin{restatable}{theorem}{toastconstruction}
\label{thm:toast_in_tail}
Let $\Pi$ be an LCL that is in the class $\toast$.
Then \[
\Pi \in \olocal \left( \frac{1}{\eps} \cdot 2^{O(\sqrt{\log 1/\eps})} \right) \subseteq \olocal\left(  (1/\eps)^{1 + o(1)}  \right)
\]
and, hence, ${\bf M}(\Pi)\ge 1$.
\end{restatable}

We observe in \cref{cor:toast_lower_bound} that faster than $O(1/\eps)$ uniform complexity is not possible. 

\paragraph{Proof Sketch}

Here we provide an intuition for the proof that is otherwise somewhat technical. We use that maximal independent sets of power graphs of the grid $G = \mathbb{Z}^d$ can be constructed effectively by the uniform distributed MIS algorithm of Ghaffari \cite{ghaffari2016MIS}. The slowdown incurred by constructing MISes can be neglected in this high-level overview.  

We take a sequence $r_k = 2^{ck^2}$ and construct sparser and sparser MIS in the power graphs $G^{r_k}$ for $k = 1, 2, \dots$.
Recall that edges in the power graph $G^\ell$ are pairs of distinct vertices of graph distance at most $\ell$ in the original graph $G$.
After constructing MIS of $G^{r_k}$, we construct Voronoi cells induced by this MIS. These Voronoi cells $V_1, V_2, \dots$ will be, roughly, the pieces of rank $k$, but we need to ensure that all the boundaries are at least $q$ hops apart. 

To do so, we first shrink each Voronoi cell by $t_k = O(\sum_{j < k} r_j) \ll r_k$ hops (grey boundary area in \cref{fig:toast_construction}) and get new, smaller, cells $C_1, C_2, \dots$. The final sequence of pieces $H_1, H_2, \dots$ is then defined inductively, namely to define $H_i$ we start with $C_1$, then add all Voronoi cells from $k-1$-th iteration that $C_1$ intersects, then add all Voronoi cells from $k_2$-th iteration that the new, bigger, cluster intersects, \dots
The definition of $t_k$ ensures that the final pieces constructed after all $k$ iterations are sufficiently far from each other, and also compatible with the smaller rank pieces. We always have $H_i \subseteq V_i$. 

It remains to argue that after $k$ steps of this process, only a small fraction of nodes remains uncovered. This is done roughly as follows. The vertices from $V_i \setminus C_i$ that are not covered by $C_i$s are roughly $t_i / r_i \approx A r_{i-1} / r_i$ fraction of vertices where $A$ is some (large) constant. 
If we now fix some step $k$, we can use this argument for all $1 \le i \le k$ to conclude that the fraction of remaining vertices after step $k$ is at most $\frac{A r_{k-1}}{r_{k}} \cdot \frac{A r_{k-2}}{r_{k-1}} \dots \frac{A }{r_1} = \frac{A^k}{r_k}$. 
To compute the uniform complexity, consider any $r$ such that $r_k \le r < r_{k+1}$. Above discussion tells us that we can write \[ \P(R \ge r) \le \frac{A^k}{r_k}.\] 
On the other hand, $r \le r_{k+1} = r_k \cdot 2^{O(k)}$, so we have \[ \P(R \ge r) \le \frac{A^k}{r / 2^{O(k)}} = \frac{2^{O(k)}}{r}.\] 
Finally, since $r \ge r_k = 2^{ck^2}$, we have $k = O(\sqrt{\log r})$ and thus
\[ \P(R \ge r) \le  \frac{2^{O(\sqrt{\log r})}}{r}.\] 
A substitution $\eps = \frac{2^{O(\sqrt{\log r})}}{r}$ then gives the inequality $\P(R \ge 2^{O(\sqrt{\log 1/\eps})}/\eps) < \eps$ that we wanted to prove.

\begin{figure}
    \centering
    \includegraphics[width = .95\textwidth]{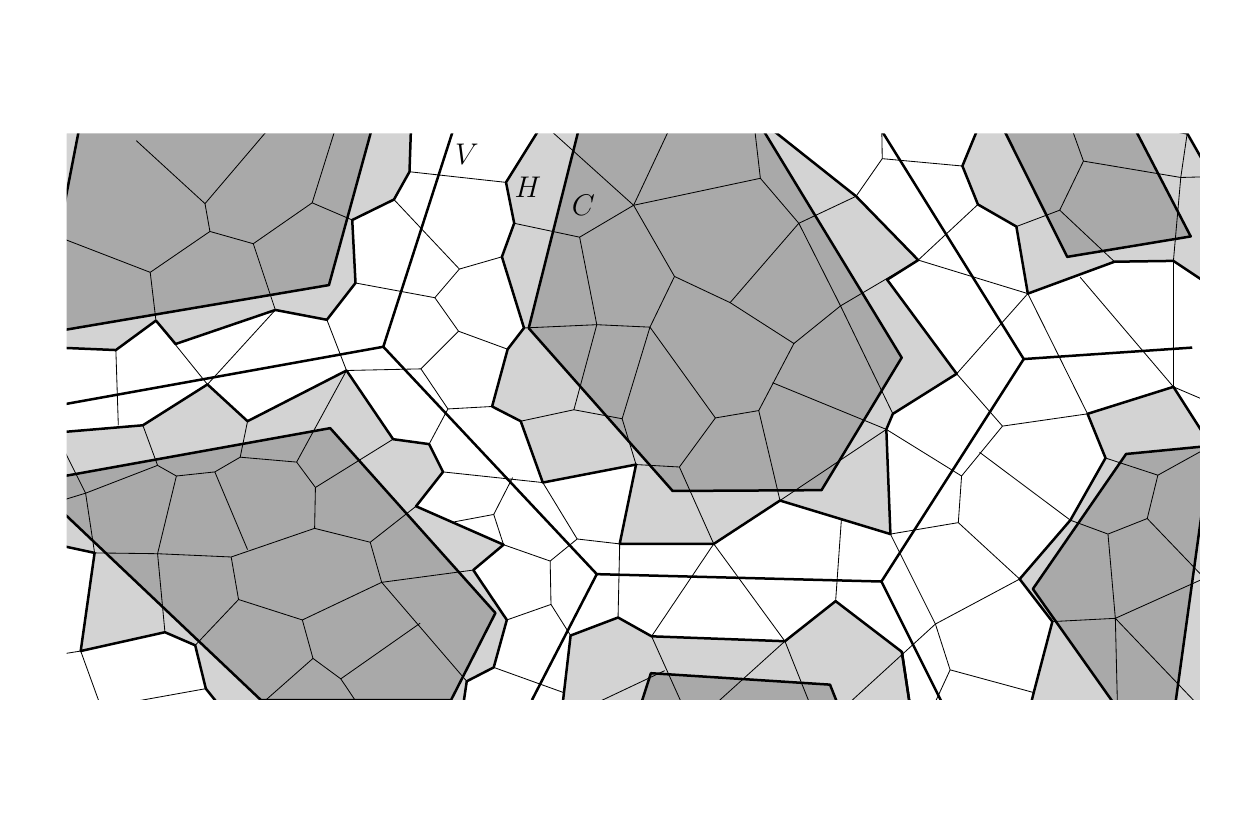}
    \caption{A step during the toast construction. Voronoi cell $V$ is first shrunk to $C$ and we get the final piece $H$ by taking the union of smaller-level Voronoi cells intersecting $C$ (in this picture, there is only one lower level but the construction is recursive in general). }
    \label{fig:toast_construction}
\end{figure}

\subsection{$\rtoast = \toast$}
\label{subsec:rtoast=toast}

In this section, we show that randomness does not help spontaneous toast algorithms.

\begin{restatable}{theorem}{toastrtoast}\label{thm:derandomOfToast}
Let $\Pi$ be an LCL.
Then $\Pi$ is in the class $\rtoast$ if and only if $\Pi$ is in the class $\toast$.
\end{restatable}

\begin{proof}
It is easy to see that $\toast\subseteq \rtoast$.
Suppose that $\Pi$ is in the class $\rtoast$ and fix a witnessing $q\in \mathbb{N}$ together with a spontaneous randomized $\toast$ algorithm $\fA$.
We may assume that $t$, the locality of $\Pi$, satisfies $t\le q$.

Let $(A,c)$ be such that $A\subseteq \mathbb{Z}^d$ is finite and connected and $\dom(c)\subseteq A$.
We say that $(A,c)$ comes from a \emph{partial random $q$-toast} if there is a collection $\mathcal{D}$ of connected subsets of $A$ together with an event $\mathcal{F}$ on $\dom(c)$, i.e., $\mathcal{F}$ is a measurable subset of $[0,1]$ assignments to vertices from $A$, such that
\begin{itemize}
    \item if $D,E\in \mathcal{D}$ and $d(\partial D,\partial E)\le q$, then $D=E$,
    \item $A\in \mathcal{D}$,
    \item $\bigcup_{E\in \mathcal{D}\setminus \{A\}}E=\dom(c)$,
    \item $\mathcal{F}$ has a non-zero probability. If we fix a particular assignment of randomness from $\mathcal{F}$ and apply inductively $\fA$ to $\mathcal{D}\setminus \{A\}$ we obtain $c$.
\end{itemize}
If $(A,c)$ comes from a partial random $q$-toast, then we let $\mathcal{E}(A,c)$ to be any event $\mathcal{F}$ as above.
Otherwise, we put $\mathcal{E}(A,c)$ to be any event on $\dom(c)$ of non-zero probability.
It is easy to see that $\mathcal{E}$ does not depend on the position of $A$ in $\mathbb{Z}^d$.
This is because $\fA$ is such.

We now define a spontaneous toast algorithm $\fA'$.
Given $(A,c)$, we let $\mathcal{F}(A,c)$ to be any event on $A$ that has non-zero probability, satisfies
$$\mathcal{F}(A,c)\upharpoonright \dom(c)\subseteq \mathcal{E}(A,c)$$
and $\fA(A,c)$ gives the same outcome on $\mathcal{F}(A,c)$.
Let $\fA'(A,c)$ be equal to this outcome on $\mathcal{F}(A,c)$.
It is clear that $\fA'$ is a spontaneous $\toast$ algorithm.
Observe that if $(A,c)$ comes from a partial random $q$-toast, then $(A,\fA'(A,c))$ comes from a partial random $q$-toast.
Namely, use the same $q$-toast witness $\mathcal{D}$ as for $(A,c)$ together with the event $\mathcal{F}(A,c)$.

It remains to show that $\fA'(\mathcal{D})$ is a $\Pi$-coloring for every $q$-toast $\mathcal{D}$.
Induction on $\mathfrak{r}_{\mathcal{D}}$ shows that $(D,\fA'(D))$ comes from a partial random $q$-toast for every $D\in \mathcal{D}$.
This is easy when $\mathfrak{r}_{\mathcal{D}}(D)=0$ and follows from the observation above in general.
Let $v\in \mathbb{Z}^d$ and pick a minimal $D\in \mathcal{D}$ such that there is $D'\in \mathcal{D}$ that satisfies $v\in D'$ and $D'\subseteq D$.
Since $t\le q$ we have that $\fB(v,t)\subseteq D$.
Since $(D,\fA'(D))$ comes from a partial random $q$-toast it follows that there is a non-zero event $\mathcal{F}$ such that 
$$\fA'(D)\upharpoonright \fB(v,t)=\fA(\mathcal{D})\upharpoonright \fB(v,t)$$
on $\mathcal{F}$.
This shows that the constraint of $\Pi$ at $v$ is satisfied and the proof is finished.
\end{proof}


\subsection{$\ffiid = \rltoast$}
\label{subsec:tail=randomlazytoast}

In this section, we prove that the class $\rltoast$ in fact covers the whole class $\ffiid = \bigcup_{t(\eps)} \olocal(t(\eps))$. 

\begin{restatable}{theorem}{tailrltoast}
\label{thm:tail=rltoast}
Let $\Pi$ be an LCL.
Then $\Pi$ is in the class $\lrtoast$ if and only if $\Pi$ is in the class $\ffiid$.
\end{restatable}

\begin{proof}
Suppose that $\Pi$ is in the class $\lrtoast$, i.e, there is $q\in\mathbb{N}$ together with a randomized $\toast$ algorithm $\fA$ that solves $\Pi$.
It follows from Theorem~\ref{thm:toast moment} that there is a uniform local algorithm $F$ that produces $q$-toast with probability $1$.

Define a uniform local algorithm $\fA'$ as follows.
Let $\fB(v,r)$ be a $\Sigma_{in}$-labeled neighborhood of $v$.
Assign real numbers from $[0,1]$ to $\fB(v,r)$ and run $F$ and $\mathcal{A}$ on $\fB(v,r)$.
Let $a_{v,r}$ be the supremum over all $\Sigma_{in}$-labelings of $\fB(v,r)$ of the probability that $\fA'$ fails, that is, the output at $v$ is not defined.
If $a_{v,r}\to 0$ as $r \to \infty$, then $\fA'$ is a uniform local algorithm.

Suppose not, i.e., $a_{v,r}>a>0$.
Since $\Sigma_{in}$ is finite, we find by a compactness argument a labeling of $\mathbb{Z}^d$ with $\Sigma_{in}$ such that the event $\mathcal{F}_r$ that $\fA'$ fails on $\fB(v,r)$ with the restricted labeling has probability bigger than $a$.
Note that since the labeling is fixed, we have $\mathcal{F}_r\supseteq \mathcal{F}_{r+1}$.
Set $\mathcal{F}=\bigcap_{r\in \mathbb{N}} \mathcal{F}_r$.
Then by $\sigma$-additivity, we have that the probability of the event $\mathcal{F}$ is bigger than $a>0$.
However, $F$ produces toast almost surely on $\mathcal{F}$ and $\fA$ produces $\Pi$-coloring with probability $1$ for any given toast.
This contradicts the fact that the probability of $\mathcal{F}$ is not $0$.

The idea to prove the converse is to ignore the $q$-toast part and wait for the coding radius.
Suppose that $\Pi$ is in the class $\ffiid$.
By the definition, there is a uniform local algorithm $\fA$  of complexity $t(\eps)$ that solves $\Pi$. 
Given an input labeling of $\mathbb{Z}^d$ by $\Sigma_{in}$, we see by the definition of $R_\fA$ that every node $v$ must finish after finitely many steps with probability $1$.
Therefore it is enough to wait until this finite neighborhood of $v$ becomes a subset of some piece of a given toast $\mathcal{D}$.
This happens with probability $1$ independently of the given toast $\mathcal{D}$.
\end{proof}

\section{$\toast$: Specific Problems}\label{sec:SpecificProblems}

In this section, we focus on specific problems that are not in the class $\LOCAL(O(\log^* n))$ but they admit a solution by a (spontaneous) $\toast$ algorithm. Recall, that proving that $\Pi \in \toast$ implies that there is a uniform local algorithm with uniform local complexity $(1/\eps)^{1+o(1)}$ via \cref{thm:toast_in_tail} and we discuss in \cref{subsec:descriptive_combinatorics_discussion} that by a result of Gao, Jackson, Krohne, and Seward \cite{GJKS2} it also implies $\Pi\in \borel$. 

The problems we discuss are the following: (a) vertex coloring, (b) edge coloring, and (c) tiling the grid with boxes from a prescribed set (this includes the perfect matching problem).



All these problems were recently studied in the literature from different points of view, see \cite{ArankaFeriLaci, HolroydSchrammWilson2017FinitaryColoring, GJKS, spencer_personal}.
In all cases, the main technical tool, sometimes in a slight disguise, is the toast construction.
Our aim here is to show these problems in a unified context.
We note that sometimes the results in the aforementioned literature are stronger or work in a bigger generality (for example, all results discussed here can also be proven to hold on unoriented grids as discussed in \cite{ArankaFeriLaci}, or the constructed coloring could be aperiodic \cite{spencer_personal}).
In what follows we only give a ``proof by picture'', sketch the description of the corresponding spontaneous toast algorithm and refer the reader to other papers for more details and stronger results.

Throughout this section, we assume that $d>1$.
We refer the reader to \cref{sec:OtherClasses} for the definition of classes $\cont$, $\borel$, etc.
Our main contribution in this section is the following:
\begin{enumerate}
    \item The fact that the problem of vertex $3$-coloring is in $\toast$ together with the speed-up \cref{thm:grid_speedup} answers a question (i) from \cite{HolroydSchrammWilson2017FinitaryColoring}. 
    \item We observe that the problem of edge $2d$-coloring the grid $\mathbb{Z}^d$ is in $\toast$, thereby answering a question of \cite{GJKS} coming from the area of descriptive combinatorics. Independently, this was proven in \cite{ArankaFeriLaci} and \cite{Felix}. 
\end{enumerate}

In general, our approach gives that if an LCL $\Pi$ is in the class $\toast\setminus \local(O(\log^*(n)))$, then $1\le {\bf M}(\Pi)\le d-1$ by \cref{thm:toast_in_tail} and \cref{thm:grid_speedup}.
Recall that $\bf M$ is the supremum of all possible finite moments.
The examples in Section~\ref{sec:Inter} show that there are LCLs with ${\bf M}(\Pi)=\ell$ for every natural number $1\le \ell\le d-1$, but for most natural LCLs we do not know the exact value of ${\bf M}(\Pi)$, unless $d=2$.

We invite the reader to first look at \cref{fig:3-coloring,fig:4-edge-coloring,fig:Tiling} that suggest how concrete problems are usually solved in the class $\toast$.
We usually have a certain nice periodic solution in mind, for example, for 3-coloring, we think of the simple periodic 2-coloring of the grid as the idealized solution we want to build.
Although the solution that we build in one piece of the toast is not necessarily this idealized solution, it looks like the idealized solution on the boundary.
Hence, we can think of the idealized solution as an ``interface'' between toast pieces of different levels.
In other words, we simply have to show how to combine several idealized solutions together in one toast piece so that we get a solution that looks like the idealized one on the outside of that piece. 
For example, in the case of 3-coloring, we need to be able to combine several 2-colorings together so that the outside looks like a 2-coloring. As seen in \cref{fig:3-coloring}, the third color gives us the flexibility to change the ``parity'' of combined 2-colorings, which makes such a $\toast$ algorithm possible.

\subsection{Vertex Coloring}
\label{subsec:3coloring}

Let $d>1$.
Recall that $\Pi^d_{\chi,3}$ is the \emph{proper vertex $3$-coloring problem}.
This problem was studied extensively in all contexts.
It is known that $\Pi^d_{\chi,3}\not\in \LOCAL(O(\log^* n))$ \cite{brandt_grids}, $\Pi^d_{\chi,3}\not\in \CONT$ \cite{GJKS}, $\Pi^d_{\chi,3}\in \BOREL$ \cite{GJKS2} and $0<{\bf M}(\Pi^d_{\chi,3})\le 2$ \cite{HolroydSchrammWilson2017FinitaryColoring}.
A version of the following theorem is proven in \cite{GJKS2}, here we give a simple ``proof by picture''. 

\begin{theorem}[\cite{GJKS2}]\label{thm:3in toast}
Let $d>1$.
The LCL $\Pi^d_{\chi,3}$ is in the class $\toast$.
\end{theorem}

\begin{proof}[Proof Sketch]
The toast algorithm tries to color all nodes in a given piece with two colors, blue and red. However, it can happen that two partial colorings of small pieces that have different parity in a bigger piece have to be connected (see \cref{fig:3-coloring}). 
In this case, we take the small pieces colored with incorrect parity and ``cover'' their border with the third color: green. From the inside of a small piece, the green color corresponds to one of the main two colors, while from the outside it corresponds to the other color. This enables us to switch color parity. 
\end{proof}

As a consequence of results from \cref{sec:Toast} we can now fully determine the moment behavior of the problem which answers a question from \cite{HolroydSchrammWilson2017FinitaryColoring}. 

\begin{corollary}
Let $d>1$.
Every uniform local algorithm that solves $\Pi^d_{\chi,3}$ has the $1$-moment infinite, but there is a uniform local algorithm solving $\Pi^d_{\chi,3}$ with all $(1-\eps)$-moments finite. 
In particular, ${\bf M}(\Pi^d_{\chi,3})=1$. 
\end{corollary}
\begin{proof}
By \cref{thm:3in toast}, $\Pi^d_{\chi,3}$ is in the class $\toast$.
It follows from \cref{thm:toast_in_tail} that ${\bf M}(\Pi^d_{\chi,3})\ge 1$.
By the speedup \cref{thm:grid_speedup} and the fact that $3$-coloring was proven not to be in class $\local(O(\log^* n)) = \olocal(O(\log^* 1/\eps)) = \cont$ \cite{brandt_grids,HolroydSchrammWilson2017FinitaryColoring,GJKS}, we have that every uniform local algorithm $\fA$ solving $\Pi^2_{\chi,3}$ must have the $1$-moment infinite.
In particular, ${\bf M}(\Pi^2_{\chi,3})=1$.
It remains to observe that every vertex $3$-coloring of $\mathbb{Z}^d$ yields a vertex $3$-coloring of $\mathbb{Z}^2$ by forgetting the last $(d-2)$ dimensions.
Consequently, every uniform local algorithm $\fA$ that solves $\Pi^d_{\chi,3}$ defines a uniform local algorithm $\fA'$ that solves $\Pi^2_{\chi,3}$ with the same uniform complexity.
This finishes the proof.
\end{proof}

The $3$-coloring discussion above also implies that the toast itself cannot be constructed faster than with $O(1/\eps)$ uniform complexity, otherwise, $3$-coloring could be constructed with such a decay, too. 
\begin{corollary}
\label{cor:toast_lower_bound}
There is a constant $C > 0$ such that a $q$-toast cannot be constructed with uniform complexity $C/\eps$ on any $\mathbb{Z}^d$.
Here, by constructing a $q$-toast we mean that the coding radius of each node $u$ is such that $u$ can describe the minimal toast piece $D \in \fD$ such that $u \in D$. 
\end{corollary}

\begin{figure}
    \centering
    \includegraphics[width = .95\textwidth]{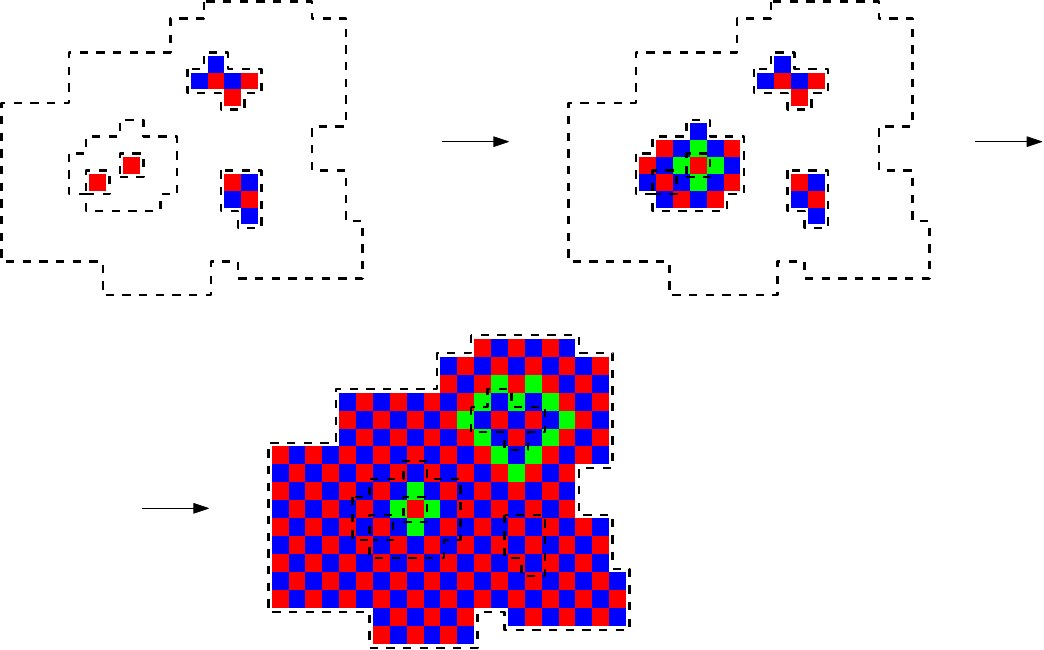}
    \caption{The $\toast$ algorithm is given small pieces of a toast that are already vertex-colored. The first step is to match the parity of colors with respect to the bigger piece. Then color the bigger piece.}
    \label{fig:3-coloring}
\end{figure}

\subsection{Edge Coloring}
\label{subsec:edgecoloring}

\begin{figure}
    \centering
    \includegraphics[width = .95\textwidth]{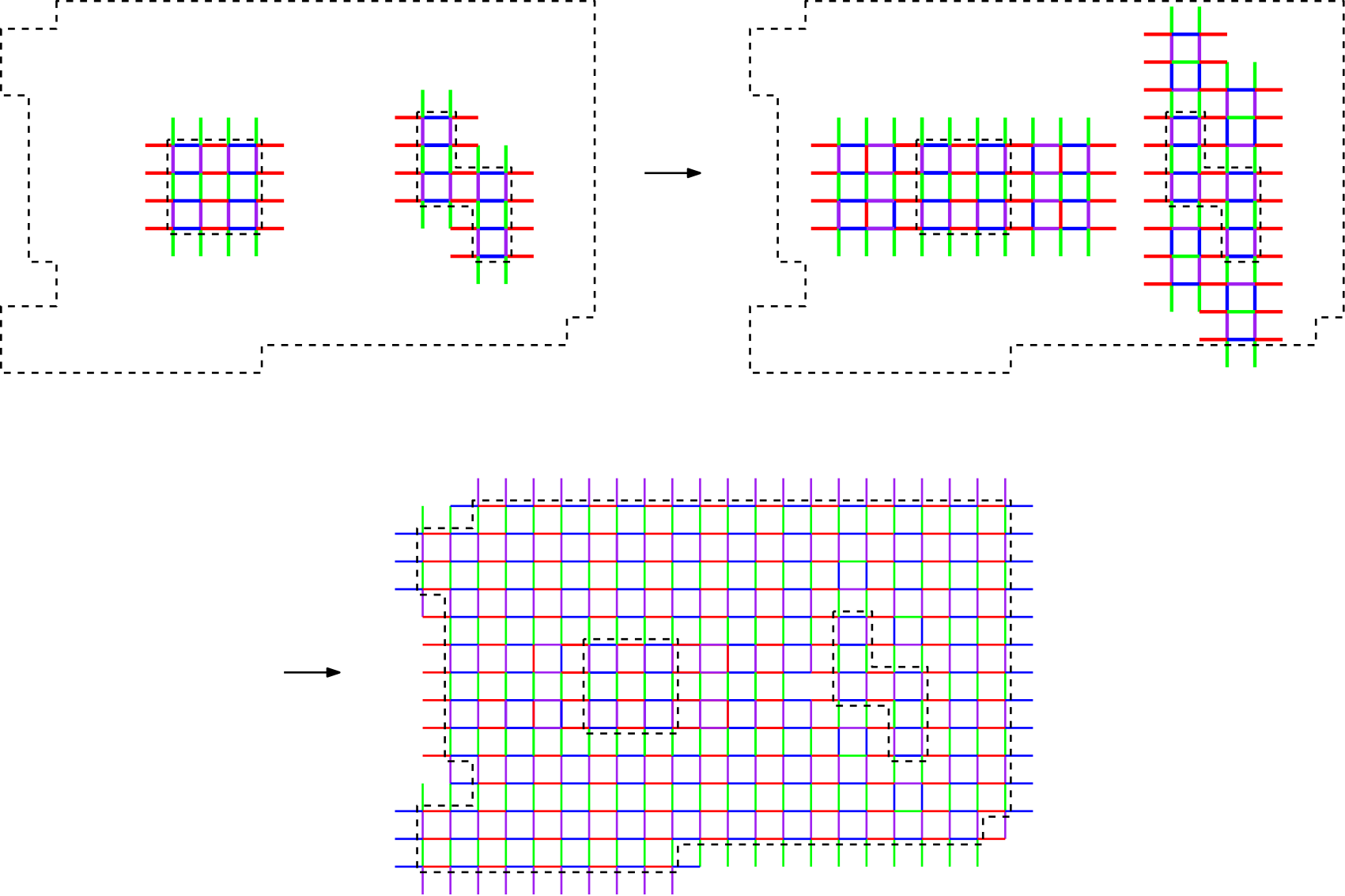}
    \caption{The $\toast$ algorithm is given small pieces of a toast that are already edge-colored. The first step is to match the parity of colors with respect to the bigger piece. This is done by adding a three-layered gadget in the problematic directions. After that, the coloring is extended to the bigger piece.}
    \label{fig:4-edge-coloring}
\end{figure}
Let $d>1$.
We denote as $\Pi^d_{\chi',2d}$ the \emph{proper edge $2d$-coloring problem}.
Observe that a solution to $\Pi^d_{\chi',2d}$ yields a solution to the perfect matching problem $\Pi^d_{pm}$ (because $\mathbb{Z}^d$ is $2d$-regular).
This shows that $\Pi^d_{\chi',2d}$ is not in the class $\LOCAL(O(\log^* n))$, because the perfect matching problem clearly cannot be solved with a local algorithm: if it was, we could find a perfect matching in every large enough finite torus, which is clearly impossible since it can have an odd number of vertices \cite{brandt_grids}.

\begin{theorem}[See also \cite{ArankaFeriLaci} and \cite{Felix}]\label{thm:edge toast}
Let $d>1$.
Then $\Pi^d_{\chi',2d}$ is in the class $\toast$.
\end{theorem}
\begin{proof}[Proof Sketch]
The spontaneous toast algorithm is illustrated for $d=2$ in \cref{fig:4-edge-coloring}.
The ``interface'' used by the toast algorithm between smaller and bigger pieces is: smaller pieces pretend they are colored such that in the vertical direction violet and green edges alternate, while in the horizontal direction blue and red edges alternate. If all small pieces in a bigger one have the same parity, this interface is naturally extended to the bigger piece. Hence, we need to work out how to fix pieces with bad horizontal/vertical parity, similarly to \cref{thm:3in toast}. 

\cref{fig:4-edge-coloring} shows how to fix the vertical parity, adding the highlighted gadget in the middle picture above and below the given pieces keeps its interface the same but switches the parity. The same can be done in the horizontal direction. 

To show the result for $d>2$ it is enough to realize that switching parity in one dimension can be done with a help of one extra additional dimension while fixing the parity in the remaining dimensions.
Consequently, this allows iterating the previous argument at most $ d$ many times to fix all the wrong parities.
\end{proof}

\begin{corollary}
Let $d>1$.
Then $1 \le {\bf M}(\Pi^d_{\chi',2d}) \le d-1$. 
In particular, ${\bf M}(\Pi^2_{\chi',4})=1$.
\end{corollary}

We leave it as an open problem to determine the exact value of ${\bf M}(\Pi^d_{\chi',2d})$, where $d>2$.

The following corollary answers a question of \cite{GJKS}. See \cref{sec:OtherClasses} for the definition of the class $\BOREL$. 
\begin{corollary}
Let $d\ge 1$.
Then $\Pi^d_{\chi',2d}$ is in the class $\BOREL$.
In particular, if $\mathbb{Z}^d$ is a free Borel action and $\mathcal{G}$ is the Borel graph generated by the standard generators of $\mathbb{Z}^d$, then $\chi'_\mathcal{B}(\mathcal{G})=2d$, i.e., the Borel edge chromatic number is equal to $2d$.
\end{corollary}

\begin{remark}[Decomposing grids into finite cycles]
Bencs, Hrušková, Tóth \cite{ArankaFeriLaci} showed recently that there is a $\toast$ algorithm that decomposes an unoriented grid into $d$-many families each consisting of finite pairwise disjoint cycles.
Having this structure immediately gives the existence of edge coloring with $2d$-colors as well as Schreier decoration.
Since every $\toast$ algorithm implies Borel solution on graphs that look locally like unoriented grids \cite{AsymptoticDim}, this result implies that all these problems are in the class $\BOREL$.
We refer the reader to \cite{ArankaFeriLaci} for more details.
\end{remark}

\subsection{Tiling with Rectangles}
\label{subsec:rectangle_tiling}

Let $d>1$ and $\mathcal{R}$ be a finite collection of $d$-dimensional rectangles.
We denote as $\Pi^d_{\mathcal{R}}$ the problem of \emph{tiling $\mathbb{Z}^d$ with elements of $\mathcal{R}$}.
It is not fully understood when the LCL $\Pi^d_{\mathcal{R}}$ is in the class $\LOCAL(O(\log^* n))$ (or equivalently $\CONT$).
Currently the only known condition implying that $\Pi^d_{\mathcal{R}}$ is \emph{not} in $\LOCAL(O(\log^* n))$ is when GCD of the volumes of elements of $\mathcal{R}$ is not $1$.
The minimal unknown example is when $\mathcal{R}=\{2\times 2, 2\times 3,3\times 3\}$. 
For the other complexity classes the situation is different.
Given a collection $\mathcal{R}$ we define a condition \eqref{eq:GCD} as
\begin{equation}
\label{eq:GCD}\tag{$\dagger$}
    \text{GCD of side lengths in the direction of} \ i \text{-th generator is} \ 1 \ \text{for every} \ i\in [d].
\end{equation}
It follows from the same argument as for vertex $2$-coloring, that if $\mathcal{R}$ does not satisfy \eqref{eq:GCD}, then $\Pi^d_{\mathcal{R}}$ is a global problem, i.e., it is not in any  class from \cref{fig:big_picture_grids}.
On the other hand, the condition \eqref{eq:GCD} is sufficient for a spontaneous toast algorithm. For example, this implies $\Pi^2_{\{2\times 3,3\times 2\}}\in \BOREL \setminus \CONT$.
A version of the following theorem was announced in \cite{spencer_personal}, here we give a simple ``proof by picture''.

\begin{theorem}[\cite{spencer_personal, stephen}]
\label{thm:tilings}
Let $d>1$ and $\mathcal{R}$ be a finite collection of $d$-dimensional rectangles that satisfy \cref{eq:GCD}.
Then $\Pi^d_{\mathcal{R}}$ is in the class $\toast$.
Consequently, ${\bf M}(\Pi^d_{\mathcal{R}})\ge 1$.
\end{theorem}

\begin{figure}
    \centering
    \includegraphics[width=.9\textwidth]{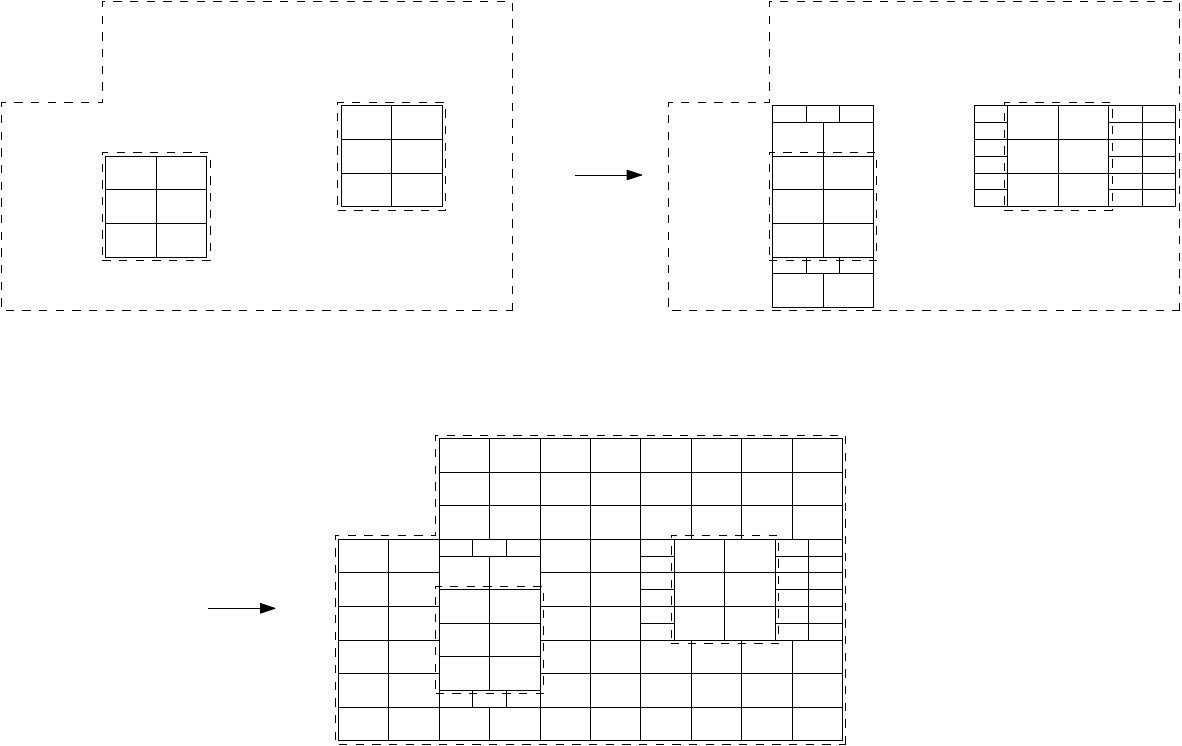}
    \caption{The $\toast$ algorithm is given small pieces of a toast that are already tiled. The first step is to match the parity of the tiling with respect to the bigger piece. Then tile the bigger piece.}
    \label{fig:Tiling}
\end{figure}

\begin{proof}[Proof Sketch]
The two-dimensional version of the construction is shown in \cref{fig:Tiling}. The interface of pieces this time is that they cover themselves with boxes of side length $L$ which is the least common multiple of all the lengths of basic boxes. As in the previous constructions, we just need to work out how to ``shift'' each piece so that all small pieces agree on the same tiling with boxes of side length $L$ which can then be extended to the whole piece. This can be done as in \cref{fig:Tiling}. 
\end{proof}

\begin{corollary}
Let $d=2$ and $\mathcal{R}$ be a finite collection of $2$-dimensional rectangles.
Then $\Pi^2_{\mathcal{R}}$ is either in the class $\LOCAL(O(\log^* n))$, satisfies ${\bf M}(\Pi^2_{\mathcal{R}})=1$, or is global (that is, in no class in \cref{fig:big_picture_grids}).
\end{corollary}

For any $d$, \cref{thm:tilings} implies that if a tiling problem is ``hard'' (in no class in \cref{fig:big_picture_grids}), it is because of a simple reason: a projection to one dimension gives a periodic problem. It would be interesting to understand whether there are problems that are hard but their solutions do not have any periodicity in them (see \cref{sec:open_problems}). 

We do not know if there is $\mathcal{R}$ that would satisfy $1< {\bf M}(\Pi^d_{\mathcal{R}})\le d-1$.
Let $d\ge \ell>0$ and denote as $\mathcal{R}_\ell$ the restriction of rectangles to first $\ell$-dimensions.
As in the case of proper $3$-vertex coloring problem, we have that ${\bf M}(\Pi^d_{\mathcal{R}})\le {\bf M}(\Pi^\ell_{\mathcal{R}_{\ell}})$.
In particular, if, e.g., $\Pi^2_{\mathcal{R}_2}$ is not in the class $\LOCAL(O(\log^* n))$, then ${\bf M}(\Pi^d_{\mathcal{R}})\le 1$.

\begin{remark}
In the context of non-oriented grids, we need to assume that $\mathcal{R}$ is invariant under swapping coordinates. 
Similar argument, then shows that $\Pi^d_\mathcal{R}\in \toast$ if and only if $\mathcal{R}$ satisfies \eqref{eq:GCD}.
Combined with the fact that there is a Borel toast on such Borel graphs (see~\cite{AsymptoticDim}), we deduce that \eqref{eq:GCD} for $\mathcal{R}$ is equivalent to the existence of Borel tiling with rectangles from $\mathcal{R}$.
\end{remark}

\begin{remark}
It follows easily from \eqref{eq:GCD} that if $|\mathcal{R}|=1$, then $\Pi^d_\mathcal{R}$ is global unless $\mathcal{R}=\{1\times 1\times \dots \times 1\}$.
More generally, the authors together with Greenfeld and Tao \cite{MeasurableTiling} showed that the same is true if we consider a tiling problem in $\mathbb{Z}^d$ with one tile that is not necessarily rectangle or connected.
Combination of this result together with the recent counterexample to the Periodic Tiling Conjecture that was found by Greenfeld and Tao \cite{FailurePTC} implies the existence of a solvable LCL problem (for $d$ large enough) that is global (it is not in any  class from \cref{fig:big_picture_grids}), yet all of its solutions are not periodic.
In another words, the reason why this LCL problem is global is different from, say, $2$-vertex coloring or (global) rectangular tilings.
\end{remark}


\paragraph{Perfect matching}

Let $d>1$.
A special case of a tiling problem is the \emph{perfect matching problem} $\Pi^d_{pm}$.
Since there are arbitrarily large tori of odd size we conclude that $\Pi^d_{pm}\not \in \LOCAL(O(\log^* n))$.
Consequently, $\Pi^d_{pm}\not \in \olocal(O(\log^* 1/\eps))$ by \cref{thm:tail=uniform}, and it follows from the Twelve Tile Theorem of Gao, Jackson, Krohne and Seward \cite[Theorem 5.5]{GJKS} that $\Pi^d_{pm}\not \in \CONT$.
Similarly, as with the vertex $3$-coloring problem, the proof of $\Pi^d_{pm}\in \BOREL$, where $d>1$, from \cite{GJKS2} uses the toast construction and, in our notation, yields the following result. 

\begin{theorem}[\cite{GJKS2}]\label{thm:perfect_matching_intoast}
Let $d>1$.
The LCL $\Pi^d_{pm}$ is in the class $\toast$.
\end{theorem}

\begin{corollary}
Let $d>1$.
Then ${\bf M}(\Pi^d_{pm})\ge 1$ and every ffiid $F$ that solves $\Pi^d_{pm}$ has the $(d-1)$-moment infinite.
In particular, ${\bf M}(\Pi^2_{pm})=1$.
\end{corollary}

We leave it as an open problem to determine the exact value of ${\bf M}(\Pi^d_{pm})$, where $d>2$.
In fact, this problem is similar to intermediate problems that we study in \cref{sec:Inter}.

\renewcommand{\pm}[2]{\operatorname{PM}^{#1, #2}_\boxtimes}
\renewcommand{\pm}[1]{\operatorname{PM}^{#1}_\boxtimes}
\newcommand{\Zbox}{\mathbb{Z}^d_\boxtimes}

\section{Intermediate Problems}\label{sec:Inter}

In this section, we construct local problems with uniform local complexities $\sqrt[j-o(1)]{1/\eps}$ for $1 \le j \le d-1$.
In particular, these problems will satisfy ${\bf M}(\Pi)=j$. 
We first define the problem $\pm{d}$: this is just a perfect matching problem but in the graph $\Zbox$ where two different nodes $u = (u_1, \dots, u_d)$ and $v = (v_1, \dots, v_d)$ are connected with an edge if $\sum_{1 \le i \le d} |u_i - v_i| \le 2$. That is, we also add some ``diagonal edges'' to the grid. We do not know whether perfect matching on the original graph (Cayley graph of $\mathbb{Z}^d$) is also an intermediate problem (that is, not $\local$ but solvable faster than $\toast$). 

We will show that the uniform local complexity of $\pm{d}$ is $\sqrt[d - 1 -o(1)]{1/\eps}$. 
After this is shown, the hierarchy of intermediate problems is constructed as follows. 
Consider the problem $\pm{d,\ell}$. This is a problem on $\mathbb{Z}^d$ where one is asked to solve $\pm{d}$ in every hyperplane $\mathbb{Z}^\ell$ given by the first $\ell$ coordinates in the input graph. 
The $\olocal$ complexity of $\pm{d,\ell}$ is clearly the same as the complexity of $\pm{\ell}$, that is, $\sqrt[\ell - 1 -o(1)]{1/\eps}$. 
Hence, we can focus just on the problem $\pm{d}$ from now on. First, we claim there is no nontrivial local algorithm for this problem. 

\begin{claim}\label{cl:PM is global}
Let $2\le d\in \mathbb{N}$.
Then the local coloring problem $\pm{d}$ is not in the class $\LOCAL(O(\log^* n))$.
In particular, by \cref{thm:grid_speedup}, $\pm{d} \not\in \olocal(o(\sqrt[d-1]{1/\eps}))$. 
\end{claim}

This follows from the fact that if there is a local algorithm, then in particular there exists a solution on every large enough torus. But this is not the case if the torus has an odd number of vertices. 

Second, we need to show how to solve the problem $\pm{d}$ faster than with a $\toast$ construction algorithm. 

\begin{restatable}{theorem}{intermediate}
\label{thm:intermediate}
The problem $\pm{d}$ can be solved with a uniform local algorithm of uniform complexity $\sqrt[d-1]{1/\eps} \cdot 2^{O(\sqrt{\log 1/\eps})}$. 
\end{restatable}

Here we give only a proof sketch, the formal proof is deferred to \cref{app:intermediate_problems}. 

\begin{proof}[Proof Sketch]
We only discuss the case $d=3$ (and the actual proof differs a bit from this exposition due to technical reasons). As usual, start by tiling the grid with boxes of side length roughly $r_0$. 
The \emph{skeleton} $\fC_0$ is the set of nodes we get as the union of the $12$ edges of each box. 
Its \emph{corner points} $X_0 \subseteq \fC_0$ is the union of the $8$ corners of each box.

We start by matching vertices inside each box that are not in the skeleton $\fC_0$. It is easy to see that we can match all inside vertices, maybe except for one node per box that we call a \emph{pivot}. We can ensure that any pivot is a neighbor of a vertex from $\fC_0$. 

We then proceed similarly to the toast construction of \cref{thm:toast_in_tail} and build a sequence of sparser and sparser independent sets that induce sparser and sparser subskeletons of the original skeleton.
More concretely, we choose a sequence $r_k = 2^{O(k^2)}$ and in the $k$-th step we construct a new set of corner points $X_k \subseteq X_{k-1}$ such that any two corner points are at least $r_k$ hops apart but any point of $\mathbb{Z}^d$ has a corner point at distance at most $2r_k$ (in the metric induced by $\mathbb{Z}^d$). 
This set of corner points $X_k$ induces a subskeleton $\fC_{k} \subseteq \fC_{k-1}$ by connecting nearby points of $X_k$ with shortest paths between them in $\fC_{k-1}$. 

After constructing $\fC_k$, the algorithm solves the matching problem on all nodes of $\fC_{k-1} \setminus \fC_k$. This is done by orienting those nodes towards the closest $\fC_k$ node in the metric induced by the subgraph $\fC_k$. 
Then, the matching is found on connected components that have a radius bounded by $r_k$. There is perhaps only one point in each component, a pivot, that neighbors with $\fC_{k}$ and remains unmatched. 

After the $k$-th step of this algorithm, only nodes in $\fC_k$ (and a few pivots) remain unmatched. Hence, it remains to argue that only roughly a $1/r_k^{d-1}$ fraction of nodes survives to $\fC_k$. This should be very intuitive -- think about the skeleton $\fC_k$ as roughly looking as the skeleton generated by boxes of radius $r_k$. This is in fact imprecise, because in each step the distance between corner points of $X_k$ in $\fC_k$ may be stretched by a constant factor with respect to their distance in the original skeleton $\fC_{k-1}$. This is the reason why the final volume of $\fC_k$ is bounded only by $2^{O(k)} / r_k^{d-1}$ and by a similar computation as in \cref{thm:toast_in_tail} we compute that the uniform local complexity is $\sqrt[d-1]{1/\eps} \cdot 2^{O(\sqrt{\log 1/\eps})}$. 
\end{proof}

\section{Other Classes of Problems}\label{sec:OtherClasses}

In this mostly expository section, we discuss other classes of problems from \cref{fig:big_picture_grids} and how they relate to our results from previous sections.

\subsection{$\CONT$}\label{subsec:CONT}

Here we explain the fact that an LCL $\Pi$ on $\mathbb{Z}^d$ is in the class $\local(O(\log^*(n)))$ if and only if it is in the class $\cont$.
For $\mathbb{Z}^d$, the model $\cont$ was extensively studied by Gao, Jackson, Krohne, Tucker-Drob, Seward, and others, see \cite{GaoJackson,GJKS,ST,ColorinSeward}.
In fact, the connection with local algorithms is implicit in \cite{GJKS} under the name Twelve Tiles Theorem \cite[Theorem~5.5]{GJKS}.
Recently, independently Bernshteyn \cite{Bernshteyn2021local=cont} and Seward \cite{Seward_personal} showed that the same holds in bigger generality for all Cayley graphs of finitely generated groups.
Our aim here is to formally introduce the model $\CONT$ and give an informal overview of the argument for $\mathbb{Z}^d$ that follows the lines of Gao, Jackson, Krohne, and Seward \cite{GJKS}.
The reasons not to simply put a reference to the aforementioned papers are two.
First, we work with slightly more general LCLs that have input.
Second, our goal is to introduce all classes in a way accessible to a broader audience. 

The setup of the $\CONT$ model is somewhat similar to the deterministic $\local$ model: there is no randomness in the definition and there is a certain symmetry-breaking input on the nodes. Specifically, each vertex is labeled with one label from the set $\{0, 1\}$; the space of all such labelings is denoted as $2^{\mathbb{Z}^d}$. 
However, for the same reason as why we insist on the uniqueness of the identifiers in the $\local$ model, we insist that the input $\{0,1\}$ labeling breaks all symmetries. Namely, we require that there is no $g\in \mathbb{Z}^d$ such that shifting all elements by $g$ preserves the labeling of $\mathbb{Z}^d$. We write $\operatorname{Free}(2^{\mathbb{Z}^d}) \subseteq 2^{\mathbb{Z}^d}$ for this subspace. 
A \emph{continuous algorithm} is a pair $\mathcal{A}=(\fC,f)$, where $\fC$ is a collection of $\{0,1\} \times \Sigma_{in}$-labeled neighborhoods and $f$ is a function that assigns to each element of $\fC$ an output label. It needs to satisfy the following. Suppose we fix a node $v$, a $\{0,1\}$-labeling $\Lambda$ of $\mathbb{Z}^d$ that satisfies $\Lambda \in \operatorname{Free}(2^{\mathbb{Z}^d})$ as well as an input $\Sigma_{in}$-labeling of $\mathbb{Z}^d$.
If we let $v$ explore its neighborhood, then it encounters an element from $\fC$ after finitely many steps. 
We write $\mathcal{A}(\Lambda)(v)$, or simply $\mathcal{A}(v)$, when $\Lambda$ is understood, for the corresponding output.

\begin{remark}
This is equivalent to the definition given in e.g. \cite{GJKS} in the same way as the correspondence between uniform local algorithms and ffiid, see \cref{subsec:subshifts}.
The fact that the collection $\fC$ is dense, i.e., every node finishes exploring its neighborhoods in finite time, means that $\mathcal{A}({\bf 0})$ is \emph{continuous} with respect to the product topology and defined on the whole space $\operatorname{Free}(2^{\mathbb{Z}^d})$.
Moreover, the value computed at $v$ does not depend on the position of $v\in \mathbb{Z}^d$ but only on the isomorphism type of the labels, this is the same as saying that $\mathcal{A}(\_)$ is \emph{equivariant}.
\end{remark}

We say that an LCL $\Pi$ is in the class $\cont$ if there is a continuous algorithm $\fA$ that outputs a $\Pi$-coloring for every $\Lambda \in \operatorname{Free}(2^{\mathbb{Z}^d})$.
Note that, in general, we cannot hope for the diameters of the neighborhoods in $\fC$ to be uniformly bounded.
In fact, if that happens then $\Pi \in \local(O(1))$. 
This is because there are elements in $\operatorname{Free}(2^{\mathbb{Z}^d})$ with arbitrarily large regions with constant $0$ label.
Applying $\mathcal{A}$ on that regions gives an output label that depends only on $\Sigma_{in}$ after a bounded number of steps, hence we get a local algorithm with constant local complexity.

The main idea to connect this context with distributed computing is to search for a labeling in $\operatorname{Free}(2^{\mathbb{Z}^d})$ where any $\fA$ necessarily has a uniformly bounded diameter.
Such labeling does exist!
In \cite{GJKS} these labelings are called \emph{hyperaperiodic}.
To show that hyperaperiodic labelings exist on $\mathbb{Z}^d$ is not that difficult, unlike in the general case for all countable groups \cite{ColorinSeward}. 
In fact, there is enough flexibility to encode in a given hyperaperiodic labeling some additional structure in such a way that the new labeling remains hyperaperiodic.

We now give a definition of a hyperaperiodic labeling. 
Let $\Lambda \in 2^{\mathbb{Z}^d}$ be a labeling of $\mathbb{Z}^d$.
We say that $\Lambda$ is \emph{hyperaperiodic} if for every $g\in \mathbb{Z}^d$ there is a finite set $S\subseteq \mathbb{Z}^d$ such that 
for every $t\in \mathbb{Z}^d$ we have 
$$\Lambda(t+S) \not= \Lambda(g+t+S)$$
This notation means that the $\Lambda$-labeling of $S$ differs from the $\Lambda$-labeling of $S$ shifted by $g$ everywhere. 
In other words, given some translation $g\in \mathbb{Z}^d$, there is a uniform witness $S$ for the fact that $\Lambda$ is nowhere $g$-periodic.
It follows from a standard compactness argument that any continuous algorithm $\fA$ used on a hyperaperiodic labeling $\Lambda$ finishes after a  number of steps that is uniformly bounded by some constant that depends only on $\fA$ and $\Lambda$.

\begin{theorem}[\cite{GJKS,Bernshteyn2021local=cont}]
Let $d>0$.
Then an LCL $\Pi$ on $\mathbb{Z}^d$ is in the class $\CONT$ if and only if it is in the class $\local(O(\log^*(n)))$.
\end{theorem}

\begin{figure}
    \centering
    \includegraphics[width=.9\textwidth]{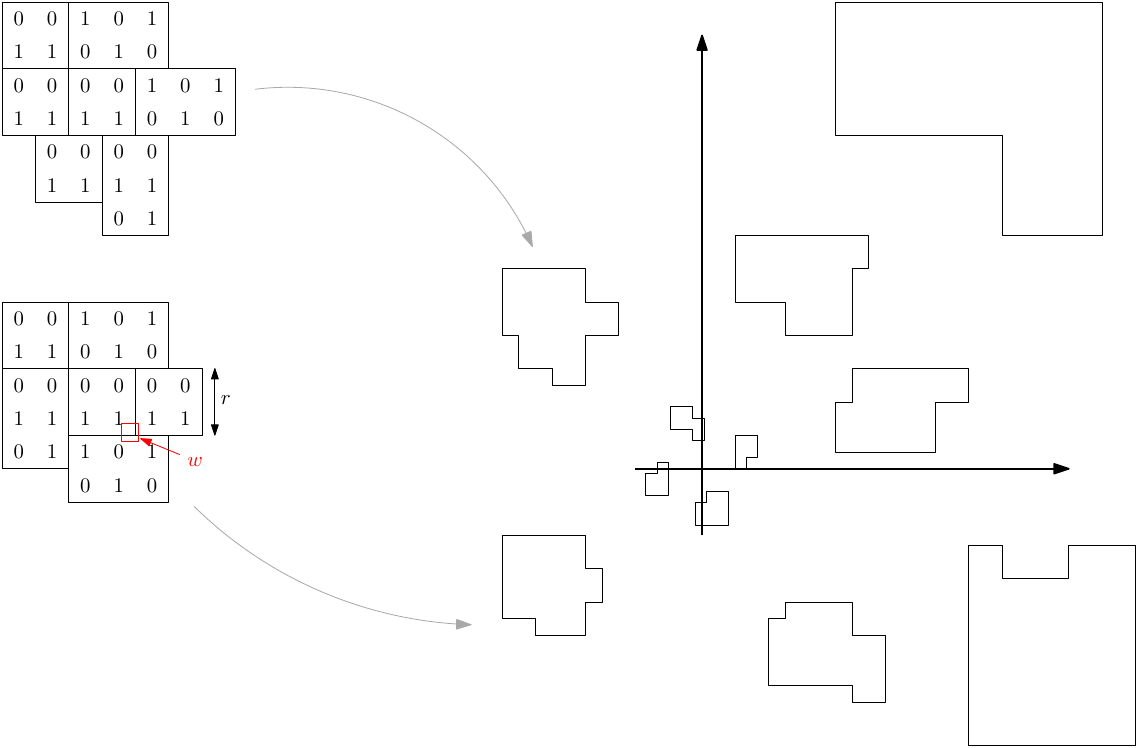}
    \caption{On the left side we have examples of labeled local tiling configurations compared with a possible continuity window $w$. These are then sparsely embedded into $\mathbb{Z}^2$.}
    \label{fig:my_label}
\end{figure}

\begin{proof}[Proof Sketch]
One direction was observed by Bernshteyn \cite{Bernshteyn2021LLL} and follows from the fact that there is a continuous algorithm that produces a maximal $r$-independent set for every $r\in \mathbb{N}$ (recall \cref{thm:classification_of_local_problems_on_grids}).
This is implicit, even though in a different language, already in Gao and Jackson \cite{GaoJackson}.

Next, suppose that $\Pi$ is in the class $\CONT$ and write $\fA$ for the continuous algorithm for $\Pi$. 
We will carefully construct a certain hyperaperiodic element $\Lambda$ and let $\fA$ solve $\Pi$ on it. 
We know that there is a certain number $w$ such that $\fA$ solves $\Pi$ on $\Lambda$ with local complexity $w$. 
We note that even though we do not know the value $w$ in advance, it depends on $\Lambda$ that is due to be constructed, the flexibility of the construction guarantees to ``encode'' in $\Lambda$ all possible scenarios.
See the discussion at the end of the argument. 

The overall approach is as follows. 
Our local algorithm $\fA'$ first tiles the plane with  tiles of length much bigger than $w$ with $O(\log^* n)$ local complexity. We cannot tile the plane in a periodic manner in $O(\log^* n)$ steps, but it can be done with a constant number of different tiles such that the side length of each tile is either $r$ or $r+1$ for some $r \gg w$ (this construction is from \cite[Theorem~3.1]{GaoJackson}, see also the proof of \cref{thm:grid_speedup}). 
For each tile $T$, $\fA'$ labels it with certain $\{0,1\}$ labeling. 
This labeling depends only on the shape of the tile.
We make sure that not only $T$ with its labeling, but also $T$ together with all possible local configurations of tiles around it and corresponding labelings appears in $\Lambda$.
In particular, as $r\gg t,w$ ($t$ is the radius necessary to check the solution $\Pi$), we have that the $(w+t)$-hop neighborhood of any node is labeled in a way that also appears somewhere in the labeling $\Lambda$. 
After labeling each tile, $\fA'$ simulates $\fA$. We know that $\fA$ finishes after $w$ steps on $\Lambda$ and is correct there, hence $\fA'$ is correct and has local complexity $O(\log^* n)$. What remains is to describe is how to get the labeling of each tile.  

We use the fact that there is a lot of flexibility when constructing a hyperaperiodic labeling. 
We define the hyperaperiodic labeling $\Lambda$ of the input graph $\mathbb{Z}^d$ in such a way that, after fixing $r$, each one of constantly many local configurations of $\{r,r+1\}$-tiles are embedded in $\Lambda$, each one in fact in many copies that exhaust possible $\Sigma_{in}$-labelings.

This is done as follows. 
Fix any hyperaperiodic element $\Lambda'$, these exist by \cite[Lemma~2.8]{GJKS}.
Given a tile $T$, that depends on $r$, $\{0,1\}$-label it by any $T$-pattern that appears in $\Lambda'$.
Consider the collection of all possible local configurations of such labeled tiles that are moreover labeled by $\Sigma_{in}$.
This collection is finite. Since we do not know $w$ in advance, do this for any value of $r$. This yields a countable collection of possible configurations. 
Embed this collection sparsely into $\Lambda'$, \cref{fig:my_label}. Extend the $\Sigma_{in}$-labeling arbitrarily. 
This defines $\Lambda$.
It is routine to check that $\Lambda$ is hyperaperiodic, see \cite[Lemma~5.9]{GJKS}.



It remains to recall that since $\Lambda$ is hyperaperiodic, there is $w\in \mathbb{N}$ such that $\fA$ has uniformly bounded locality by $w$ on $\Lambda$.
Pick $r\gg w,t$.
This works as required since any $(w+t)$-hop neighborhood $\fA'$ can encounter is present in $\Lambda$.
\end{proof}

\subsection{Descriptive Combinatorics Classes}
\label{subsec:descriptive_combinatorics_discussion}

In this section we define the classes of problems that are studied in descriptive combinatorics.
There are two equivalent ways how the notion of a \emph{Borel graph} can be defined.
The less abstract one is as follows.
Consider the unit interval $[0,1]$ with the $\sigma$-algebra of Borel sets, that is the smallest algebra of subsets that contains sub-intervals and is closed under countable unions and complements.
Let $A,B\subseteq [0,1]$ be Borel sets and consider a bijection $\varphi:A \to B$ between them.
We say that $\varphi$ is \emph{Borel measurable} if $\varphi^{-1}(D\cap B)$ is a Borel set for every Borel set $D$.
Now, let $\fG$ be a graph with vertex set $[0,1]$ and assume that the degree of every vertex is bounded by some $\Delta<\infty$.
We say that $\fG$ is a \emph{Borel graph} if the edge relation is induced by Borel bijections, that is there are sets $A_1,\dots, A_\ell$, $B_1,\dots, B_\ell$, and Borel bijections $\varphi_1,\dots, \varphi_\ell$ such that $(x,y)$ is an edge in $\fG$ if and only if there is $1\le i\le \ell$ such that either $\varphi_i(x)=y$ or $\varphi_i(y)=x$.

The more abstract way how to define Borel graphs uses the notion of a \emph{standard Borel space}.
For example, $[0,1]$ with the $\sigma$-algebra of Borel sets, when we forget that it comes from the standard topology, is a standard Borel space.
Formally, a standard Borel space is a pair $(X,\fB)$, where $X$ is a set and $\fB$ is a $\sigma$-algebra that comes from some separable and complete metric structure on $X$.
A uniformly bounded degree graph $\fG$ on $(X,\fB)$ is Borel if it satisfies the same condition as in the case of the particular space $[0,1]$.

It is easy to see that with this definition all the basic constructions from graph theory are \emph{relatively} Borel, e.g. if $A$ is a Borel set of vertices, then $N_\fG(A)$ the neighbors of elements in $A$ form a Borel set.
Also note that since we assume that the degree is uniformly bounded, every connected component in $\fG$ is at most countable.
Typical questions that are studied in this context are Borel versions of standard notions, e.g., \emph{Borel} chromatic number, \emph{Borel} chromatic index, \emph{Borel} perfect matching, etc.
We refer the reader to \cite{pikhurko2021descriptive_comb_survey,kechris_marks2016descriptive_comb_survey} for more details.

The most relevant Borel graphs for us are the ones that are induced by a free Borel action of $\mathbb{Z}^d$.
That is, there are $d$-many aperiodic Borel bijections $\varphi_i:X\to X$ that pairwise commute, where \emph{aperiodic} means that $\varphi_i^k(x)\not=x$ for every $k>0$, $i\in [d]$ and $x\in X$.
Then the Borel (labeled) graph $\fG$ that is induced by this action consists of edges of the form $(x,\varphi_i(x))$, where $x\in X$ and $i\in [d]$.
The most important result in this context is that $\fG$ admits a \emph{Borel toast structure}.
This follows from \cite{GJKS2}, see Marks and Unger \cite{Circle} for a proof.
We note that it is not understood in general what Borel graphs admit a Borel toast structure.
This is connected with the notoriously difficult question about Borel hyperfiniteness, see \cite{AsymptoticDim}.

Given an LCL $\Pi$, we say that $\Pi$ is in the class $\borel$ if there is a Borel measurable $\Pi$-coloring of every Borel graph $\fG$ as above, i.e., induced by a free Borel action of $\mathbb{Z}^d$.
Equivalently, it is enough to decide whether $\Pi$ can be solved on the canonical Borel graph $\fG$ on $\operatorname{Free}(2^{\mathbb{Z}^d})$.
Moreover, if $\Pi$ has some inputs $\Sigma_{in}$, then we consider any Borel labeling $X\to \Sigma_{in}$ and then ask if there is a Borel solution that satisfies all the constraints on $\Sigma_{in}\times \Sigma_{out}$.

\begin{theorem}[Basically \cite{GJKS}]
Let $\Pi$ be an LCL that is in the class $\ltoast$.
Then $\Pi\in\borel$.
\end{theorem}

In fact, all the LCL problems without input that we are aware of and that are in the class $\borel$ are in fact in the class $\toast$.
It is an open problem to decide whether it is always the case.

Another interesting perspective on how to see Borel constructions is as a countably infinite analog of $\local(O(\log^*(n)))$.
Namely, suppose that you are allowed to inductively (a) construct Borel MIS with some parameter \cite{KST} (b) do local $O(1)$-construction that depends on the previously built structure.
It is clear that objects that are built in this way are Borel, however, it is not clear if this uses the whole potential of Borel sets to find Borel solutions of LCLs.

As a last remark we note that traditionally in descriptive combinatorics, people consider relaxations of Borel constructions.
This is usually done by ignoring how a solution behaves on some negligible Borel set.
Prominent examples are \emph{measure $0$ sets} for some Borel probability measure on $(X,\fB)$ or \emph{meager sets} for some compatible complete metric on $(X,\fB)$.
We denote the classes of problems that admit such a solution for \emph{every} Borel probability measure or compatible complete metric as $\MEASURE$ and $\BAIRE$, respectively.
Without going into any details we mention that on general graphs these classes are much richer than $\BOREL$, see e.g. \cite{DetMarks,BrooksMeas}.
However, it is still plausible that on $\mathbb{Z}^d$ we have $\borel=\BAIRE=\MEASURE$.
It is well-known that this is true on $\mathbb{Z}$ \cite{grebik_rozhon2021LCL_on_paths}.

\subsection{Finitely Dependent Processes}

In this section we use the notation that was developed in \cref{subsec:subshifts}, also we restrict our attention to LCLs without input.
Let $\Pi=(b,t,\fP)$ be such an LCL.
Recall that $X_\Pi$ is the space of all $\Pi$-colorings.

We say that a distribution $\mu$ on $b^{\mathbb{Z}^d}$ is \emph{finitely dependent} if it is shift-invariant and there is $k\in \mathbb{N}$ such that the events $\mu\upharpoonright A$ and $\mu \upharpoonright B$ are independent for every finite set $A,B\subseteq \mathbb{Z}^d$ such that $d(A,B)\ge k$.
Note that shift-invariance is the same as saying that the distribution $\mu\upharpoonright A$ does not depend on the position of $A$ in $\mathbb{Z}^d$ but rather on its shape.

Here we study the following situation:
Suppose that there is a finitely dependent process $\mu$ that is concentrated on $X_\Pi$.
In that case, we say that $\Pi$ is in the class $\FINdep$. 
In other words, $\Pi \in \FINdep$ if there is a distribution from which we can sample solutions of $\Pi$ such that if we restrict the distribution to any two sets that are sufficiently far away from each other, the corresponding restrictions of the solution of $\Pi$ that we sample are independent. 

It is easy to see that all problems in $\local(O(1))$ (that is, problems solvable by a block factor) are finitely dependent, but examples of finitely dependent processes outside of this class were rare and unnatural, see the discussion in \cite{FinDepCol}, until Holroyd and Liggett \cite{FinDepCol} provided an elegant construction for 3- and 4-coloring of a path and, more generally, for distance coloring of a grid with constantly many colors. This implies the following result.

\begin{theorem}[\cite{FinDepCol}]
$\local(O(\log^* n)) \subseteq \FINdep$ (see \cref{fig:big_picture_grids}).
\end{theorem} 

On the other hand, recently Spinka \cite{Spinka} proved that any finitely dependent process can be described as a finitary factor of iid process.\footnote{His result works in a bigger generality for finitely generated amenable groups.}
Using \cref{thm:toast_in_tail}, we sketch how this can be done on grids, we keep the language of LCLs.

\begin{theorem}[\cite{Spinka}]
Let $\Pi$ be an LCL on $\mathbb{Z}^d$ and $\mu$ be finitely dependent process on $X_\Pi$, i.e., $\Pi\in \FINdep$.
Then there is a finitary factor of iid $F$ (equivalently a uniform local algorithm $\fA$) with uniform local complexity $(1/\eps)^{1+o(1)}$ such that the distribution induced by $F$ is equal to $\mu$.
In fact, $\Pi$ is in the class $\toast$. 
\end{theorem}
\begin{proof}[Proof Sketch]
Suppose that every vertex is assigned uniformly two reals from $[0,1]$.
Use the first string to produce a $k$-toast as in \cref{thm:toast moment}, where $k$ is from the definition of $\mu$.

Having this use the second random string to sample from $\mu$ inductively along the constructed $k$-toast.
It is easy to see that this produces ffiid $F$ that induces back the distribution $\mu$.
In fact, it could be stated separately that the second part describes a randomized $\toast$ algorithm: in each stage of the algorithm sample from $\mu$ given the coloring on smaller pieces that were already fixed. The finitely dependent condition is exactly saying that the joint distribution over the solution on the small pieces of $\mu$ is the same as the distribution we sampled from.
Now use \cref{thm:derandomOfToast} to show that $\Pi\in \toast$. 
\end{proof}

In our formalism, $\FINdep \subseteq \toast$, or equivalently we get $\olocal(O(\log^* 1/\eps)) \subseteq \FINdep \subseteq \olocal((1/\eps)^{1+o(1)})$.
It might be tempting to think that perhaps $\FINdep =\toast$.
It is however a result of Holroyd, Schramm, and Wilson \cite[Corollary~25]{HolroydSchrammWilson2017FinitaryColoring} that proper vertex $3$-coloring, $\Pi^{d}_{3,\chi}$, is not in $\FINdep$ for every $d>1$.
The exact relationship between these classes remains unknown.

\subsection{Uniform Complexities in Distributed and Centralised Models}
\label{subsec:congestLCA}

Although we are not aware of systematic work on uniform local complexities in the distributed algorithms community, the notion of uniform algorithms (algorithms that do not know $n$) comes from  \cite{Korman_Sereni_Viennot2012Pruning_algorithms_+_oblivious_coloring}. Similar notions such as the node-average complexity \cite{feuilloley2017node_avg_complexity_of_col_is_const,barenboim2018avg_complexity,chatterjee2020sleeping_MIS} or energy complexity \cite{chang2018energy_complexity_exp_separation,chang2020energy_complexity_bfs} are also studied. 
We also believe that the breakthrough MIS algorithm of Ghaffari \cite{ghaffari2016MIS} may have found many applications exactly because of the fact that its basis is a uniform algorithm and its uniform complexity can be computed. 

A useful property of algorithms with nontrivial uniform complexities is that if their first moment is bounded, their average complexity is constant. 
For example, with a little more work, we believe that the result from \cref{thm:3in toast}, that is $3$-coloring of grids has uniform complexity $(1/\eps) \cdot 2^{O(\sqrt{\log 1/\eps})}$, can be adapted to the setting of finite grids where it would yield a $3$-coloring algorithm with average complexity $2^{O(\sqrt{\log n})}$. 
The average complexity is studied also for nonconstant $\Delta$.  In \cite{barenboim2018avg_complexity} it was shown that Luby's $\Delta+1$ coloring algorithm has uniform complexity $O(\log 1/\eps)$ even for nonconstant $\Delta$, but the question of whether MIS can have constant average coding radius in nonconstant degree graphs is open \cite{chatterjee2020sleeping_MIS}. 

Uniform bounds make sense also in the related $\lca$ model \cite{lca11,lca12,RosenbaumSuomela2020volume_model}: there one has a huge graph and instead of solving some problem, e.g., coloring, on it directly, we are only required to answer queries about the color of some given nodes. It is known that, as in the $\local$ model, one can solve $\Delta+1$ coloring with $O(\log^* n)$ queries on constant degree graphs by a variant of Linial's algorithm \cite{even03_conflictfree}. But using the uniform version of Linial's algorithm instead, we automatically get that $O(\log^* n)$ is only a worst-case complexity guarantee and on average we expect only $O(1)$ queries until we know the color of a given node. That is, the amortized complexity of a query is in fact constant. The same holds e.g. for the $O(\log n)$ query algorithm for LLL from \cite{brandt_grunau_rozhon2021LLL_in_LCA}.

\section{Open Questions}
\label{sec:open_problems}

Here we collect some open problems that we find interesting. 

\begin{enumerate}
    \item Is it true that $\lrtoast = \ltoast$? If yes, this implies $\ffiid \subseteq \borel$. 
    \item Is it true that $\lrtoast = \rtoast$? If yes, this implies $\ffiid = \olocal((1/\eps)^{1+o(1)})$ and answers the fourth question of Holroyd, Schramm and Wilson \cite{HolroydSchrammWilson2017FinitaryColoring}. 
	\item Can the uniform complexity $(1/\eps) \cdot 2^{O(\sqrt{\log 1/\eps})}$ in the construction of the toast be improved to $(1/\eps) \cdot \poly \log (1/\eps)$? The same can be asked for the complexity of intermediate problems from \cref{sec:Inter}. 
	\item Is there a local problem $\Pi$ such that its uniform complexity is $\Theta(\sqrt[r]{(1/\eps)})$ for $r \not \in \mathbb{N}$? 
    \item What is the uniform complexity of the perfect matching problem on grids? 
	\item Is it true (on $\mathbb{Z}^d$) that $\BOREL=\MEASURE$ or $\borel = \baire$? 
\end{enumerate}

\section*{Acknowledgement}
We would like to thank F. Bencs, A. Bernshteyn, Y. Chang, M. Ghaffari, A. Hrušková, S. Jackson, O. Pikhurko, B. Seward, Y. Spinka, L. Tóth, and Z. Vidnyánszky for many engaging discussions. We would also like to thank the reviewers for many helpful comments. 
The first author was supported by Leverhulme Research Project Grant RPG-2018-424. 
This project has received funding from the European Research Council (ERC) under the European
Unions Horizon 2020 research and innovation programme (grant agreement No. 853109).

\bibliographystyle{alpha}
\bibliography{ref}

\newcommand{\etalchar}[1]{$^{#1}$}
\begin{thebibliography}{CMTD16}

\bibitem[ARVX12]{lca12}
Noga Alon, Ronitt Rubinfeld, Shai Vardi, and Ning Xie.
\newblock Space-efficient local computation algorithms.
\newblock In {\em Proc.\ 23rd ACM-SIAM Symp.\ on Discrete Algorithms (SODA)},
  pages 1132--1139, 2012.

\bibitem[BBOS20]{balliu2020almost_global_problems}
Alkida Balliu, Sebastian Brandt, Dennis Olivetti, and Jukka Suomela.
\newblock Almost global problems in the local model.
\newblock {\em Distributed Computing}, pages 1--23, 2020.

\bibitem[BCG{\etalchar{+}}22]{brandt_chang_grebik_grunau_rozhon_vidnyaszky2021LCLs_on_trees_descriptive}
Sebastian Brandt, Yi-Jun Chang, Jan Greb{\'\i}k, Christoph Grunau, V{\'a}clav
  Rozho{\v{n}}, and Zolt{\'a}n Vidny{\'a}nszky.
\newblock Local problems on trees from the perspectives of distributed
  algorithms, finitary factors, and descriptive combinatorics.
\newblock {\em Innovations of Theoretical Computer Science (ITCS)}, 2022.

\bibitem[Ber23a]{Bernshteyn2021LLL}
Anton Bernshteyn.
\newblock Distributed algorithms, the {L}ovász {L}ocal {L}emma, and
  descriptive combinatorics.
\newblock {\em to appear Inventiones Mathematicae}, 2023.

\bibitem[Ber23b]{Bernshteyn2021local=cont}
Anton Bernshteyn.
\newblock Probabilistic constructions in continuous combinatorics and a bridge
  to distributed algorithms.
\newblock {\em Advances in Mathematics}, 415:108895, 2023.

\bibitem[BFH{\etalchar{+}}16]{brandt}
S.~Brandt, O.~Fischer, J.~Hirvonen, B.~Keller, T.~Lempi{\"a}inen, J.~Rybicki,
  J.~Suomela, and J.~Uitto.
\newblock A lower bound for the distributed {L}ov{\'a}sz local lemma.
\newblock In {\em Proc.\ 48th ACM Symp.\ on Theory of Computing (STOC)}, pages
  479--488, 2016.

\bibitem[BGR20]{brandt_grunau_rozhon2020tightLLL}
Sebastian Brandt, Christoph Grunau, and V{\'a}clav Rozho{\v{n}}.
\newblock Generalizing the sharp threshold phenomenon for the distributed
  complexity of the lov{\'a}sz local lemma.
\newblock In {\em Proceedings of the 39th Symposium on Principles of
  Distributed Computing}, pages 329--338, 2020.

\bibitem[BGR21]{brandt_grunau_rozhon2021LLL_in_LCA}
Sebastian Brandt, Christoph Grunau, and V\'{a}clav Rozho\v{n}.
\newblock The randomized local computation complexity of the lov\'{a}sz local
  lemma.
\newblock In {\em Proceedings of the 2021 ACM Symposium on Principles of
  Distributed Computing}, PODC'21, page 307–317, New York, NY, USA, 2021.
  Association for Computing Machinery.

\bibitem[BHK{\etalchar{+}}17]{brandt_grids}
Sebastian Brandt, Juho Hirvonen, Janne~H. Korhonen, Tuomo Lempi\"{a}inen,
  Patric~R.J. \"{O}sterg\r{a}rd, Christopher Purcell, Joel Rybicki, Jukka
  Suomela, and Przemys\l{}aw Uzna\'{n}ski.
\newblock {L}{C}{L} problems on grids.
\newblock In {\em Proceedings of the ACM Symposium on Principles of Distributed
  Computing}, PODC '17, page 101–110, New York, NY, USA, 2017. Association
  for Computing Machinery.

\bibitem[BHK{\etalchar{+}}18]{balliu2018new_classes-loglog*-log*}
Alkida Balliu, Juho Hirvonen, Janne~H Korhonen, Tuomo Lempi{\"a}inen, Dennis
  Olivetti, and Jukka Suomela.
\newblock New classes of distributed time complexity.
\newblock In {\em Proceedings of the 50th Annual ACM SIGACT Symposium on Theory
  of Computing}, pages 1307--1318, 2018.

\bibitem[BHT21]{ArankaFeriLaci}
Ferenc Bencs, Aranka Hru{\v{s}}kov{\'a}, and L{\'a}szl{\'o}~M{\'a}rton
  T{\'o}th.
\newblock Factor-of-iid {S}chreier decorations of lattices in {E}uclidean
  spaces.
\newblock {\em arXiv preprint arXiv:2101.12577}, 2021.

\bibitem[BMU19]{brandt_maus_uitto2019tightLLL}
Sebastian Brandt, Yannic Maus, and Jara Uitto.
\newblock A sharp threshold phenomenon for the distributed complexity of the
  lov{\'a}sz local lemma.
\newblock In {\em Proceedings of the 2019 ACM Symposium on Principles of
  Distributed Computing}, pages 389--398, 2019.

\bibitem[BT19]{barenboim2018avg_complexity}
Leonid Barenboim and Yaniv Tzur.
\newblock Distributed symmetry-breaking with improved vertex-averaged
  complexity.
\newblock In {\em Proceedings of the 20th International Conference on
  Distributed Computing and Networking}, pages 31--40, 2019.

\bibitem[CDHP20]{chang2020energy_complexity_bfs}
Yi-Jun Chang, Varsha Dani, Thomas~P Hayes, and Seth Pettie.
\newblock The energy complexity of bfs in radio networks.
\newblock In {\em Proceedings of the 39th Symposium on Principles of
  Distributed Computing}, pages 273--282, 2020.

\bibitem[CGM{\etalchar{+}}]{OLegLLL}
Endre Cs\'{o}ka, Łukasz Grabowsk, András M\'{a}th\'{e}, Oleg Pikhurko, and
  Konstantinos Tyros.
\newblock {M}oser-{T}ardos {A}lgorithm with small number of random bits.
\newblock {\em arXiv:2203.05888}.

\bibitem[CGP20a]{chatterjee2020sleeping_MIS}
Soumyottam Chatterjee, Robert Gmyr, and Gopal Pandurangan.
\newblock Sleeping is efficient: Mis in o (1)-rounds node-averaged awake
  complexity.
\newblock In {\em Proceedings of the 39th Symposium on Principles of
  Distributed Computing}, pages 99--108, 2020.

\bibitem[CGP20b]{CGP}
Clinton~T. Conley, Jan Greb\'ik, and Oleg Pikhurko.
\newblock Divisibility of spheres with measurable pieces.
\newblock {\em accepted by L'Enseignement Mathematique}, 2020.

\bibitem[Cha20]{chang2020n1k_speedups}
Yi-Jun Chang.
\newblock {The Complexity Landscape of Distributed Locally Checkable Problems
  on Trees}.
\newblock In Hagit Attiya, editor, {\em 34th International Symposium on
  Distributed Computing (DISC 2020)}, volume 179 of {\em Leibniz International
  Proceedings in Informatics (LIPIcs)}, pages 18:1--18:17, Dagstuhl, Germany,
  2020. Schloss Dagstuhl--Leibniz-Zentrum f{\"u}r Informatik.

\bibitem[CJM{\etalchar{+}}20]{MarksLLL}
Clinton Conley, Steve Jackson, Andrew Marks, Brandon Seward, and Robin
  Tucker-Drob.
\newblock Hyperfiniteness and {B}orel combinatorics.
\newblock {\em J. Eur. Math. Soc. (JEMS)}, 22(3):877--892, 2020.

\bibitem[CJM{\etalchar{+}}23]{AsymptoticDim}
Clinton Conley, Steve Jackson, Andrew Marks, Brandon Seward, and Robin
  Tucker-Drob.
\newblock Borel asymptotic dimension and hyperfinite equivalence relations.
\newblock {\em to appear Duke Mathematical Journal}, 2023.

\bibitem[CKP16]{chang2016exp_separation}
Yi-Jun Chang, Tswi Kopelowitz, and Seth Pettie.
\newblock An exponential separation between randomized and deterministic
  complexity in the {LOCAL} model.
\newblock In {\em Proc.\ 57th IEEE Symp.\ on Foundations of Computer Science
  (FOCS)}, 2016.

\bibitem[CKP{\etalchar{+}}19]{chang2018energy_complexity_exp_separation}
Yi-Jun Chang, Tsvi Kopelowitz, Seth Pettie, Ruosong Wang, and Wei Zhan.
\newblock Exponential separations in the energy complexity of leader election.
\newblock {\em ACM Transactions on Algorithms (TALG)}, 15(4):1--31, 2019.

\bibitem[CM16]{FirstToast}
Clinton Conley and B.~D. Miller.
\newblock A bound on measurable chromatic numbers of locally finite {B}orel
  graphs.
\newblock {\em Mathematical Research Letters}, 23(6):1633--1644, 2016.

\bibitem[CMTD16]{BrooksMeas}
Clinton Conley, Andrew~S. Marks, and Robin~D. Tucker-Drob.
\newblock Brooks' theorem for measurable colorings.
\newblock {\em Forum Math. Sigma}, e16(4):23pp, 2016.

\bibitem[CP17]{chang2017time_hierarchy}
Yi-Jun Chang and Seth Pettie.
\newblock A time hierarchy theorem for the {LOCAL} model.
\newblock In {\em Proc.\ 58th IEEE Symp.\ on Foundations of Computer Science
  (FOCS)}, pages 156--167, 2017.

\bibitem[CUS23]{spencer_personal}
Nishant {C}handgotia and {U}nger Spencer.
\newblock Borel factors and embeddings of systems in subshifts.
\newblock {\em accepted by Israel Journal of Mathematics}, 2023.

\bibitem[Ele18]{elek}
Gábor Elek.
\newblock Qualitative graph limit theory. cantor dynamical systems and
  constant-time distributed algorithms.
\newblock {\em arXiv:1812.07511}, 2018.

\bibitem[ELRS03]{even03_conflictfree}
Guy Even, Zvi Lotker, Dana Ron, and Shakhar Smorodinsky.
\newblock Conflict-free colorings of simple geometric regions with applications
  to frequency assignment in cellular networks.
\newblock {\em SIAM J.\ Computing}, 33(1):94--136, 2003.

\bibitem[Feu20]{feuilloley2017node_avg_complexity_of_col_is_const}
Laurent Feuilloley.
\newblock How long it takes for an ordinary node with an ordinary id to output?
\newblock {\em Theoretical Computer Science}, 811:42--55, 2020.

\bibitem[FG17]{fischer2017sublogarithmic}
Manuela Fischer and Mohsen Ghaffari.
\newblock Sublogarithmic distributed algorithms for lovász local lemma, and
  the complexity hierarchy.
\newblock In {\em 31st International Symposium on Distributed Computing (DISC
  2017)}, volume~91, page~18. Schloss Dagstuhl, Leibniz-Zentrum fuer
  Informatik, 2017.

\bibitem[GGR21]{ghaffari_grunau_rozhon2020improved_network_decomposition}
Mohsen Ghaffari, Christoph Grunau, and V\'{a}clav Rozho\v{n}.
\newblock Improved deterministic network decomposition.
\newblock In {\em Proc. of the 32. ACM-SIAM Symp. on Discrete Algorithms
  (SODA)}, page 2904–2923, USA, 2021. Society for Industrial and Applied
  Mathematics.

\bibitem[GGRT23]{MeasurableTiling}
Jan Grebík, Rachel Greenfeld, Václav Rozhoň, and Terence Tao.
\newblock Measurable tilings by abelian group actions.
\newblock {\em International Mathematics Research Notices}, 2023.

\bibitem[Gha16]{ghaffari2016MIS}
Mohsen Ghaffari.
\newblock An improved distributed algorithm for maximal independent set.
\newblock In {\em Proc.\ ACM-SIAM Symp.\ on Discrete Algorithms (SODA)}, pages
  270--277, 2016.

\bibitem[Gha19]{ghaffari2019distributed}
Mohsen Ghaffari.
\newblock Distributed maximal independent set using small messages.
\newblock In {\em Proceedings of the Thirtieth Annual ACM-SIAM Symposium on
  Discrete Algorithms}, pages 805--820. SIAM, 2019.

\bibitem[GHK18]{ghaffari_harris_kuhn2018derandomizing}
Mohsen Ghaffari, David Harris, and Fabian Kuhn.
\newblock On derandomizing local distributed algorithms.
\newblock In {\em Proc.~Foundations of Computer Science (FOCS)}, pages
  662--673, 2018.

\bibitem[GJ15]{GaoJackson}
Su~Gao and Steve. Jackson.
\newblock Countable abelian group actions and hyperfinite equivalence
  relations.
\newblock {\em Inventiones Math.}, 201(1):309--383, 2015.

\bibitem[GJKS]{GJKS2}
Su~Gao, Steve Jackson, Edward Krohne, and Brandon. Seward.
\newblock Borel combinatorics of countable group actions.
\newblock {\em in preparation}.

\bibitem[GJKS23]{GJKS}
Su~Gao, Steve Jackson, Edward Krohne, and Brandon Seward.
\newblock Continuous combinatorics of abelian group actions.
\newblock {\em to appear Memoirs of the American Mathematical Society}, 2023.

\bibitem[GJS09]{ColorinSeward}
Su~Gao, Steve Jackson, and Brandon. Seward.
\newblock A coloring property for countable groups.
\newblock {\em Mathematical Proceedings of the Cambridge Philosophical
  Society}, 147(3):579--592, 2009.

\bibitem[GKM17]{ghaffari_kuhn_maus2017slocal}
Mohsen Ghaffari, Fabian Kuhn, and Yannic Maus.
\newblock On the complexity of local distributed graph problems.
\newblock In {\em Proc.\ 49th ACM Symp.\ on Theory of Computing (STOC)}, pages
  784--797, 2017.

\bibitem[GMP17]{OlegCircle}
Lukasz Grabowski, András Máthé, and Oleg Pikhurko.
\newblock Measurable circle squaring.
\newblock {\em Ann. of Math.}, 185(2):671--710, 2017.

\bibitem[GR21]{grebik_rozhon2021LCL_on_paths}
Jan Greb{\'\i}k and V{\'a}clav Rozho{\v{n}}.
\newblock Classification of local problems on paths from the perspective of
  descriptive combinatorics.
\newblock In {\em Extended Abstracts EuroComb 2021: European Conference on
  Combinatorics, Graph Theory and Applications}, pages 553--559. Springer,
  2021.

\bibitem[GRB22]{brandt_grebik_grunau_rozhon2021classification_of_lcls_trees_and_grids}
Christoph Grunau, V{\'a}clav Rozho{\v{n}}, and Sebastian Brandt.
\newblock The landscape of distributed complexities on trees and beyond.
\newblock In {\em Proceedings of the 2022 ACM Symposium on Principles of
  Distributed Computing}, pages 37--47, 2022.

\bibitem[GT22]{FailurePTC}
Rachel Greenfeld and Terence Tao.
\newblock A counterexample to the periodic tiling conjecture.
\newblock {\em arXiv:2211.15847}, 2022.

\bibitem[HL16]{FinDepCol}
Alexander~E. Holroyd and Thomas~M. Liggett.
\newblock Finitely dependent coloring.
\newblock {\em Forum Math. Pi}, e9(4):43pp, 2016.

\bibitem[HSW17]{HolroydSchrammWilson2017FinitaryColoring}
Alexander~E. Holroyd, Oded Schramm, and David~B. Wilson.
\newblock Finitary coloring.
\newblock {\em Ann. Probab.}, 45(5):2867--2898, 2017.

\bibitem[Jac]{stephen}
Steve Jackson.
\newblock Personal communication.

\bibitem[KM20]{kechris_marks2016descriptive_comb_survey}
Alexander~S. Kechris and Andrew~S. Marks.
\newblock Descriptive graph combinatorics.
\newblock 2020.

\bibitem[KST99]{KST}
Alexander~S. Kechris, Slawomir Solecki, and Stevo Todor\v{c}evi\'{c}.
\newblock Borel chromatic numbers.
\newblock {\em Adv. Math.}, 141(1):1--44, 1999.

\bibitem[KSV12]{Korman_Sereni_Viennot2012Pruning_algorithms_+_oblivious_coloring}
Amos Korman, Jean-Sébastien Sereni, and Laurent Viennot.
\newblock Toward more localized local algorithms: removing assumptions
  concerning global knowledge.
\newblock {\em Distributed Computing}, 26(5-6):289–308, Sep 2012.

\bibitem[Kun13]{KunLLL}
Gabor Kun.
\newblock Expanders have a spanning {L}ipschitz subgraph with large girth.
\newblock {\em arXiv:1303.4982v2}, 2013.

\bibitem[Lac90]{laczkovich}
Miklós Laczkovich.
\newblock Equidecomposability and discrepancy; a solution of {T}arski's
  circle-squaring problem.
\newblock {\em J. Reine Angew. Math.}, 404:77--117, 1990.

\bibitem[Lin92]{linial92LOCAL}
Nati Linial.
\newblock Locality in distributed graph algorithms.
\newblock {\em SIAM Journal on Computing}, 21(1):193--201, 1992.

\bibitem[Mar16]{DetMarks}
Andrew~S. Marks.
\newblock A determinacy approach to {B}orel combinatorics.
\newblock {\em J. Amer. Math. Soc.}, 29(2):579--600, 2016.

\bibitem[MNP]{JordanCircle}
András Máthé, Jonathan~A. Noel, and Oleg Pikhurko.
\newblock Circle {S}quaring with {P}ieces of {S}mall {B}oundary and {L}ow
  {B}orel {C}omplexity.
\newblock {\em https://arxiv.org/abs/2202.01412}.

\bibitem[MU17]{Circle}
Andrew~S. Marks and Spencer~T. Unger.
\newblock Borel circle squaring.
\newblock {\em Ann. of Math.}, 186(2):581--605, 2017.

\bibitem[NS95]{naorstockmeyer}
Moni Naor and Larry Stockmeyer.
\newblock What can be computed locally?
\newblock {\em SIAM Journal on Computing}, 24(6):1259--1277, 1995.

\bibitem[Pik21]{pikhurko2021descriptive_comb_survey}
Oleg Pikhurko.
\newblock Borel {C}ombinatorics of {L}ocally {F}inite {G}raphs.
\newblock {\em Surveys in Combinatorics 2021 (edited by K.K.Dabrowski et al),
  the invited volume of the 28th British Combinatorial Conference}, pages
  267--319, 2021.

\bibitem[RG20]{RozhonG19}
Václav Rozho\v{n} and Mohsen Ghaffari.
\newblock Polylogarithmic-time deterministic network decomposition and
  distributed derandomization.
\newblock In {\em Proc.~Symposium on Theory of Computation (STOC)}, 2020.

\bibitem[RS20]{RosenbaumSuomela2020volume_model}
Will Rosenbaum and Jukka Suomela.
\newblock Seeing far vs. seeing wide: Volume complexity of local graph
  problems.
\newblock In {\em Proceedings of the 39th Symposium on Principles of
  Distributed Computing}, PODC '20, page 89–98, New York, NY, USA, 2020.
  Association for Computing Machinery.

\bibitem[RTVX11]{lca11}
Ronitt Rubinfeld, Gil Tamir, Shai Vardi, and Ning Xie.
\newblock Fast local computation algorithms.
\newblock In {\em Proc.\ 2nd Symp.\ on Innovations in Computer Science (ICS)},
  pages 223--238, 2011.

\bibitem[Sew]{Seward_personal}
Brandon Seward.
\newblock Personal communication.

\bibitem[Spi20]{Spinka}
Yinon Spinka.
\newblock Finitely dependent processes are finitary.
\newblock {\em Ann. Probab.}, 48(4):2088--2117, 2020.

\bibitem[STD16]{ST}
Brandon Seward and Robin Tucker-Drob.
\newblock Borel structurability on the 2-shift of a countable group.
\newblock {\em Annals of Pure and Applied Logic}, 167(1):1--21, 2016.

\bibitem[Wei21]{Felix}
Felix Weilacher.
\newblock Borel edge colorings for finite dimensional groups.
\newblock {\em to appear Israel Journal of Mathematics}, 2021.

\end{thebibliography}

\appendix

\section{Preliminaries with uniform algorithms}
\label{subsec:mis_cancer_constructions}

In this section, we discuss some technical results about uniform local algorithms needed in \cref{app:missing_toast_proofs,app:intermediate_problems}. 
First, we state a lemma that bounds the uniform complexity of a parallel composition of two uniform algorithms. 

\begin{lemma}[Parallel Composition]
\label{lem:parallel_composition}
Let $\fA_1$ and $\fA_2$ be two uniform local algorithms with a uniform complexity of $t_1(\eps)$ and $t_2(\eps)$, respectively. 
Let $\fA$ be the parallel composition of $\fA_1$ and $\fA_2$. That is, $\fA$ is defined as a uniform local algorithm that outputs at a given node $u\in \mathbb{Z}^d$ solutions to both $\fA_1$ and $\fA_2$.
Similarly, the coding radius of $\mathcal{A}$ at $u$ is the smallest radius needed to output a solution to both $\fA_1$ and $\fA_2$ at $u$. 
Then, $t_\fA(\eps) \leq \max\{t_{\fA_1}(\eps/2), t_{\fA_2}(\eps/2)\}$.

More generally, for $k$ uniform algorithms $\fA_1, \dots, \fA_k$ and their parallel composition $\fA$, we get $t_\fA(\eps) \leq \max_{1 \le i \le k}t_{\fA_i}(\eps/k) $. 
\end{lemma}
\begin{proof}
The probability that the coding radius of the algorithm $\mathcal{A}_1$ at vertex $u$ is larger than $\max\{t_{\fA_1}(\eps/2), t_{\fA_2}(\eps/2)\}$ is at most $\eps/2$, the same holds for the coding radius of algorithm $\mathcal{A}_2$ at vertex $u$.
Hence, by the union bound, the probability that the coding radius of the algorithm $\mathcal{A}$ at a vertex $u$ is larger than $\max(t_{\fA_1}(\eps/2), t_{\fA_2}(\eps/2))$ is at most $\eps/2 + \eps/2 = \eps$.
Consequently, $t_\fA(\eps) \leq \max\{t_{\fA_1}(\eps/2), t_{\fA_2}(\eps/2)\}$, as desired.
\end{proof}

We will also need to use the uniform distributed algorithm for MIS of \cite[Theorem 1.1]{ghaffari2016MIS}. Note that in the following statement, we do not regard the maximum degree $\Delta$ of a graph as a constant.

\begin{lemma}
\label{lem:mohsen}
There is a uniform local algorithm that constructs a maximal independent set (MIS) in any graph $G$ of maximum degree $\Delta$ in uniform complexity $O(\log\Delta + \log 1/\eps)$ where the constant hidden by the $O$ notation does not depend on $\Delta$. 
\end{lemma}

Let $d>1$, define $G=\mathbb{Z}^d$ and set $r_k=2^{a k^2}$ for some sufficiently large constant $a>0$.
Write $G^{r_k}$ for the $r_k$-power graph of $\mathbb{Z}^d$, i.e., two elements are connected if their graph distance in $G$ is at most $r_k$.
We use \cref{lem:mohsen,lem:parallel_composition} to bound the coding radius needed at a given vertex $u\in \mathbb{Z}^d$ to construct simultaneously maximal independent sets (MISes) in the graphs $G^{r_1}, G^{r_2},\dots, G^{r_k}$ for $k\in \mathbb{N}$ in some neighborhood of $u$.

\begin{lemma}
\label{lem:mis_on_grids}
    Let $k\in \mathbb{N}$.
    There is a uniform local algorithm of uniform complexity $O(10^{20k}r_k \cdot \log (k r_k/\eps))$ (with the hidden constant depending on $d$ but not $k$) that outputs solutions to the maximal independent set (MISes) problems in $G^{r_1}, G^{r_2}, \dots, G^{r_k}$ on a whole neighborhood $\fB(u, 10^{10k} r_k)$ of a given vertex $u\in \mathbb{Z}^d$.
    
    In other words, the probability that a given vertex $u$ looks at distance $O(10^{20k}r_k \cdot \log (k r_k/\eps))$ and there is some $v\in \mathcal{B}(u,10^{10k} r_k)$ for which its membership in one of the independent sets cannot be determined is at most $\epsilon$. 
    
\end{lemma}
\begin{proof}
    The uniform local algorithm is simply defined as the parallel composition of the algorithms of Ghaffari from \cref{lem:mohsen} defined for graphs of maximum degree $(2r_1+1)^d, (2r_2+1)^d, \dots, (2r_k+1)^d$.
    Next, we compute the uniform complexity of such an algorithm using \cref{lem:parallel_composition}. 

    First, the uniform complexity of constructing the $i$-th MIS, where $1\le i\le k$, at a fixed node $u$ is $O(r_i \cdot (d \log r_i + \log 1/\eps))$. This follows from the uniform complexity of \cref{lem:mohsen}: The first factor of $r_i$ in our complexity is because one step of the algorithm of \cref{lem:mohsen} in $G^{r_i}$ requires $r_i$ steps in $G$. The second factor is due to the fact that the maximum degree of $G^{r_i}$ is bounded by $(2r_i+1)^d$. 
    Using \cref{lem:parallel_composition}, we get that the uniform complexity to output solution to all MISes at $u$ is upper bounded by $O( r_k (d\log r_k + \log k/\eps)) = O(r_k \cdot \log (kr_k/\eps) ) $.
    
    Using \cref{lem:parallel_composition} again, we get that the uniform complexity of constructing all MISes in all nodes of $\fB(u, 10^{10k}r_k)$ is at most $O(r_k \cdot \log (kr_k \cdot (10^{10k}r_k)/\eps) ) = O(kr_k \cdot \log (k r_k/\eps))$ as $|\fB(u, 10^{10k}r_k)| \le (2\cdot 10^{10k}r_k+ 1)^d$.
    This means that if $u$ looks at distance $10^{10k}r_k + O(kr_k \cdot \log (k r_k/\eps)) = O(10^{20k}r_k \cdot \log (k r_k/\eps))$, it can output solution to all MISes in $\fB(u, 10^{10k}r_k)$ with probability at least $1-\epsilon$, as needed.  
\end{proof}

\section{Uniform Construction of the Toast -- Proof of \cref{thm:toast_in_tail}}
\label{subsec:toast_construction}
\label{app:missing_toast_proofs}

Here we prove \cref{thm:toast_in_tail} that we restate here for convenience. Recall that the sketch of the proof is given in \cref{subsec:toast_construction_in_tail_sketch}
, in particular see \cref{fig:toast_construction}. 

\toastconstruction*

We show that for every $q\in \mathbb{N}$ there is a uniform local algorithm $\fA$ that produces a $q$-toast with uniform complexity as in the theorem statement. 
More formally, let $q\in \mathbb{N}$.
Recall that a $q$-toast on $\mathbb{Z}^d$ is a collection $\mathcal{D}$ of finite connected subsets of $\mathbb{Z}^d$ such that
\begin{itemize}
    \item for every $g,h\in \mathbb{Z}^d$ there is $D\in\mathcal{D}$ such that $g,h\in D$,
    \item if $D,E\in \mathcal{D}$ and $d(\partial D,\partial E)<q$, then $D=E$.
\end{itemize}
When we say that a uniform local algorithm $\fA$ produces a $q$-toast, we mean that the algorithm computes for every vertex $v$ the minimal $D \in \fD$ such that $v \in D$.
We denote the corresponding coding radius as $\widetilde{R}(\fA)$.



\begin{theorem}\label{thm:toast moment}
Let $q\in \mathbb{N}$ be large enough. 
Then there is a uniform local algorithm $\fA$ that produces a $q$-toast such that 
\[
\P(\widetilde{R}(\fA) > (1/\eps)2^{O(\sqrt{\log 1/\eps})}) < \eps. 
\]
\end{theorem}

\cref{thm:toast_in_tail} is an immediate corollary of \cref{thm:toast moment} that is proven in the rest of this section. 

\begin{proof}[Proof of Theorem~\ref{thm:toast moment}]
Let $q$ and $a$ be large enough constants (depending on $d$) whose value is determined later.
Set $r_k=2^{ak^2}$ for every $k\in \mathbb{N}$; intuitively, $r_k$ is the scale on which we construct the toast sets constructed in the $k$-th iteration of the algorithm. 
Fix a sequence of uniform local algorithms $\tilde{\fA}=\{\fA_k\}_{k>0}$ such that $\fA_k$ computes an MIS in $G^{r_k}$. Write $X_k$ for the set that is produced by $\fA_k$. 
Define $s_1:=r_1-q$ and 
\begin{equation}
    s_{k+1}:=r_{k+1}-1-5\sum_{i\le k} r_i.
\end{equation}
Intuitively, $s_k$ is the radius of toast sets constructed in the $k$-th iteration. 

Denote as $V(x,k)$ the Voronoi cell centered at $x\in X_k$, that is $y\in V(x,k)$ if $d(y,x)=d(y,X_k)$, where we break ties arbitrarily.
Since $X_k$ is maximal $r_k$-independent, we have that $\{V(x,k)\}_{x\in X_k}$ is a decomposition of $\mathbb{Z}^d$ and
\begin{equation}\label{eq:voronoi1}
\fB\left(x,\frac{r_k}{2}\right)\subseteq V(x,k)\subseteq \fB\left(x,r_k\right)
\end{equation}
for every $k\in \mathbb{N}$. We note that unless stated otherwise, balls $\fB$ are with respect to the graph metric of the graph $G=\mathbb{Z}^d$. Thus, we have
\begin{equation}\label{eq:l1_volume}
    (2r/d)^d \le |\fB(x, r)| \le (2r+1)^d. 
\end{equation}
where the estimates follow from the facts that for any $r$ we have
\begin{equation}
    \fB_{\ell_\infty}(x, r/d) \subseteq \fB_{\ell_1}(x, r) \subseteq \fB_{\ell_\infty}(x, r). 
\end{equation}

Set
\begin{equation}\label{eq:create voronoi}
C(x,k)=\left\{y\in V(x,k):d(y,\partial V(x,k))\ge r_k-s_k\right\}   
\end{equation}
for every $k\in \mathbb{N}$ and $x\in X_k$. The set $C(x,k)$ corresponds to the set $C$ in \cref{fig:toast_construction}. 

\paragraph{The toast construction algorithm}
Now we are ready to define inductively a $q$-toast (see the discussion in  \cref{subsec:toast_construction_in_tail_sketch} and  \cref{fig:toast_construction}).
We start with setting $\mathcal{D}_1=\{D(x,1)\}_{x\in X_1}$, where $D(x,1):=C(x,1)$.
Suppose that $\mathcal{D}_k=\{D(x,k)\}_{x\in X_k}$ has been defined.
Let
\begin{itemize}
	\item $H^{k+1}(x,k+1):=C(x,k+1)$ for every $x\in X_{k+1}$,
	\item if $1\le i\le k$ and $\left\{H^{k+1}(x,i+1)\right\}_{x\in X_{k+1}}$ has been defined, then we put
	$$H^{k+1}(x,i)=\bigcup \left\{V(y,i):H^{k+1}(x,i+1)\cap V(y,i)\not=\emptyset\right\}$$
	for every $x\in X_{k+1}$,
	\item set $D(x,k+1):=H^{k+1}(x,1)$ for every $x\in X_{k+1}$, this defines $\mathcal{D}_{k+1}$.
\end{itemize}
Finally, put $\mathcal{D}=\{D(x,k)\in \mathcal{D}_k:x\in X_k , \ k\in \mathbb{N}\}$. This construction intuitively follows \cref{fig:toast_construction}. 
Before we define the uniform local algorithm $\fA$ that produces $\mathcal{D}$, we compute how quickly the cells cover $\mathbb{Z}^d$ in \cref{cl:nvm} and show that $\mathcal{D}$ is a $q$-toast in \cref{cl:toast}. 
In the claim below, $\P(S)$ denotes the probability that any fixed node $v$ is in the set $S$.

\begin{claim}
\label{cl:nvm}
We have 
\begin{equation}
    \P\left(\bigcup_{i\le k, \ x\in X_i} D(x,i)\right)\ge 1-\frac{\left(5q(100d)^d\right)^{k}}{r_k}
\end{equation}
for every $k\in \mathbb{N}$.
\end{claim}
\begin{proof}
Let $k > 1$. 
We rely on the fact that 
\begin{align}
 \label{eq:eight}
    \P\left(\bigcup_{i\le k, \ x\in X_i} D(x,i)\right)
    \ge \P\left(\bigcup_{i\le k, \ x\in X_i} C(x,i)\right)
\end{align}
thus in the rest of the proof, we  lower bound the right-hand side of \cref{eq:eight}. 

Let $x \in X_k$.
We upper bound $|V(x,k) \setminus C(x,k)|$. To this end, first, note that we can write 
\begin{align}
\label{eq:fml1}
|V(x,k) \setminus C(x,k)| \le \sum_{x' \in X_k \cap \fB(x, 3r_k)} |S(x, x') \cap \fB(x, r_k)|    
\end{align}
where $S(x, x')$ is defined as the set of points $y \in V(G)$ for which 
\begin{align}
 0 \le d(y, x') - d(y, x) \le r_k - s_k.   
\end{align}
That is, any point $x' \in X_k$ sufficiently close to $x$ can define a face of the Voronoi cell $V(x,k)$ and hence also a ``slice'' $S(x, x') \cap V(x,k)$ of points of $V(x,k)$ that are too close to the boundary of $V(x,k)$. 
We consider only points $x' \in \fB(x, 3r_k)$ in \cref{eq:fml1} since points outside $\fB(x, 3r_k)$ define a slice $S(x,x')$ disjoint with $\fB(x, r_k)$.
In the next two paragraphs, we bound the sizes of the set $X_k \cap \fB(x, 3r_k)$ and $S(x, x') \cap \fB(x, r_k)$ for each $x'\in X_k \cap \fB(x, 3r_k)$.
Altogether we aim to show that
$$    |V(x,k) \setminus C(x,k)| \le (50d)^d r_k^{d-1} (r_k - s_k + d).$$

First, we estimate the size of the set $X_k \cap \fB(x, 3r_k)$ using a standard volume argument.
Using \cref{eq:l1_volume}, we conclude that every $x' \in X_k \cap \fB(x, 3r_k)$ can be assigned a set of $|\fB(x', r_k/2)| \ge (r_k/d)^d$ points of $\fB(x, 3.5r_k)$, disjoint from other such points $x'$. Thus, using \cref{eq:l1_volume} again, we get 
\begin{align}
\label{eq:fml2}
 |\{x' \in X_k \cap \fB(x, 3r_k)\}| \le (7r_k)^d/(r_k/d)^d = (7d)^d   .
\end{align}

Next, we upper bound the size of each set $S(x, x') \cap \fB(x, r_k)$ as follows. Note that there is a natural inclusion of $\mathbb{Z}^d$ with the graph metric to $\R^d$ with the $\ell_1$ metric, we also consider the $\ell_\infty$ and $\ell_2$ metrics on $\mathbb{R}^d$.
We extend the notation of balls naturally and use $\fB_{\R^d, \ell_1}(x, r)$ and $\fB_{\R^d, \ell_\infty}(x, r)$.
Define a set $S'(x, x') \subseteq \R^d$ as the set of all points $y \in \fB_{\R^d, \ell_1}(x, r_k)$ such that 
\begin{align}
 -d/2 \le d_{\R^d, \ell_1}(y, x') - d_{\R^d, \ell_1}(y,x) \le r_k - s_k + d/2.   
\end{align}
Consider any point $y \in S(x, x') \cap \fB(x, r_k)$ and assign a ball $\fB_{\R^d, \ell_\infty}(y, 1/2)$ to it (as we consider $\ell_\infty$ metric, this ``ball'' is actually a cube of unit volume). Each such ball is a subset of $S'(x, x')$, since the $\ell_1$ radius of  $\fB_{\R^d, \ell_\infty}(y, 1/2)$ is $d/2$ which is also the slack we used to define $S'(x, x')$. Moreover, the balls of different nodes $y \in S(x, x')$ are disjoint up to measure zero, thus we conclude that 
\begin{align}
    \label{eq:fml3}
    |S(x, x') \cap \fB(x, r_k)| 
    \le \frac{\mu( S'(x, x') \cap \fB_{\R^d, \ell_1}(x, r_k + d/2) )}{\mu(\fB_{\R^d, \ell_\infty}(x, 1/2))} 
    = \mu( S'(x, x') \cap \fB_{\R^d, \ell_1}(x, r_k + d/2) )
\end{align}
where $\mu$ is the standard Lebesgue measure on $\R^d$ and in particular $\mu(\fB_{\R^d, \ell_\infty}(x, 1/2)) = 1$.
Notice that the set $S(x, x')$ is exactly the set of points contained between two parallel hyperplanes of $\ell_1$-distance at most $r_k - s_k + d$: this is because the equation $d_{\R^d, \ell_1}(y, x') - d_{\R^d, \ell_1}(y, x) = c$ for some $c\in \R$ defines a hyperplane. Thus, the right hand side of \cref{eq:fml3} can be upper bounded by the volume of a $(d-1)$-dimensional $\ell_1$-ball of radius $r_k + d/2$ times $(r_k - s_k + d)$.
Now we estimate the volume of this body using the fact that $\ell_2$ distance dominates $\ell_1$ and $\mu\left(\fB_{\R^{d-1}, \ell_2}(x, s)\right)\le \mu\left(\fB_{\R^{d-1}, \ell_\infty}(x, s)\right) \le (2s)^{d-1}$ for every $s>0$ and $d\ge 2$.
Altogether, this gives 
\begin{align}
    \label{eq:fml4}
    |S(x, x') \cap \fB(x, r_k)| 
    \le (2(r_k+d/2))^{d-1} \cdot (r_k - s_k + d)
\end{align}
and, putting \cref{eq:fml1,eq:fml2,eq:fml4} together, we conclude that 
\begin{align}
    |V(x,k) \setminus C(x,k)| \le (50d)^d r_k^{d-1} (r_k - s_k + d)
\end{align}
as desired.

Using \cref{eq:voronoi1}, we have 
\begin{equation}\label{eq:est k}
\begin{split}
\frac{|C(x,k)|}{|V(x,k)|}\ge & \ \frac{|V(x,k)|-(50d)^d r^{d-1}_k (r_k-s_k+ d))}{|V(x,k)|} \\
\ge & \ 1-\frac{(50d)^dr^{d-1}_k(r_k-s_k+ d)}{r^d_k}=1-(50d)^d\frac{r_k - s_k +  d}{r_k} \\
=& \ 1-(50d)^d \frac{1 + 5 \sum_{j \le k-1} r_{k-1} + d}{r_k} \\
\ge& \ 1-(100d)^d \frac{5 r_{k-1}}{r_k}. \\
\end{split}
\end{equation}
Similarly for $k=1$, we have 
\begin{equation}\label{eq:est 0}
\frac{|C(x,1)|}{|V(x,1)|}\ge 1-\frac{q(100d)^d}{r_1}
\end{equation}
by the choice of $s_1$ and $a$, for every $x\in X_1$.

We finish with a telescoping product argument (see \cref{subsec:toast_construction_in_tail_sketch} for the intuition behind it). 
In iteration $k$, consider any point $x$ which is included in Voronoi cells $V(x_1, 1), V(x_2, 2), \dots, V(x_k, k)$ where for every $1\le i \le k$ we have $x_i \in X_i$. 
Note that $x \not\in \bigcup_{i \le k} D(x_i,i)$ in particular implies that $x \not\in \bigcup_{i \le k} C(x_i, i)$. 
However, \cref{eq:est k} can be interpreted as saying the probability of the event $x \not\in C(x_i, i)$ is at most $(100d)^d \frac{5 r_{i-1}}{r_i}$. 
Moreover, by construction, all events $x \not\in C(x_i, i)$ are independent, and we thus conclude that

\begin{align}
\label{eq:teleskop}
1 - \P\left(\bigcup_{i\le k, \ x\in X_i} D(x,i)\right)
&\le 
\left( 5q(100d)^d\frac{r_{k-1}}{r_{k}}\right) \cdot \left( 5q(100d)^d\frac{r_{k-2}}{r_{k-1}}\right)  \dots  \left( 5q(100d)^d\frac{1}{r_{1}}\right)\\
 &   = \left( 5q(100d)^d \right)^k \cdot \frac{1}{r_k}
\end{align}
as needed. 

\end{proof}

\begin{claim}
\label{cl:toast}
$\mathcal{D}$ is a $q$-toast with probability $1$. 
\end{claim}
\begin{proof}
Let $k=1$, then we have $\fB(D(x,1),q)\subseteq V(x,1)$ for every $x\in X_1$ by \eqref{eq:create voronoi}.
Similarly, it follows from \eqref{eq:voronoi1}, \eqref{eq:create voronoi} and the inductive construction of $\mathcal{D}$ that
\begin{equation}\label{eq:voronoi3}
D(x,k)\subseteq \fB\left(C(x,k), 3\sum_{i\le k-1} r_i\right)\subseteq V(x,k)    
\end{equation}
holds for every $x\in X_k$ and $k>1$.
Altogether, since $r_1\ge 2q$, we have
\begin{equation}\label{eq:voronoi2}
D(x,k)\subseteq \fB\left(D(x,k), q\right)\subseteq V(x,k)    
\end{equation}
holds for every $x\in X_k$ and $k\in \mathbb{N}$.

Let $x\in X_k$ and $x'\in X_l$ for $\ell\le k$.
Suppose that $D(x,k)\not =D(x',\ell)$.
If $k=\ell$, then $d(\partial D(x,k),\partial D(x',\ell))\ge q$ by \eqref{eq:voronoi2}.
Suppose that $\ell<k$.
There are two possibilities.

First case, $H^{k}(x,\ell+1)\cap V(x',\ell)\not=\emptyset$.
In that case, we have
$$\fB\left(D(x',\ell), q\right)\subseteq V(x',\ell)\subseteq D(x,k)$$ by \eqref{eq:voronoi2}.
This gives $d(\partial D(x,k),\partial D(x',\ell))\ge q$.

Second case, $H^{k}(x,\ell+1)\cap V(x',\ell)=\emptyset$.
If $\ell=1$, then we are done by \eqref{eq:voronoi2}.
Otherwise, it is easy to see from the inductive construction that 
$$d(z,\partial V(x',\ell))\le \sum_{i\le \ell-1}2r_i$$
for every $z\in V(x',\ell)\cap D(x,k)$.
By \eqref{eq:voronoi3}, we have
$$d(z,D(x',\ell))\ge \sum_{i\le \ell-1}r_i\ge q.$$
This shows that $\mathcal{D}$ satisfies the boundary condition.

It follows from \cref{cl:nvm} that with probability $1$ every vertex is covered by some cell $D(x,k)$ for some $k>0$ and $x\in X_k$.
Pick vertices $v,w$ and consider any path $P$ between them.
Since every vertex on $P$ is covered by some cell with probability $1$ and the boundaries are uniformly separated by $q>1$, we see that there must be a cell that contains both $v,w$.
This shows that $\mathcal{D}$ is indeed a $q$-toast.
\end{proof}

Observe that if a vertex $u$ knows the output of $\{\fA_i\}_{0<i\le k}$ on $\mathcal{B}(u,10r_k)$, then it can compute $\{\mathcal{D}_i\}_{0<i\le k}$ on $\fB(u,5r_k)$.
By \cref{cl:toast} we know that with probability $1$ a vertex $u$ is covered by some $D(x,k)$ for some $x\in X_k$ and $k>0$.
We define the toast constructing algorithm $\fA$ as follows.
Let $u$ explore its neighborhood until it finds $k\in \mathbb{N}$ such that it can output solutions to $\{\fA_i\}_{0<i\le k}$ on $\mathcal{B}(u,10r_k)$, and it is the case that there is some $x\in X_k$ such that $u\in D(x,k)$.
Note that this happens with probability $1$ and therefore $\fA$ is a well-defined uniform local algorithm.

Moreover, once $u$ finishes we have $u\in D(x,k)\subseteq \fB(u,5r_k)$, thus $u$ can compute the whole cell $D(x,k)$, along with the smaller cells contained in it, which is what we require to be computed.  
It remains to compute the coding radius $\widetilde{R}(\fA)$ of $\fA$.

\begin{claim}
\label{cl:pls_kill_me}
There is some $C > 0$ such that 
\begin{align}
\P \left(\widetilde{R}(\fA) > C \cdot 10^{20k} r_k \log (kr_k) \right)
\le 2\frac{\left(5q(100d)^d\right)^{k}}{r_k} 
\end{align}
\end{claim}
\begin{proof}
Let $\epsilon>0$.
There are two possible reasons why a node $u$ needs to have a larger coding radius than $C' \cdot 10^{20k} r_k \log kr_k/\eps$, where $C'$ is the constant from \cref{lem:mis_on_grids}.
Either $u$ is not contained in $\{ \fD_i \}_{1 \le i \le k}$, or it cannot construct $\{ \fD_i \}_{1 \le i \le k}$ in that coding radius. 

The probability of the first event is upper bounded by $\frac{\left(5q(100d)^d\right)^{k}}{r_k}$ by \cref{cl:nvm}. 
The probability of the second event is upper bounded by \cref{lem:mis_on_grids} since, as discussed above this claim statement, outputting the solutions to the first $k$ MISes at $\fB(u, 10r_k)$ implies we can reconstruct $\{\fD_i\}_{1 \le i \le k}$ at $u$. In particular, we use \cref{lem:mis_on_grids} with $\eps = 1/r_k < \frac{\left(5q(100d)^d\right)^{k}}{r_k}$ to arrive at the bound at the right-hand side of \cref{cl:pls_kill_me}. 
\end{proof}

To finish, consider an arbitrary $r$ such that 
\begin{align}
\label{eq:kkk}
C\cdot 10^{20k}r_k \log (kr_k) < r < C\cdot 10^{20(k+1)}r_{k+1} \log ((k+1)r_{k+1})    
\end{align}
where $C$ is the constant from \cref{cl:pls_kill_me}. That claim then implies that 
\[
\P(\widetilde{R}_\fA > r) \le \frac{2^{O(k)}}{2^{ak^2}} = \frac{1}{2^{ak^2 - O(k)}}
\]
On the other hand, the right inequality of \cref{eq:kkk} implies that
\[
r < C\cdot 10^{20(k+1)} 2^{a(k+1)^2} \log((k+1)\cdot 2^{a(k+1)^2}) = 2^{ak^2 + O(k)}. 
\]
Combining the two bounds, we get
\[
\P(\widetilde{R}_\fA > r)
= \frac{2^{O(k)}}{r} 
\]
Finally, the left hand side of $r$ at \cref{eq:kkk} implies $r \ge 2^{ak^2}$, thus $k = O(\sqrt{\log r})$. We conclude that 
\[
\P(\widetilde{R}_\fA > r)
= \frac{2^{O(\sqrt{\log r})}}{r}. 
\]
In other words for an arbitrary parameter $\eps > 0$ we have
\[
\P(\widetilde{R}_\fA > 1/\eps \cdot 2^{O(\sqrt{\log 1/\eps})}) < \eps,
\]
as needed. 

\end{proof}

\section{Intermediate Problems -- Proof of \cref{thm:intermediate}}
\label{app:intermediate_problems}

In this section, we prove \cref{thm:intermediate} that we restate here for convenience. Our notation is the same as in \cref{app:missing_toast_proofs}. 

\intermediate*

\begin{proof}
We start with the following combinatorial claim where we use the fact that we are solving perfect matching in a graph with ``diagonal'' edges. 

\begin{claim}\label{cl:PM}
Let $d\in \mathbb{N}$ and $A\subseteq \mathbb{Z}^d$ be a finite set that is connected in $\mathbb{Z}^d$.
Then $A$ admits a perfect matching in $\Zbox$ if and only if $|A|$ is even.

In particular, if $|A|$ is odd and $v\in A$ such that $A\setminus\{v\}$ remains connected in the standard graph, then $A\setminus\{v\}$ admits perfect matching in the graph $\Zbox$.

\end{claim}
\begin{proof}
It is clear that $|A|$ being even is a necessary condition.
We prove by induction that it is also sufficient.
If $|A|=2$, then the claim clearly holds.
Suppose that the claim holds for every set of size $2k$ and $|A|=2k+2$.

Let $T$ be a spanning tree of $A$ in the Cayley graph $\mathbb{Z}^d$, i.e., $T$ is a connected tree that contains all vertices of $A$.
Fix $v\in T$ and let $w\in T$ be a vertex of maximum distance from $v$ in $T$.
Since $T$ is a tree, $w$ is a leaf.
Denote as $u$ its unique neighbor and as $P$ the set of neighbors of $u$ that have maximum distance in $T$ from $v$.
Clearly, $w\in P$ and all elements in $P$ are leaves of $T$.

Suppose $|P|\ge 2$.
Note that any $w_0\not =w_1\in P$ form an edge in $\Zbox$ and $A\setminus \{w_0,w_1\}$ is still connected because $T\setminus \{w_0,w_1\}$ is a connected spanning tree of it.
Therefore put $\{w_0,w_1\}$ in the matching and apply induction on $A\setminus \{w_0,w_1\}$.

Suppose that $|P|=1$.
Then the degree of $u$ is $2$ because $|A|>2$.
In that case put $\{u,w\}$ in the matching and apply inductive hypothesis on $A\setminus \{u,w\}$.
Note that $A\setminus \{u,w\}$ is connected because $T\setminus \{u,w\}$ remains connected and spans it.
\end{proof}

It follows from Claim~\ref{cl:PM} that we can find a perfect matching of every finite set $A$, that is connected in $\mathbb{Z}^d$, possibly up to one point.
A rough strategy to prove that ${\bf M}(\pm{d})=d-1$ is to push this one problematic element inductively to infinity along a $1$-dimensional skeleton of some rectangular tiling of $\mathbb{Z}^d$.
The main difficulty is that we need a sequence of uniform local algorithms that produce this structure, which in turn leads to computing simultaneously a maximal independent set of increasing diameters as in \cref{app:missing_toast_proofs}. 


We start by defining the ($1$-dimensional) skeleton -- the object with which our algorithm works. In the following definition, $\fD$ is the set of points of a skeleton, $Y \subseteq \fD$ is a set of ``corner points'' of the skeleton, and $r$ is (roughly) the upper bound on the distance of any point of $\Z^d$ from the skeleton. We will slightly abuse notation and use $\fD$ both for a subset of $\Z^d$ and the subgraph of the Cayley graph of $\Z^d$ induced by $\fD$. 

\begin{definition}[Skeleton]
\label{def:skeleton}
A (1-dimensional) \emph{skeleton} $(\fD, Y, r)$ is a triplet, where $Y \subseteq \fD \subseteq \mathbb{Z}^d$, $Y$ is a subset of nodes that have degree greater than $2$ in the subgraph of $\Z^d$ induced by $\fD$, and $r > 0$, that satisfies the following:
\begin{enumerate}
    \item any two points in $Y$ are at least distance $10$ apart in $\mathbb{Z}^d$,
    \item $\fD \setminus Y$ induces a collection of paths in $\mathbb{Z}^d$,
    \item any $x \in \mathbb{Z}^d$ is of distance at most $10dr$ to the set $Y$, 
    \item the path metric $d_\fD({-},{-})$ induced by the subgraph $\fD$ satisfies $d_\fD(y_1, y_2) \le (10rd) d(y_1, y_2)$ for any two points $y_1, y_2 \in Y$. 
\end{enumerate}
\end{definition}

\begin{claim}
\label{claim:mlha}
Let $r>0$.
Then a skeleton $(\fD, Y, r)$ can be constructed with a uniform local algorithm of uniform complexity $O(\log^* 1/\eps)$. 
\end{claim}
\begin{proof}[Proof sketch]
This result is a variant of \cite[Theorem~3.1]{GaoJackson} that shows how to partition $\mathbb{Z}^d$ into boxes with each dimension from the set $\{r_0, r_0+1\}$. Note that once such a set of boxes is found, we may take $Y$ as their corner points and $\fD$ as their edges. The distances between the corner points are stretched only by a small factor that can be upper bounded by $10rd$.

The construction of \cite{GaoJackson} is a continuous one, but as we know from \cref{subsec:CONT}, continuous constructions lead to local constructions of local complexity  $O(\log^* 1/\eps)$.
One can adapt this construction (distort the $1$-dimensional boundary of the boxes) so that all the corner points of the produced skeleton are at least distance $10$ apart.
\end{proof}

Recall, that an $(\alpha, \beta)$-ruling set is a set $S$ where any $x\not=y \in S$ have $d(x,y) \ge \alpha$ and for any $x \not\in S$ we can find $y \in S$ such that $d(x,y) \le \beta$.
Choose $r_k:=2^{ak^2}$ for some $a$ large enough.
Write $\tilde{\fA}=\{\fA_k\}_{k\in \mathbb{N}}$ for the sequence of uniform local algorithms where $\fA_1$ produces a skeleton $(\mathcal{C}_1, X_1, r_1)$ and $\fA_k$ for $k \ge 2$ produces an $(r_k, 2r_k)$-ruling set $X_k$, such that $X_{k+1} \subseteq X_k$. We now prove that such a sequence of uniform algorithms exists and analyze its uniform complexity to be used later. 

\begin{claim}
\label{cl:we_are_crazy}
There is a sequence $\tilde{\fA}=\{\fA_k\}_{k\in \mathbb{N}}$ of uniform local algorithms as above such that once $k\in \mathbb{N}$ is fixed, then the uniform complexity of $\fA_1, \dots, \fA_k$ constructing a skeleton $(\mathcal{C}_1, X_1, r_1)$ and $X_2, \dots, X_k$ such that every $u$ computes the output of $\fA_1, \dots, \fA_k$ on the whole ball $\fB\left(u, 10^{10k} \cdot r_{k}\right)$ is $O(10^{30k}r_k\log 1/\eps)$. 
\end{claim}
\begin{proof}
This follows from \cref{claim:mlha} and \cref{lem:mis_on_grids} put together by \cref{lem:parallel_composition}. Note that \cref{lem:mis_on_grids} provides a way to output MISes $X'_2, \dots, X'_k$ that, however, in general lack the property that $X'_{k+1} \subseteq X'_k$. As a simple remedy, we may inductively replace each $x'\in X'_k$, where $k\ge 2$, with the closest point $x\in X_{k-1}$, where in the step $k=2$ we use $X_1$ from the definition of the skeleton. A straightforward induction argument shows that $X_k$ is a $(r_k, 2r_k)$-ruling set. Finally, our definition of $r_k$ implies that the uniform complexity of the construction is $O(10^{20k}r_k \cdot \log (k r_k/\eps)) = O(10^{30k}r_k\log 1/\eps)$. 
\end{proof}

\paragraph{A ``Skeleton'' Construction}
We now describe a construction of a sequence $\{\fC_k\}_{k>0}$ of finer and finer subsets of the skeleton $\fC_1$.
We start with $\fC_1$ and then inductively, given the ruling set $X_{k+1}$, we define $\fC_{k+1}$ that is a subset of $\fC_k$.
In each step of the construction, we make sure that $X_{k}\subseteq \mathcal{C}_{k}$.

Suppose that $\mathcal{C}_k$ has been constructed for some $k\in \mathbb{N}$.
We have $X_{k+1}\subseteq X_k\subseteq \mathcal{C}_k$ by the properties of $\tilde{\fA}$.
Let $y,z\in X_{k+1}$ be such that $d(y,z)\le 5r_{k+1}$.
Pick a shortest path $I^{k+1}_{x,y}$ in $\mathcal{C}_k$ that connects $x$ and $y$.
Set
$$\mathcal{C}_{k+1}=\left\{I^{k+1}_{x,y}:x,y\in X_{k+1}, \ d(x,y)\le 5r_{k+1}\right\}.$$
This defines the sequence $\{\fC_k\}_{k>0}$.

\paragraph{Construction Properties}
We now prove that the construction has the following properties. 
\begin{enumerate}
    \item [(a)] $X_k\subseteq \mathcal{C}_k$, $\mathcal{C}_{k+1}\subseteq \mathcal{C}_k$
    \item [(b)] if $x\in \mathcal{C}_{k-1}\setminus \mathcal{C}_{k}$, then $d(x,\mathcal{C}_{k})\le 2r_{k}$ (where we interpret $\mathcal{C}_{0}=\mathbb{Z}^d$).
    
    \item [(c)] write $\partial_k$ for the induced graph distance on $\mathcal{C}_k$, then
    $$d(x,y)\le \partial_k(x,y)\le 5^k(10dr_1)d(x,y)$$
    whenever $x,y\in X_k$,
\item [(d)] \begin{equation}\label{eq:estimate on MIS}
    \P(\mathcal{C}_k)\le A\frac{5^{k}}{r^{d-1}_k}.
\end{equation}
for some constant $A$ that depends on $d$. Here, $\P(S)$ is a shorthand for $\P(v \in S)$ for an arbitrary node $v \in \Z^d$. 
\end{enumerate}
By construction, it is clear that (a) holds. This implies also (b) because $X_k$ is $2r_{k}$-maximal for every $k>0$ and the first equality of (c) is trivially satisfied.

\begin{claim*}
Condition (c) is satisfied.
\end{claim*}
\begin{proof}
Let $k=0$ and $x,y\in \mathcal{C}_1$.
Then we have that $\partial_0(x,y)\le (10dr_1)d(x,y)$ because $(\mathcal{C}_1,X_1,r_1)$ is a skeleton (see \cref{def:skeleton}).
Suppose that (c) holds for $k\in \mathbb{N}$.
We show that it remains valid for $k+1$.
Note that any set $I^{k+1}_{x,y}$ for any $x,y$ constructed in the $(k+1)$th iteration satisfies
\begin{equation}\label{eq:path}
\left|I^{k+1}_{x,y}\right|\le 5^{k}(10dr_1)r_{k+1}
\end{equation}
by condition (c) in the level $k$. 
In other words, if $d(x,y) \le 5r_{k+1}$, we added the set $I^{k+1}_{x,y}$ to $\fC_{k+1}$ and, hence, for this pair the condition (c) is satisfied. We now need to show this also for a (far) pair $x,y \in X_{k+1}$ with $p = d(x,y) > 5r_{k+1}$. 
Take any path $P$ from $x$ to $y$ in $\mathbb{Z}^d$ such that $|P|=p$.
We now split the path $P$ into subpaths of length $r_{k+1}$, up to the first and the last piece: we let $(p_i)_{i\le [q]}\subseteq P$ be a minimal sequence such that
\begin{itemize}
    \item $d(x,p_1)=3r_{k+1}$,
    \item  $d(p_i,p_{i+1})=r_{k+1}$ for every $i<q$,
    \item $2r_{k+1} \le d(p_q,y)\le 3r_{k+1}$.
\end{itemize}
Then we have $(q+5)r_{k+1}\le p\le (q+6)r_{k+1}$.
For every $i\in [q]$ pick $s_i\in X_{k+1}$ such that $d(p_i,s_i)\le 2r_{k+1}$.
This is possible because $X_{k+1}$ is $2r_{k+1}$-maximal.
Observe that
$$d(s_i,s_{i+1}),d(x,s_{0}),d(s_q,y)\le 5r_{k+1}$$
for every $i<q$.
Let $Q_0:=I^{k+1}_{x,s_0}$, $Q_{q}:=I^{k+1}_{s_q,y}$ and $Q_{i}:=I^{k+1}_{s_i,s_{i+1}}$ for every $i<q$.
Let $Q:={Q_0}^\frown ...^\frown Q_q$ that is, we $Q$ is a concatenation of those paths. The set $Q$ is a subset of $\fC_{k+1}$.
We have
\begin{equation*}
\begin{split}
\partial_{k+1}(x,y)\le & \ |Q|=\sum_{i\le q} |Q_i| \\
= & \ 5^{k+1}(10dr_1)r_{k+1}(q+1)\le 5^{k+1}(10dr_1)p=5^{k+1}(10dr_1)d(x,y)
\end{split}
\end{equation*}
by \eqref{eq:path}.
This finishes the proof of the claim. 
\end{proof}

Next, we show (d).
\begin{claim*}
Condition (d) is satisfied.
\end{claim*}
\begin{proof}
We start by upper bounding the size of $|X_k \cap \fB(x, 5r_k)|$ for any $x \in X_k$. Since $X_k$ is $r_k$-independent, any $x' \in X_k \cap \fB(x, 5r_k)$ defines a ball $\fB(x', r_k/2)$ disjoint from other such balls, we have $|\fB(x', r_k/2)| \ge (r_k/d)^d$. Since $|\fB(x, (5+1/2)r_k)| \le (15r_k)^d$, we conclude that $|X_k \cap \fB(x, r_k)| \le (3r_k)^d / (r_k/d)^d \le (100d)^d$.  

Thus, by \eqref{eq:path}, we have
\begin{align}
    \label{eq:mexican}
    \P(\mathcal{C}_{k})\le \P(X_k)(100d)^d 5^{k}(10dr_1)r_k
\end{align}
for every $k>0$.
Since $X_k$ is $r_k$-independent, each $x \in X_k$ defines a ball $\fB(x, r_k/2)$ of volume $|\fB(x, r_k/2)| \ge (r_k/d)^d$ disjoint from other such balls for other $x' \in X_k$ and we conclude that
$$\P(X_k)\le \frac{d^d}{r^d_k}.$$
Plugging into \cref{eq:mexican}, we conclude that there is a constant $A>0$ such that 
\begin{equation}
    \P(\mathcal{C}_k)\le \frac{A \cdot 5^{k}}{r^{d-1}_k}
\end{equation}
holds for every $k>0$.
This shows (d).
\end{proof}

\paragraph{Solving the Problem using $\{\fC_k\}_{k>0}$}
As a next step, we use the sequence $\{\fC_k\}_{k>0}$ and the combinatorial \cref{cl:PM} to show how one can  solve the problem $\pm{d}$ on almost the whole set $\fC_k \setminus \fC_{k+1}$. 
We first define a function $ \Omega : \mathbb{Z}^d\to\mathbb{Z}^d$ that moves every $x\in \mathcal{C}_k\setminus \mathcal{C}_{k+1}$ closer to $\mathcal{C}_{k+1}$ along $\mathcal{C}_k$.
The main observation is that the iterated preimage of every point under $\Omega$ is a finite connected set and by \cref{cl:PM}, every connected finite set contains at most one problematic point with respect to $\pm{d}$.
We then inductively move  this problematic point to infinity.



We first define $\Omega$ formally on $\mathbb{Z}^d \setminus \fC_1$.
For each $v \in \mathbb{Z}^d \setminus \fC_1$, let $\ell$ be its distance to $X_1$. We define $\Omega(x)$ to be any neighbor of $v$ with its distance to $X_1$ being $\ell - 1$.

Before finishing the full definition of $\Omega$, we introduce the following notation: We write $\Omega^{\rightarrow}(y)$ for the union of $\Omega^{-\ell}(y)$, where $\ell\in \mathbb{N}$, for every $y\in {\mathbb{Z}^d}$. Observe that for each $y \in X_1$, $\Omega^{\rightarrow}(y)$ is a finite set of points that are connected in $\mathbb{Z}^d$.

Now let $k>0$ and $x\in \mathcal{C}_k\setminus \mathcal{C}_{k+1}$.
Define $\Omega(x)\in \mathcal{C}_k$ such that $(x,\Omega(x))$ is an edge in $\fC_k$ and  $\partial_k(\Omega(x),\mathcal{C}_{k+1})<\partial_k(x,\mathcal{C}_{k+1})$.
This defines $\Omega$ with probability $1$ by (d).

We now observe that any point $x\in \fC_k$ cannot be too far away from $\fC_{k+1}$ in the metric induced by $\fC_k$. 
By the construction of $\mathcal{C}_k$, we find $y,y',z\in X_k$ such that $x\in I^{k}_{y,y'}$, $z\in X_{k+1}$ and $d(y,z)\le 2r_{k+1}$.
This implies, by \eqref{eq:path} and (c), that
\begin{equation}\label{eq:distance bound}
\begin{split}
\partial_{k}(x,\mathcal{C}_{k+1})\le & \ \partial_k(x,y)+\partial_k(y,z)\le 5^{k}(10dr_1)r_k+5^k(10dr_1)2r_{k+1} \\
\le & \ 5^{k+1}(10dr_1)r_{k+1}.
\end{split}
\end{equation}

Induction on $k\in \mathbb{N}$ together with \eqref{eq:distance bound} shows that $\left|\Omega^{\rightarrow}(y)\right|<\infty$.
Define $\operatorname{Pivot}_k=\Omega^{-1}(\mathcal{C}_{k+1})\cap \mathcal{C}_k$ for every $k\ge 1$.
It follows that $\Omega^{\rightarrow}(y)$ is connected for every $y\in \operatorname{Pivot}_k$ and for every $x\not \in \mathcal{C}_{k+1}$ there is a unique $p_k(x)\in \operatorname{Pivot}_k$ such that $x\in \Omega^\rightarrow(p_k(x))$.
Moreover, we have $d(x,p_k(x))\le \partial_k(x,p_k(x))\le 5^{k+1}(10dr_1)r_{k+1}$ by \eqref{eq:distance bound} whenever $x\in \mathcal{C}_{k}\setminus \mathcal{C}_{k+1}$.
Note that it follows from the assumption that $(\fC_1,X_1,r_1)$ is a skeleton that every $y\in \operatorname{Pivot}_k$ has degree $1$ in $\Omega^{\rightarrow}(y)$, in particular, $\Omega^{\rightarrow}(y)\setminus \{y\}$ is connected.
This is because $y$ must have a degree at most $2$ in $\mathbb{Z}^d$ because $\Omega(y)$ is a corner point.

We are ready to describe a solution to $\pm{d}$.
The construction proceeds by induction on $k\ge 1$.
Let $k=1$. Pick any $y\in \operatorname{Pivot}_1$.
Use Claim~\ref{cl:PM} to find a matching of $\Omega^{\rightarrow}(y)$ possibly up to $y$.
Write $M_1$ for the set of all unmatched vertices in $\operatorname{Pivot}_1$.
Note that if $y\in M_1$, then $\Omega(y)\in \mathcal{C}_2$.

Suppose that we have defined a solution up to $k\ge 1$.
This means that if we write $M_k$ for the set of unmatched vertices in the complement of $\mathcal{C}_{k+1}$, then $M_k\subseteq \operatorname{Pivot}_k$.
In particular, for every $y\in M_k$ we find $z\in \operatorname{Pivot}_{k+1}$ such that $y\in \Omega^{\rightarrow}(z)$. 
We show how to extend the construction to step $k+1$.
Let $y\in \operatorname{Pivot}_{k+1}$.
Then it follows that $\Omega^{\rightarrow}(y) \cap(\mathcal{C}_{k+1}\cup M_k)$ is connected because $\Omega(z)\in \mathcal{C}_{k+1}$ for every $z\in M_k$.
Another use of Claim~\ref{cl:PM} gives a matching of $\Omega^{\rightarrow}(y) \cap(\mathcal{C}_{k+1}\cup M_k)$ possibly up to $y$.
Note that every unmatched vertex in the complement of $\mathcal{C}_{k+2}$ is in $\operatorname{Pivot}_{k+1}$.

\paragraph{Uniform complexity of $\fA$}
It remains to argue that the solution can be constructed by a uniform local algorithm $\fA$ and compute its uniform complexity. 
Observe that if $y\not\in \mathcal{C}_{k+1}\cup \operatorname{Pivot}_k$, then $y$ can output its position in the matching whenever $y$ knows $\Omega^{\rightarrow}(p_k(y))$.
By \eqref{eq:distance bound}, $y\in \mathcal{C}_k\setminus \mathcal{C}_{k+1}$ can compute $p_{k}(y)$ if it knows $\{\mathcal{C}_i\}_{i\le k+1}$ on $\fB(y,5^{k+1}(10dr_1)r_{k+1})$.
An induction on $j\le k$ shows that $y\not \in \mathcal{C}_{k+1}$ can compute $p_{k}(y)$ if it knows $\{\mathcal{C}_i\}_{i\le k+1}$ on
$$\fB\left(y,2(10dr_1)+\sum_{j\le k}5^{j+1}(10dr_1)r_{j+1}\right)\subseteq \fB\left(y,10^3 \cdot 5^k \cdot r_{k+1}\right).$$
Moreover, we have $d(y,p_k(y))\le 10^3 \cdot 5^k \cdot r_{k+1}$ by similar inductive argument.
Consequently, $y\not\in \mathcal{C}_{k+1}$ can compute $\Omega^{\rightarrow}(p_k(y))$ if it knows $\{\mathcal{C}_i\}_{i\le k+1}$ on
\begin{equation}\label{eq:comp eq class}
\fB\left(y,10^4 \cdot 5^k \cdot r_{k+1}\right)
\end{equation}

\begin{claim}
\label{cl:5}
Let $k\in \mathbb{N}$ and $y\in \mathbb{Z}^d$.
In order to compute $\{\mathcal{C}_i\}_{i\le k}$ on \eqref{eq:comp eq class} it is enough that $y$ knows the output of $\{\fA_i\}_{i\le k}$ on 
$$
\fB\left(y,10^5 \cdot 5^k \cdot r_{k}\right).
$$
\end{claim}
\begin{proof}
Recall that $r_k=2^{ak^2}$ for every $k\in \mathbb{N}$ and put $S_k=(10dr_1)\sum_{i\le k}5^ir_i$.
Fix $N\in \mathbb{N}$.
We are now going to prove inductively that to output $\{\mathcal{C}_i\}_{i\le k+1}$ on $\fB(y,N)$ it is enough to know $\{\fA_i\}_{i\le k+1}$ on
$$\fB\left(y,N+S_{k+1}\right).$$
Setting $N=10^4 \cdot 5^k \cdot r_{k+1}$ finishes the proof since $S_{k+1}\le 10^4 \cdot 5^k \cdot r_{k+1}$.

Let $k=0$ and $N\in \mathbb{N}$.
We have $\mathcal{C}_1=\fA_1$ and $|I^1_{y_1,y_2}|\le 5(10dr_1)r_1$ by \eqref{eq:path} whenever $y_1, y_2\in Y_1$.
This implies that to compute $\mathcal{C}_1$ on $\fB(y,N)$ it is enough to know $\fA_1$ on $\fB(y,N+5(10dr_1)r_1)\subseteq \fB(y,N+S_1)$.

Suppose that the claim holds for $k\in \mathbb{N}$, we show how to extend it to $k+1$.
It is clear that any line segment of the form $I^{k+2}_{y_1,y_2}$ for $y_1, y_2\in Y_{k+2}$ that intersects $\fB(y,N)$ is contained in
\begin{equation}\label{eq:radius}
\fB(y,N+5^{k+2}(10dr_1)r_{k+2})
\end{equation}
by \eqref{eq:path}.
Therefore to output $\{\mathcal{C}_i\}_{i\le k+2}$ on $\fB(y,N)$, it is enough to output $\{\mathcal{C}_i\}_{i\le k+1}$ and $X_{k+2}$ on \eqref{eq:radius}.
By the inductive assumption, we have that it is enough to output $\{\fA_i\}_{i\le k+2}$ on 
$$\fB(y,N+5^{k+2}(10dr_1)r_{k+2}+S_{k+1}).$$
It remains to observe that $S_{k+2}=5^{k+2}(10dr_1)r_{k+2}+S_{k+1}$.
\end{proof}

As an immediate consequence we see that $\fA$ is a correct uniform local algorithm, i.e., given $y$ we find minimal $k\in \mathbb{N}$ such that $y\not \in \mathcal{C}_{k+1}\cup \operatorname{Pivot}_k$ compute $\{\fA_i\}_{i\le k+1}$ on $\fB\left(y,10^5 5^{k+1} r_{k+1}\right)$, obtain $\Omega^{\rightarrow}(p_k(y))$ and get the position in the matching.
This can be done with probability $1$ because
$$\P\left(\mathcal{C}_{k+1}\cup \operatorname{Pivot}_k\right)\le \P\left(\mathcal{C}_{k+1}\cup \Omega^{-1}(\mathcal{C}_{k+1})\right)\le 2d A\frac{5^{k+2}}{r^{d-1}_{k+1}}\to 0$$
by \eqref{eq:estimate on MIS}.

We can now compute the uniform complexity of $\fA$. 
Note that there is large enough $C$ such that for any $\eps > 0$, if $y$ does not know the answer after $2^{Ck} r_k \log (1/\eps)$ steps, either it is in $\fC_{k+1} \cup \operatorname{Pivot}_k$, which happens with probability $O(5^k / r_k^{d-1})$ by \cref{cl:5}, or the algorithm cannot construct $\{\fC_i\}_{i \le k}$ which by \cref{cl:we_are_crazy} happens with probability at most $\eps$. Choosing $\eps = O(5^k / r_k^{d-1})$, we get that there are some $C_1, C_2 > 0$ such that
\[
\P( R_\fA > 2^{C_1k} r_k ) \le 2^{C_2k} / r_k^{d-1}. 
\]

To finish, consider any $r$ and let $k$ be such that $2^{C_1k} r_k \le r < 2^{C_1(k+1)} r_{k+1} $. On one hand, we just computed that 
\[
\P( R_\fA > r) \le 2^{C_2k} / r_k^{d-1} = \frac{1}{2^{(d-1)ak^2 - O(k)}}. 
\]
On the other hand, we have $r < 2^{C_1(k+1)} 2^{a(k+1)^2} = 2^{ak^2 + O(k)}$. Combining the two bounds, we get
\[
\P(R_\fA > r) = \frac{2^{O(k)}}{r^{d-1}}
\]
Finally, since $r \ge 2^{ak^2}$, we have $k = O(\sqrt{\log r})$, thus we have
\[
\P(R_\fA > r) = \frac{2^{O(\sqrt{\log r})}}{r^{d-1}}
\]
or equivalently for arbitrary $\eps$ we have
\[
\P(R_\fA > \sqrt[d-1]{1/\eps} \cdot 2^{O(\sqrt{\log 1/\eps})}) < \eps
\]
as needed. 

\end{proof}

\end{document}